\documentclass[10pt,article,reqno]{amsart}
\usepackage{amssymb}
\usepackage{enumitem}
\usepackage{amsfonts}
\usepackage{eucal}
\usepackage{txfonts}
\usepackage{amsmath}
\usepackage{comment}
\usepackage{hyperref}
\usepackage{extarrows}
\usepackage{mathrsfs}
\usepackage{bm}

\newtheorem{thm}{Theorem}[section]

\newtheorem{lem}[thm]{Lemma}

\newtheorem{exam}[thm]{Example}
\theoremstyle{definition}
\newtheorem{defn}[thm]{Definition}
\theoremstyle{remark}
\newtheorem{rem}[thm]{Remark}
\numberwithin{equation}{section}

\makeatletter
\@namedef{subjclassname@2020}{\textup{2020} Mathematics Subject Classification}
\makeatother

\begin{document}
	\title[]
	{On the rigidity of manifolds with respect to the Gagliardo-Nirenberg inequalities
	}
	
	\author{Liang Cheng}

	
	\subjclass[2020]{Primary 53C24; Secondary 	53E20 .}

	\keywords{  scalar curvature; rigidity theorems; Gagliardo-Nirenberg inequalities,  	unweighted Yamabe-type constants}
	
	\thanks{Liang Cheng's  Research partially supported by
		Natural Science Foundation of China 12171180
	}
	
	\address{School of Mathematics and Statistics, and Key Laboratory of Nonlinear Analysis $\&$ Applications (Ministry of Education), Central  China Normal University, Wuhan, 430079, P.R.China}
	
	\email{chengliang@ccnu.edu.cn }

	\begin{abstract}
		In this paper, we investigate local rigidity properties related to Gagliardo-Nirenberg constants and unweighted Yamabe-type constants. Let $V$ be an open bounded subset of an $n$-dimensional Riemannian manifold $(M,g)$ whose Gagliardo-Nirenberg constant satisfies
		\[
		\mathbb{G}_{\alpha}^{\pm}(V,g) \geq \mathbb{G}_{\alpha}^{\pm}(\mathbb{R}^n,g_{\mathbb{R}^n}),
		\]
		where $(\mathbb{R}^n,g_{\mathbb{R}^n})$ denotes the $n$-dimensional Euclidean space with its standard metric.
		We show that
		for $\alpha \in (0,1) \cup \left(1,\frac{n+6}{n+2}\right)$ when $n \leq 6$ or $\alpha \in (0,1) \cup \left(1,\frac{n}{n-2}\right]$ when $n \geq 7$, if the first eigenvalue of the Ricci tensor satisfies
		\[
		\int_V \lambda_1(\operatorname{Rc}) \, d\mu_g \geq 0,
		\]
		then $V$ must be flat.
		When $\alpha$ belongs to a specific subinterval around $1$ within the above range, 
		$\mathbb{G}_{\alpha}^{\pm}(V,g) \geq \mathbb{G}_{\alpha}^{\pm}(\mathbb{R}^n,g_{\mathbb{R}^n})$ and the weaker curvature condition of the scalar curvature
		\[
		\int_{V} \operatorname{Sc} \, d\mu_g \geq 0
		\]
		already imply that $V$ is flat.
		Moreover, we prove that for $\alpha$ sufficiently close to 1, the condition 
		\[
		\mathbb{Y}_{\alpha}^{\pm}(V,g) \geq \mathbb{G}_{\alpha}^{\pm}(\mathbb{R}^n,g_{\mathbb{R}^n})
		\] 
		on the unweighted Yamabe-type constants guarantees the flatness of $V$.
	\end{abstract}
	\maketitle	
	
	\section{Introduction and Main results}
	
	\subsection{Introduction}
	It has been established that if a complete noncompact Riemannian manifold $M$ has nonnegative Ricci curvature and the optimal constants of many Sobolev-type inequalities are not less than those of Euclidean space, then $M$ must be flat. These Sobolev-type inequalities include the Sobolev inequality \cite{LSobolev}, the log-Sobolev inequality \cite{BCL, Ni2, Ni3, BKT}, the Gagliardo-Nirenberg inequality \cite{X2005, BKT}, the Caffarelli-Kohn-Nirenberg inequality \cite{DX, K2}, etc.
	The following natural problems arise:
	\begin{enumerate}
		\item Do these rigidity results hold locally? That is, if we only assume that these optimal constants of some open subset $V \subset M$  are not less than those of Euclidean space in the aforementioned rigidity results, does this imply flatness for $V$?

		\item Can the nonnegative Ricci curvature condition in the aforementioned rigidity results be weakened to nonnegative scalar curvature (or something weaker)?
	\end{enumerate}

	In this paper, we first investigate these problems concerning the Gagliardo-Nirenberg constants. To present our results, we begin by recalling the definition of the Gagliardo-Nirenberg constants.
	
	\begin{defn}
		Let $(M^n,g)$ be an $n$-dimensional Riemannian manifold, and $V \subset M^n$. Denote $\|u\|_{L^p(V)} := \left(\int_V |u|^p \, \mathrm{d}\mu_g\right)^{\frac{1}{p}}$, where $\mathrm{d}\mu_g$ is the volume form of $V$ induced by $g$. The Gagliardo-Nirenberg constants of $V$ with respect to $g$ are defined as:
		$$
		\mathbb{G}_{\alpha}^-(V,g) := \inf_{u \in W_0^{1,2}(V)} \frac{\|\nabla u\|^2_{L^2(V)} \|u\|_{L^{2\alpha}(V)}^{\frac{2(1-\gamma)}{\gamma}}}{\|u\|_{L^{\alpha+1}(V)}^{\frac{2}{\gamma}}}, \quad \text{for } 0 < \alpha < 1,
		$$
		and
		$$
		\mathbb{G}_{\alpha}^+(V,g) := \inf_{u \in W_0^{1,2}(V)} \frac{\|\nabla u\|^2_{L^2(V)} \|u\|_{L^{\alpha+1}(V)}^{\frac{2(1-\theta)}{\theta}}}{\|u\|_{L^{2\alpha}(V)}^{\frac{2}{\theta}}}, \quad \text{for } 1 < \alpha \leq \frac{n}{n-2},
		$$
		where  
		\begin{equation}\label{def:gamma}
			\gamma:=\frac{2^*(1-\alpha)}{\left(2^*-2\alpha \right)(\alpha +1)}, \text{i.e. $\gamma$ is determined by }\frac{1}{\alpha+1} = \frac{\gamma}{2^*} + \frac{1-\gamma}{2\alpha},
		\end{equation}
		here $2^* := \frac{2n}{n-2}$, and 
		\begin{equation}\label{def:theta}
			\theta:=\frac{2^*(\alpha-1)}{2\alpha \left(2^*-\alpha -1\right)}, \text{i.e. $\theta$ is determined by  }	\frac{1}{2\alpha} = \frac{\theta}{2^*} + \frac{1-\theta}{\alpha+1}.
		\end{equation}
	\end{defn}
	\begin{rem}
		When $V=\mathbb{R}^n$ is $n$-dimensional Euclidean space, Del Pino and Dolbeault \cite{DD} proved 
		\begin{equation}\label{def_optimal_const}
			\begin{aligned}
				\mathbb{G}_{\alpha}^-(\mathbb{R}^n,g_{\mathbb{R}^n})&=\mathcal{N}_{\alpha, n}:=\left(\frac{1-\alpha}{2}\right)^\gamma \frac{\left(\frac{2}{n}\right)^{\frac{\gamma}{2}+\frac{\gamma}{n}}\left(\frac{1+\alpha}{1-\alpha}+\frac{n}{2}\right)^{\frac{\gamma}{2}-\frac{1}{\alpha+1}}\left(\frac{1+\alpha}{1-\alpha}\right)^{\frac{1}{\alpha+1}}}{\left(\omega_n \mathrm{~B}\left(\frac{1+\alpha}{1-\alpha}, \frac{n}{2}\right)\right)^{\frac{\gamma}{n}}},\\
				\mathbb{G}_{\alpha}^+(\mathbb{R}^n,g_{\mathbb{R}^n})&=\mathcal{G}_{\alpha, n}:=\left(\frac{\alpha-1}{2}\right)^\theta \frac{\left(\frac{2}{n}\right)^{\frac{\theta}{2}+\frac{\theta}{n}}\left(\frac{\alpha+1}{\alpha-1}-\frac{n}{2}\right)^{\frac{1}{2\alpha }}\left(\frac{\alpha+1}{\alpha-1}\right)^{\frac{\theta}{2}-\frac{1}{2\alpha }}}{\left(\omega_n \mathrm{~B}\left(\frac{\alpha+1}{\alpha-1}-\frac{n}{2}, \frac{n}{2}\right)\right)^{\frac{\theta}{n}}},
			\end{aligned}
		\end{equation}
		where $\mathrm{B}(\cdot, \cdot)$ is the Euler beta-function.
		Moreover, they proved
		$\mathbb{G}_{\alpha}^-(\mathbb{R}^n,g_{\mathbb{R}^n})$ and  $\mathbb{G}_{\alpha}^+(\mathbb{R}^n,g_{\mathbb{R}^n})$ both can be achieved by the family of functions 
		$$
		h_{\alpha}^\lambda(x)=\left(\lambda+(\alpha-1)\|x\|^{2}\right)_{+}^{\frac{1}{1-\alpha}}, \quad x \in \mathbb{R}^n,
		$$
		where  $r_{+}=\max \{0, r\}$ for $r \in \mathbb{R}$. 
	\end{rem}
	\begin{rem}
		The borderline case $\alpha=\frac{n}{n-2}$ (thus $\theta=1$) of $\mathbb{G}_{\alpha}^+(V,g)$ reduces to the optimal constant of the Sobolev inequality of $V$. Furthermore,   whenever $\alpha \rightarrow 1$, both $\mathbb{G}_{\alpha}^+(V,g)$ and $\mathbb{G}_{\alpha}^-(V,g)$ degenerate to the optimal constant of logarithmic Sobolev inequality of $V$: $ \inf\limits_{\substack{u \in W_0^{1,2}(V),\int_{V} u^2 d\mu = 1}}\frac{\frac{n}{2} \log\left(\frac{2}{n\pi e} \int_{V}|\nabla u|^2\,d\mu_g\right)}{\int_{V} u^2 \log u^2\,d\mu_g}$.  When $\alpha \rightarrow 0$,
		$\mathbb{G}_{\alpha}^-(V,g)$ reduces to the optimal constant of the Faber-Krahn type inequality of $V$: 
		$
		\inf\limits_{\substack{u \in W_0^{1,2}(V)}} \frac{\|\nabla u\|_{L^2(V)}|\operatorname{supp}(u)|^{\frac{1}{2}}}{\|u\|_{L^1}}.
		$
	\end{rem}
	
	The Gagliardo-Nirenberg constants are related to
	the  weighted
	Yamabe constants considered by
	Jeffrey S. Case \cite{Jcase,Jcase2}, which constitute a one-parameter family and interpolate between the
	Yamabe constant and Perelman’s $\nu$-entropy when the parameter $m$ is zero and inﬁnity,
	respectively. 
	Let $(M^n, g, e^{-\phi} \text{dvol}, m)$ be the  quadruples consisting of a $n$-dimensional Riemannian manifold $(M^n, g)$, a weighted volume measure $e^{-\phi} \text{dvol}$, and a dimensional parameter $m \in [0, \infty]$, where the fundamental geometric quantity is the weighted scalar curvature $R_\phi^m = \operatorname{Sc}(g) + 2\Delta \phi - \frac{m+1}{m} |\nabla \phi|^2$. The weighted Yamabe constants, which serve as curved analogues for the Gagliardo-Nirenberg constants, are defined as the infimum of the quotient
	\begin{equation}\label{case_def}
		\mathcal{Q}(w) = \frac{\left(\int_M |\nabla w|^2 + \frac{m+n-2}{4(m+n-1)} R_\phi^m w^2\right) \left(\int_M |w|^{\frac{2(m+n-1)}{m+n-2}} e^{\phi/m}\right)^{\frac{2m}{n}}}{\left(\int_M |w|^{\frac{2(m+n)}{m+n-2}}\right)^{\frac{2m+n-2}{n}}},
	\end{equation}
	where all integrals are computed with respect to $e^{-\phi}\text{dvol}$.
	The weighted Yamabe quotient is conformally invariant in the sense that if
	$
	\left(M^n, \hat{g}, e^{-\hat{\phi}} \operatorname{dvol}_{\hat{g}}, m\right) = \left(M^n, e^{\frac{2 \sigma}{m+n-2}} g, e^{\frac{(m+n) \sigma}{m+n-2}} e^{-\phi} \operatorname{dvol}_g, m\right),
	$
	then $\hat{Q}(w)=Q(we^{\frac{\sigma}{2}})$.

	For the second purpose of this paper, we study the rigidity properties associated with the unweighted case (i.e., $\phi \equiv \mathrm{constant}$ in \eqref{case_def}) for weighted Yamabe constants. Specifically, we use the
	following definitions:

	\begin{defn}
		Let $(M^n,g)$ be an $n$-dimensional manifold  and $V\subset M^n$. 	Denote $\|u\|_{L^p(V)} := \left(\int_V |u|^p \, \mathrm{d}\mu_g\right)^{\frac{1}{p}}$, where $\mathrm{d}\mu_g$ is the volume form of $V$ induced by $g$. 
		The unweighted Yamabe-type constants of $V$ with respect to metric $g$ are defined as:	
		$$
		\mathbb{Y}_{\alpha}^-(V,g):=\inf_{u \in  W_0^{1,2}(V)}\frac{\left(\|\nabla u\|_{L^2(V)}^2+\frac{1}{2(1+\alpha)}  \int_V \operatorname{Sc}\cdot u^2d\mu_g	\right)\cdot \|u\|_{L^{2\alpha}(V)}^{\frac{2(1-\gamma)}{\gamma}}}{\|u\|_{L^{\alpha+1}(V)}^{\frac{2}{\gamma}}  } \text{ for $0<\alpha<1$,}
		$$		
		and
		$$
		\mathbb{Y}_{\alpha}^+(V,g):=\inf_{u \in  W_0^{1,2}(V)}\frac{\left(\|\nabla u\|_{L^2(V)}^2+\frac{1}{2(1+\alpha)}  \int_V \operatorname{Sc}\cdot u^2d\mu_g	\right)\cdot \|u\|_{L^{\alpha+1}(V)}^{\frac{2(1-\theta)}{\theta}}}{\|u\|^{\frac{2}{\theta}}_{L^{2\alpha}(V)}  } \text{ for  $1<\alpha\le \frac{n}{n-2}$,}
		$$
		where $\operatorname{Sc}$ denotes the scalar curvature with respect to $g$.
	\end{defn}
	
	\begin{rem}
		Take $\alpha=\frac{m+n}{m+n-2}$ and when $\phi\equiv constant$, the quantity $\mathbb{Y}_{\alpha}^+(V,g)$ coincides with the infimum of the quotient in \eqref{case_def}.    
	\end{rem}	
	
	\begin{rem} 
		The borderline case $\alpha=\frac{n}{n-2}$ (thus $\theta=1$ ) of $\mathbb{Y}_{\alpha}^+(V,g)$ reduces to the Yamabe constant of $V$.
	\end{rem}
	
	\subsection{Statement of main results}

	In \cite{X2005}, Xia proved that for a complete noncompact Riemannian manifold $M$ with nonnegative Ricci curvature, if $\mathbb{G}_{\alpha}^+(M, g) \geq \mathbb{G}_{\alpha}^+(\mathbb{R}^n, g_{\mathbb{R}^n})$, then $M$ must be flat. 
	Krist\'{a}ly \cite{BKT} demonstrated that this rigidity property extends to the case where $\mathbb{G}_{\alpha}^-(V, g) \geq \mathbb{G}_{\alpha}^-(\mathbb{R}^n, g_{\mathbb{R}^n})$ under the same assumptions. Moreover, Krist\'{a}ly \cite{BKT} established quantitative volume properties related to the Gagliardo-Nirenberg constant for metric measure spaces satisfying the curvature-dimension condition $CD(K, n)$ with $K \geq 0$.
	Our first main theorem establishes the following local rigidity results concerning the Gagliardo-Nirenberg constants: For an open subset $V \subset M$ that satisfies
	\begin{equation}\label{in}
		\mathbb{G}_{\alpha}^{\pm}(V,g)\geq \mathbb{G}_{\alpha}^{\pm}(\mathbb{R}^n, g_{\mathbb{R}^n}),
	\end{equation}
	we show that when $\alpha \in (0,1) \cup \left(1,\frac{n+6}{n+2}\right)$ for $n \leq 6$ and $\alpha$ is unrestricted for $n \geq 7$ (i.e.  $\alpha \in (0,1) \cup \left(1,\frac{n}{n-2}\right]$ when $n \geq 7$), if the first eigenvalue of the Ricci tensor satisfies
	\[
	\int_V \lambda_1(\operatorname{Rc}) \, d\mu_g \geq 0,
	\]
	then $V$ must be flat. More interestingly, 
	when $\alpha$ belongs to a specific subinterval around $1$ within the above range  (which will be explicitly specified in the following theorem), \eqref{in} and the weaker curvature condition on the scalar curvature
	\[
	\int_{V} \operatorname{Sc} \, d\mu_g \geq 0
	\]
	already imply that $V$ is flat.

	\begin{thm}\label{const_rigidity}
		Let $(M^n,g)$ be a Riemannian manifold of dimension $n\geq 3$, and let $V$ be an open bounded subset of $M^n$.	
		
		\noindent\textbf{(a) For the case $0<\alpha<1$:}
		\begin{enumerate}
			\item For $0<\alpha<1$, if 
			\begin{equation}\label{const_comparison_iso_2_6}
				\mathbb{G}_{\alpha}^-(V,g) \geq \mathbb{G}_{\alpha}^-(\mathbb{R}^n,g_{\mathbb{R}^n}),
			\end{equation}
			then the scalar curvature satisfies 
			$$\text{$\operatorname{Sc}(x) \leq 0$ for all $x \in V$.}$$
			
			\item  For $0<\alpha<1$,  
			if both \eqref{const_comparison_iso_2_6} and
			\begin{equation}\label{thm_Rc_point_compare_6}
				\int_{V} \lambda_1(\operatorname{Rc}) \, d\mu \geq 0 
			\end{equation}
			hold, \textbf{then} $V$ must be flat. Here, $\lambda_1(\operatorname{Rc})$ denotes the first eigenvalue of the Ricci tensor.
			
			\item Under the more restrictive range:
			\begin{equation}\label{6_i_3}
				\begin{cases}
					0<\alpha<1, & \text{if } n=3, \\
					\frac{2 n^2+n-25-\sqrt{28 n^2-16 n-287}}{2 n^2+3 n-38}<\alpha<1, & \text{if } n\geq 4,
				\end{cases}
			\end{equation}
			if both \eqref{const_comparison_iso_2_6} and 
			\begin{equation}\label{thm_R_point_compare_6}
				\int_{V} \operatorname{Sc} \, d\mu \geq 0 
			\end{equation}
			hold, \textbf{then} $V$ must be flat.
		\end{enumerate} 
		
		\noindent\textbf{(b) For the case $1<\alpha\le \frac{n}{n-2}$:}
		\begin{enumerate}
			\item For $\alpha$ in the range:
			\begin{equation}\label{6_i_2}
				\begin{cases}
					1< \alpha<\frac{n+4}{n}, & \text{if } n\leq 4, \\
					1<\alpha\leq \frac{n}{n-2}, & \text{if } n\geq 5,
				\end{cases}
			\end{equation}
			if 
			\begin{equation}\label{eq:comparison_iso_2_6}
				\mathbb{G}_{\alpha}^+(V, g) \geq \mathbb{G}_{\alpha}^+(\mathbb{R}^n, g_{\mathbb{R}^n}),
			\end{equation}
			then the scalar curvature satisfies 
			$$\text{$\operatorname{Sc}(x) \leq 0$ for all $x \in V$.}$$
			
			\item Under the more restrictive range:
			\begin{equation}\label{6_i_4666}
				\begin{cases}
					1<\alpha<\frac{n+6}{n+2}, & \text{if } n\leq 6, \\
					1<\alpha\le\frac{n}{n-2}, & \text{if } n\geq7,
				\end{cases}
			\end{equation}
			if both \eqref{eq:comparison_iso_2_6} and
			\begin{equation}\label{thm_Rc_point_compare_666}
				\int_{V} \lambda_1(\operatorname{Rc}) \, d\mu \geq 0 
			\end{equation}
			hold, \textbf{then} $V$ must be flat.
			
			\item Under the more restrictive range:
			\begin{equation}\label{6_i_4}
				\begin{cases}
					1<\alpha<\frac{n+6}{n+2}, & \text{if } n\leq 6, \\
					1<\alpha<\frac{2 n^2+n-25+\sqrt{28 n^2-16 n-287}}{2 n^2+3 n-38}, & \text{if } n\geq7,
				\end{cases}
			\end{equation}
			if both \eqref{eq:comparison_iso_2_6} and 
			\begin{equation}\label{Sc_int_nonnegative}
				\int_{V} \operatorname{Sc} \, d\mu \geq 0
			\end{equation}
			hold, \textbf{then} $V$ must be flat.
		\end{enumerate}
	\end{thm}
	
	\begin{rem}
		The rigidity properties of Riemannian manifolds with lower scalar curvature bounds are an important subject of intensive study. 
		As highlighted by M. Gromov in the following problems: \textit{ Find verifiable criteria for extremality and rigidity, decide which manifolds admit extremal/rigid metrics and describe particular extremal/rigid manifolds; see Problem C in \cite{Gdozen}.}
		For a comprehensive overview of the rigidity results concerning scalar curvature, we refer to the survey \cite{Brendle} and the lectures \cite{Gromovlecture}, along with additional references therein.	Note that \textbf{(a)}(3) and \textbf{(b)}(3) in Theorem \ref{const_rigidity} demonstrate the rigidity of the scalar curvature with respect to the scalar curvature and the Gagliardo-Nirenberg constants.
	\end{rem}
	
	For the second main theorem of this paper, we show that when $\alpha$ is sufficiently close to $1$, if an open subset $V \subset M$ has unweighted Yamabe-type constants no less than $\mathbb{G}_{\alpha}^{\pm}(\mathbb{R}^n,g_{\mathbb{R}^n})$, then $V$ must be flat.

	\begin{thm}\label{Yamabe_rigidity}
		Let $(M^n,g)$ be a Riemannian manifold of dimension $n\geq 3$, and let $V$ be an open bounded subset of $M^n$. 	
		Suppose that there exist constants $\kappa^{\pm}$, depending only on $n$, such that
		when $0<|\alpha-1|<\kappa^{\pm}$, the following holds:
		\begin{equation}\label{Yamabe_comparison_iso_2_6}
			\mathbb{Y}_{\alpha}^{\pm}(V,g)\ge \mathbb{G}_{\alpha}^{\pm}(\mathbb{R}^n,g_{\mathbb{R}^n}),
		\end{equation}
		Then $V$ must be flat.
	\end{thm}
	\begin{rem}
		The exact values of $\kappa^{\pm}$ are determined by solutions of the seventh-degree polynomial equations (see Remark~\ref{exact_number_equation}), which are generally difficult to express in closed numerical form.
	\end{rem}

	It remains unclear for us whether the ranges of $\alpha$ in Theorem~\ref{const_rigidity} (specified in \eqref{6_i_3}, \eqref{6_i_2}, \eqref{6_i_4666} and \eqref{6_i_4}) and the range \eqref{Yamabe_comparison_iso_2_6} in Theorem \ref{Yamabe_rigidity} are sharp for these results. However, the following example demonstrates that
	\begin{itemize}
		\item The conclusion of \textbf{(b)}(3) in Theorem \ref{const_rigidity} fails for the borderline case $\alpha=\frac{n}{n-2}$ , where $\mathbb{G}_{\alpha}^+(V,g)$ reduces to the Sobolev constant;
		\item The conclusion of Theorem~\ref{Yamabe_rigidity} also fails for the borderline case $\alpha=\frac{n}{n-2}$, where $\mathbb{Y}_{\alpha}^+(V,g)$ reduces to the Yamabe constant.
	\end{itemize}
	Here, we remark that, for comparison, the rigidity in \textbf{(b)}(2) in Theorem \ref{const_rigidity} holds for the borderline case $\alpha=\frac{n}{n-2}$ when $n \ge 7$. 
	This shows that when $n \ge 7$, the range of $\alpha$ for which we can guarantee that \eqref{eq:comparison_iso_2_6} and \eqref{Sc_int_nonnegative} imply the flatness of $V$ is strictly smaller than the range of $\alpha$ for which we can guarantee that \eqref{eq:comparison_iso_2_6} and \eqref{thm_Rc_point_compare_666} imply  the flatness of $V$.

	\begin{exam}[Schwarzschild metric]
		Note that the scalar curvature and Yamabe quotient of $g=u^{\frac{4}{n-2}} g_0$ on $M$ are related to the scalar curvature of $g_0$ by
		\begin{equation}\label{conformal_change_R}
			\operatorname{Sc}(g)= u^{-\frac{n+2}{n-2}} \left(\operatorname{Sc}\left(g_0\right) u-c(n)\Delta_{g_0} u\right),
		\end{equation}
		where $c(n)=\frac{4(n-1)}{n-2}$, and
		\begin{equation}\label{conformal_of_Q}
			Q_{g}(\phi)=Q_{g_0}(\phi u) \text{ for any smooth function } \phi,
		\end{equation}
		where the Yamabe quotient defined as 
		$$ Q_{g}(\phi):=\frac{\int_M \left( c(n) |\nabla_g \phi|^2+\operatorname{Sc}(g)\cdot \phi^2\right)d\mu_g}{\left(\int_M \phi^{\frac{2n}{n-2}}d\mu_g\right)^{\frac{n-2}{n}}} $$
		\text{ (cf. \cite{STbook})}.
		For a manifold $x\in\mathbb{R}^n\setminus\{0\}$ with $n \geq 3$, let $g_0:=\delta_{ij}$, $u(x):=1+ \frac{m}{2\left|x\right|^{n-2}}$, and $g=u^{\frac{4}{n-2}} g_0$, where $m$ is a positive real number. The hypersurface defined by $|x|=\frac{m}{2}$ is a totally geodesic submanifold, called the horizon. Reflecting the region $|x|>\frac{m}{2}$ across this horizon yields a complete smooth Riemannian manifold $N$, which is diffeomorphic to $S^{n-1} \times(0,1)$. Moreover, by \eqref{conformal_change_R} and $\Delta_{g_0}\frac{1}{\left| x\right|^{n-2}}=0$, we have $\operatorname{Sc}(g)\equiv 0$. 
		By \eqref{conformal_of_Q}, we have 
		$$
		\frac{\int_N |\nabla_g \phi|^2d\mu_g}{\left(\int_N \phi^{\frac{2n}{n-2}}d\mu_g\right)^{\frac{n-2}{n}}}=\frac{\int_N |\nabla_{g_0} (\phi u)|^2d\mu_{g_0}}{\left(\int_N (\phi u)^{\frac{2n}{n-2}}d\mu_{g_0}\right)^{\frac{n-2}{n}}} \text{ for any smooth function } \phi.
		$$
		This implies that the optimal constant of the Sobolev inequality for $(N,g)$ coincides with the Euclidean best constant. However, although $\operatorname{Sc}(g)\equiv 0$, the manifold $(N,g)$ is not Ricci flat. Note that $\mathbb{Y}_{\alpha}^+(N,g)$ reduces to $\mathbb{G}_{\alpha}^+(N,g)$ when $\operatorname{Sc}(g)\equiv 0$. So, this example also shows that the conclusion of Theorem~\ref{Yamabe_rigidity} fails when $\alpha=\frac{n}{n-2}$.
	\end{exam}

	\subsection{Strategy of our proofs}
	We outline the strategy of our proofs that employs the power series expansion method. In our previous work \cite{Cheng}, it was shown that the logarithmic Sobolev functional $\mathcal{L}(V,g,u,t)$ and Perelman's $\mathcal{W}$-functional $\mathcal{W}(V,g,u,t)$ (defined in \eqref{Perel_W} and \eqref{log_func}) admit the following power series expansions:
	\begin{equation}\label{expansion_nu_opt}
		\mathcal{W}(V,g,u,t)=
		-\frac{1}{6}|\operatorname{Rm}|^2(p)t^2+o(t^2),
	\end{equation}		
	and
	\begin{equation}\label{expansion_n_opt}
		\mathcal{L}(V,g,u,t)=
		-\operatorname{Sc}(p)t-\left(\Delta \operatorname{Sc}(p)+\frac{1}{6}|\operatorname{Rm}|^2(p)\right)t^2+o(t^2),
	\end{equation}   
	when $u$ is chosen as $ u = (4\pi t)^{-\frac{n}{4}}e^{-\frac{d(p,x)^2}{8t}}\eta$ (note that the function $ (4\pi t)^{-\frac{n}{4}}e^{-\frac{d(p,x)^2}{8t}}$ achieves the optimal constant in the logarithmic Sobolev inequality of Euclidean space), where $\eta$ satisfies:
	\begin{enumerate}
		\item $p\in\operatorname{supp}(\eta) \subset\subset V$,
		\item $\eta(x,t)^2$ admits the local expansion 
		$$
		\eta(x,t)^2 = \sum_{k=0}^2 \phi_k(x)t^k + o(t^2) \quad \text{around } (p,0),
		$$
		with the following regularity conditions at $p$:
		Both fourth derivatives of $\phi_0$ and second derivatives of $\phi_1$ exist at $p$,
		$\phi_2$ is continuous at $p$. Moreover, $\phi_0(p)=1$ and $\nabla\phi_0(p)=0.$
		
		\item $\phi_0$  admits the local expansion
		$$
		\phi_0(x) = 1+\frac{1}{3}\operatorname{Rc}(p)_{ij}x^ix^j + o(d^2) \quad \text{around } p,
		$$
		where $\{x^i\}$ denotes the normal coordinates on $T_pM$.
	\end{enumerate}
	Using the expansions \eqref{expansion_nu_opt} and \eqref{expansion_n_opt}, 
	the author \cite{Cheng} proved that if an open subset $V \subset M$ satisfies
	$\mathop{\int}_{V} \operatorname{Sc} \, d\mu \geq 0$ and the optimal constant of the log-Sobolev inequality (the limit case $\alpha \to 1$ of the Gagliardo-Nirenberg constants) 
	is not less than that of the Euclidean space, then $V$ must be flat.

	Motivated by these results, we first derive  the power series expansions of 
	$\mathcal{L}^{\pm}_{\alpha}$-functionals $\mathcal{L}^{\pm}_{\alpha}(V,g, u, \tau)$ and $\mathcal{W}^{\pm}_{\alpha}$-functionals $\mathcal{W}^{\pm}_{\alpha}(V,g, u, \tau)$ (see Definitions \ref{Def_sec2} and Theorem \ref{connection}) - which generalize the logarithmic Sobolev and Perelman's $\mathcal{W}$-functionals. These expansions (see Theorem \ref{expansion_L_a} and Lemma \ref{expansion_W_a}) are obtained when  $u$ takes the following forms:
	\begin{itemize}
		\item For $\mathcal{L}^{-}_{\alpha}$ and $\mathcal{W}^{-}_{\alpha}$-functionals (with normalization $\int_V u^{\alpha+1} d\mu_g = 1$):
		\[
		u = C_1t^{-\frac{n}{2(\alpha+1)}}\left(1 + \frac{(\alpha-1)d(p,x)^2}{8t}\right)_+^{\frac{1}{1-\alpha}}\xi_-;
		\]
		
		\item For $\mathcal{L}^{+}_{\alpha}$ and $\mathcal{W}^{+}_{\alpha}$-functionals (with normalization $\int_V u^{2\alpha} d\mu_g = 1$):
		\[
		u = C_2t^{-\frac{n}{4\alpha}}\left(1 + \frac{(\alpha-1)d(p,x)^2}{8t}\right)_+^{\frac{1}{1-\alpha}}\xi_+,
		\]
	\end{itemize}
	(note that $C(t)\left(1 + \frac{(\alpha-1)d(p,x)^2}{8t}\right)_+^{\frac{1}{1-\alpha}}$ attains the optimal constant $\mathbb{G}_{\alpha}^{\pm}(\mathbb{R}^n,g_{\mathbb{R}^n})$), where the cutoff functions $\xi_{\pm}$ satisfy conditions (1) and (2).
	To prove our rigidity results, we apply these expansion formulas with the cutoff functions $\xi_{\pm}$ chosen according to the precise criteria given in Remark \ref{Rk_choose_of_a} and implemented in the proofs of Theorem \ref{mu_rigidity} and Theorem \ref{mu_constant_rigidity}.
	
	Finally, we can also use the expansion formulas obtained in this paper to prove the following rigidity theorem concerning the scalar curvature and isoperimetric inequality, which was proved by the author in \cite{Cheng} (see Theorem 1.1 and Remark 1.2 in \cite{Cheng}):
	\begin{thm}\cite{Cheng}\label{rigidity_iso_profile}
		Let $(M^n,g)$ be an $n$-dimensional Riemannian manifold, and let $V$ be a bounded open subset with $\overline{V}\subset M$.
		Suppose that  the following two conditions hold: 
		
		\noindent (a) The scalar curvature of $V$ satisfies
		\begin{equation}\label{scalar_curvature_lowerbound}
			\fint_{V} \operatorname{Sc} \, d\mu :=\frac{\int_{V} \operatorname{Sc} \, d\mu}{\mathrm{Vol}(V)} \ge n(n-1)K,
		\end{equation}	
		\noindent (b)  There exists $\beta_0 > 0$ such that the isoperimetric profile of $V$ satisfies
		\begin{equation}\label{comparison_iso}
			\operatorname{I}(V, \beta) \coloneqq \inf_{\substack{\Omega \subset V \\ \mathrm{Vol}(\Omega) = \beta}} \mathrm{Area}(\partial \Omega) \geq \operatorname{I}(M^n_K, \beta) \quad \text{for all } \beta < \beta_0,
		\end{equation}
		where  $M^n_K$ is the  space form of constant sectional curvature $K$.\\
		Then the sectional curvature of $V$ satisfies 
		$$\operatorname{Sec}(x)=K  \text{ for all } x\in V.$$ 
	\end{thm}
	
	\subsection{Organization of this paper}	
	
	The present paper is organized as follows. In Section 2, we introduce the definitions of $\mathcal{L}^{\pm}_{\alpha}$-functionals and $\mathcal{W}^{\pm}_{\alpha}$-functionals. Then we investigate the relationship between $\mathbb{G}_{\alpha}^{\pm}(V,g)$ (resp. $\mathbb{Y}_{\alpha}^{\pm}(V,g)$) and $\mathbb{L}^{\pm}_{\alpha}(V,g, \tau)$ (resp. $\mathcal{\mu}^{\pm}_{\alpha}(V,g,\tau)$). In Section 3, we present preliminary calculations for the power series expansion formulas of $\mathcal{L}^{\pm}_{\alpha}$-functionals and $\mathcal{W}^{\pm}_{\alpha}$-functionals. In Section 4, we derive the power series expansion formulas for the $\mathcal{L}^{\pm}_{\alpha}$-functionals, and subsequently provide the proof of Theorem~\ref{const_rigidity}. In Section 5,
	we derive the power series expansion formulas for the $\mathcal{W}^{\pm}_{\alpha}$-functionals when choosing $\xi_{\pm}\in \mathscr{B}_p(V)$, and subsequently provide the proof of Theorem~\ref{Yamabe_rigidity}.
	In Section 6, we give the proof of Theorem \ref{rigidity_iso_profile} by using the expansion formulas of $\mathcal{L}^{\pm}_{\alpha}$-functionals.
	
	\section{Definitions of $\mathcal{L}^{\pm}_{\alpha}$-functionals and $\mathcal{W}^{\pm}_{\alpha}$-functionals }
	
	For the proofs of Theorem \ref{const_rigidity} and Theorem \ref{Yamabe_rigidity}, it is convenient to use the power series expansion formulas of the $\tau$-dependent $\mathcal{L}^{\pm}_{\alpha}$-functionals and the $\mathcal{W}^{\pm}_{\alpha}$-functionals. The $\mathcal{W}^{+}_{\alpha}$- functional was introduced by Jeffrey S. Case as a generalization of Perelman's $\mathcal{W}$-functional (cf. Definition 3.8 in \cite{Jcase}), and $\mathcal{W}^{-}_{\alpha}$-functionals and $\mathcal{L}^{\pm}_{\alpha}$-functionals can be defined in a similar way.
	\begin{defn}\label{Def_sec2}
		Let $(M^n,g)$ be an $n$-dimensional Riemannian manifold with $n \geq 3$, and let $V$ be an open bounded subset of $M^n$. The constants $\gamma$ and $\theta$ are defined in \eqref{def:gamma} and \eqref{def:theta}, respectively, while $\mathcal{N}_{\alpha, n}$ and $\mathcal{G}_{\alpha, n}$ denote the optimal constants given in \eqref{def_optimal_const}. Throughout this paper, we fix an arbitrary positive constant $\mathfrak{m}$.
		\\
		\noindent (i) For $0<\alpha<1$, denote
		$$
		\Gamma_{\alpha}:=\frac{n}{2}\cdot\frac{1-\alpha}{1+\alpha} \quad \text{and} \quad \Sigma^-_{\alpha}:= \left(\frac{\Gamma_{\alpha}}{\Gamma_{\alpha}+1}\right)^{1-\frac{\Gamma_{\alpha}}{2\Gamma_{\alpha}+1}} \mathcal{N}_{\alpha,n}^{\frac{2}{\gamma}\cdot\frac{\Gamma_{\alpha}}{2\Gamma_{\alpha}+1}},
		$$
		where $\mathcal{N}_{\alpha,n}=\mathbb{G}_{\alpha}^-(\mathbb{R}^n,g_{\mathbb{R}^n})$  is defined in  \eqref{def_optimal_const}.\\
		The\textbf{ $\mathcal{L}^-_{\alpha}$- functional} is defined by
		$$
		\begin{aligned}
			\mathcal{L}^-_{\alpha}(V,g, u, \tau) = \tau^{\Gamma_{\alpha}+1} \int_V \left|\nabla u\right|^2 d\mu + \mathfrak{m}\left( \tau^{-\Gamma_{\alpha}}\int_V u^{2\alpha}d\mu - \int_V u^{\alpha+1}d\mu\right) + \mathfrak{m}\left(1-\frac{2\Gamma_{\alpha}+1}{\Gamma_{\alpha}}\mathfrak{m}^{-\frac{\Gamma_{\alpha}}{2\Gamma_{\alpha}+1}}\Sigma^-_{\alpha}\right),
		\end{aligned}
		$$
		and the \textbf{$\mathcal{W}^-_{\alpha}$-functional}  is defined by
		$$
		\begin{aligned}
			\mathcal{W}^-_{\alpha}(V,g, u, \tau) = \mathcal{L}^-_{\alpha}(V,g, u, \tau) + \frac{1}{2(1+\alpha)}\tau^{\Gamma_{\alpha}+1} \int_V \operatorname{Sc} \cdot u^2 d\mu.
		\end{aligned}
		$$

		\noindent (ii) For $1<\alpha\le \frac{n}{n-2}$, denote
		$$
		\Theta_{\alpha}:=\frac{n}{4}\frac{\alpha-1}{\alpha} \quad \text{and} \quad \Sigma^+_{\alpha}:= \left(\frac{\Theta_{\alpha}}{1-2\Theta_{\alpha}}\right)^{1-\frac{\Theta_{\alpha}}{1-\Theta_{\alpha}}} \mathcal{G}_{\alpha,n}^{\frac{2}{\theta}\cdot\frac{\Theta_{\alpha}}{1-\Theta_{\alpha}}},
		$$
		where $ \mathcal{G}_{\alpha,n}=\mathbb{G}_{\alpha}^+(\mathbb{R}^n,g_{\mathbb{R}^n})$  is defined in  \eqref{def_optimal_const}.\\
		The \textbf{$\mathcal{L}^+_{\alpha}$- functional}  is defined by
		$$
		\begin{aligned}
			\mathcal{L}^+_{\alpha}(V,g, u, \tau) = \tau^{1-2\Theta_{\alpha}} \int_V \left|\nabla u\right|^2 d\mu  + \mathfrak{m}\left( \tau^{-\Theta_{\alpha}}\int_V u^{\alpha+1} d\mu - \int_V u^{2\alpha} d\mu \right) + \mathfrak{m}\left(1-\frac{1-\Theta_{\alpha}}{\Theta_{\alpha}}\mathfrak{m}^{-\frac{\Theta_{\alpha}}{1-\Theta_{\alpha}}}\Sigma^+_{\alpha}\right),
		\end{aligned}
		$$
		and
		the \textbf{$\mathcal{W}^+_{\alpha}$- functional}  is defined by
		$$
		\begin{aligned}
			\mathcal{W}^+_{\alpha}(V,g, u, \tau) = \mathcal{L}^+_{\alpha}(V,g, u, \tau) + \frac{1}{2(1+\alpha)}\tau^{1-2\Theta_{\alpha}} \int_V \operatorname{Sc} \cdot u^2 d\mu.
		\end{aligned}
		$$
		
		\noindent   (iii)
		The \textbf{$\mathbb{L}^{\pm}_{\alpha}$-constant} of $V$ is defined by
		$$
		\mathbb{L}^-_{\alpha}(V,g, \tau):= \inf_{\substack{u \in W_0^{1,2}(V),\int_{V} u^{\alpha+1} d\mu = 1}} \mathcal{L}^-_{\alpha}(V,g, u, \tau), \quad \mathbb{L}^{+}_{\alpha}(V,g, \tau):= \inf_{\substack{u \in W_0^{1,2}(V),\int_{V} u^{2\alpha} d\mu = 1}} \mathcal{L}^{+}_{\alpha}(V,g, u, \tau),
		$$
		and the \textbf{$\mathbb{\mu}^{\pm}_{\alpha}$-constant} of $V$ is defined by
		$$
		\mathbb{\mu}^-_{\alpha}(V,g, \tau):= \inf_{\substack{u \in W_0^{1,2}(V),\int_{V} u^{\alpha+1} d\mu = 1}} \mathcal{W}^-_{\alpha}(V,g, u, \tau),\quad	\mathbb{\mu}^{+}_{\alpha}(V,g, \tau):= \inf_{\substack{u \in W_0^{1,2}(V),\int_{V} u^{2\alpha} d\mu = 1}} \mathcal{W}^{+}_{\alpha}(V,g, u, \tau).
		$$
	\end{defn}

	\begin{rem}
		Note that $\mathcal{W}^{\pm}_{\alpha}(V,g, u, \tau)$ are the generalization of Perelman's $\mathcal{W}$-functional. Indeed,
		when $\alpha\to 1$ and let $\mathfrak{m}$ be a constant depending on $\alpha$ such that  $\mathfrak{m}:=\mathfrak{m}(\alpha)$ satisfying $\frac{|\alpha-1|}{2}\mathfrak{m}(\alpha)\to 1$, $\mathcal{W}^{\pm}_{\alpha}(V,g, u, \tau)$
		tends to the Perelman's $\mathcal{W}$-functional
		\begin{equation}\label{Perel_W}
			\mathcal{W}(V,g, u, \tau)=\tau \int_V\left( \left|\nabla u\right|^2+\frac{1}{4}\operatorname{Sc}\cdot u^2 \right)d\mu-\int_V u^2\log u^2d\mu-\frac{n}{2}\log \tau-n,
		\end{equation}
		and
		$\mathcal{L}^{\pm}_{\alpha}(V,g, u, \tau)$
		tends to following  functional which related to log-Sobolev inequality:
		\begin{equation}\label{log_func}	
			\mathcal{L}(V,g, u, \tau)=\tau \int_V \left|\nabla u\right|^2 d\mu-\int_V u^2\log u^2-\frac{n}{2}\log \tau-n.
		\end{equation}
	\end{rem}
	
	The following theorem reveals the relationship between $\mathbb{G}_{\alpha}^{\pm}(V,g)$ (resp.\ $\mathbb{Y}_{\alpha}^{\pm}(V,g)$) and $\mathbb{L}^{\pm}_{\alpha}(V,g, \tau)$ (resp.\ $\mathcal{\mu}^{\pm}_{\alpha}(V,g,\tau)$).
	
	\begin{thm}\label{connection}
		$\mathbb{G}_{\alpha}^{\pm}(V,g)\ge \mathbb{G}_{\alpha}^{\pm}(\mathbb{R}^n,g_{\mathbb{R}^n})$	if and only	if $\mathbb{L}^{\pm}_{\alpha}(V,g, \tau)\ge 0$ for any $\tau>0$. Moreover,	$\mathbb{Y}_{\alpha}^{\pm}(V,g)\ge \mathbb{G}_{\alpha}^{\pm}(\mathbb{R}^n,g_{\mathbb{R}^n})$	if and only	if $\mathcal{\mu}^{\pm}_{\alpha}(V,g,\tau)\ge 0$ for any $\tau>0$.
	\end{thm}
	\begin{proof}
		Note that $\frac{\Gamma_{\alpha}+1}{\Gamma_{\alpha}} = \frac{1-\gamma}{\gamma \alpha}$  and $\frac{1-2\Theta_{\alpha}}{\Theta_{\alpha}}=\frac{2(1-\theta)}{\theta(1+\alpha)}$.
		Theorem \ref{connection} follows from the 
		straightforward calculus exercise: if $A, B \geq 0$, then
		$$
		\inf _{x>0}\left\{A \tau^{p}+\mathfrak{m} B \tau^{-q}\right\}=\frac{\mathfrak{m}(p+q)}{p}\left(\frac{p }{\mathfrak{m} q}A B^{\frac{p}{q}}\right)^{\frac{q}{p+q}}
		$$
		for all $\tau>0$, with equality if and only if
		\begin{equation}\label{tau_connection}
			\tau=\left(\frac{\mathfrak{m} q}{p}\cdot\frac{ B}{ A}\right)^{\frac{1}{p+q}}.
		\end{equation}

	\end{proof}

	\section{Preliminaries for the calculation of Power series expansion formulas }	
	
	In this section, we present preliminary calculations for the power series expansion formulas of $\mathcal{L}^{\pm}_{\alpha}$-functionals and $\mathcal{W}^{\pm}_{\alpha}$-functionals.
	
	Here and below, we will use the following notation:
	for $p>0$
	\begin{equation}\label{def_mathcsr_B}
		\mathscr{B}(p,q):= 	\begin{cases}
			\mathrm{B}(p,q)  & \text{if } q>0,\\
			\mathrm{B}(p,-q-p+1) & \text{if } q<0 	 \text{ and } -p-q+1>0,	
		\end{cases}
	\end{equation}	
	where $\mathrm{B}(\cdot, \cdot)$ is the Euler beta-function.

	We will use the following identities in our calculations:
	When
	\begin{equation*}
		\begin{cases}
			q_1>-n, q_2>-1, & \text{if }	\alpha<1, \\
			q_1>-n, -q_2-\frac{n+q_1}{2}>0,	 & \text{if } \alpha>1,
		\end{cases}
	\end{equation*}
	we have 
	\begin{equation}\label{3.1_e0}
		\begin{aligned}
			& \int_{\mathbb{R}^n} |y|^{q_1}\left(1+\frac{(\alpha-1)|y|^2}{8}\right)_+^{q_2} dy^n \\ 
			&= \omega_{n-1}\int_0^{\infty} r^{n-1+q_1}\left(1+\frac{(\alpha-1)r^2}{8}\right)_+^{q_2} dr\\
			&\xlongequal{s=\frac{|\alpha-1|r^2}{8}}
			\begin{cases}
				\frac{\omega_{n-1}}{2} \left(\frac{1-\alpha}{8}\right)^{-\frac{n+q_1}{2}}\int_0^{1} s^{\frac{n+q_1}{2}-1}\left(1-s\right)_+^{q_2} ds, & \text{if } \alpha<1, q_1>-n, q_2>-1,\\
				\frac{\omega_{n-1}}{2} \left(\frac{\alpha-1}{8}\right)^{-\frac{n+q_1}{2}}\int_0^{\infty} s^{\frac{n+q_1}{2}-1}\left(1+s\right)^{q_2} ds, & \text{if } \alpha>1, q_1>-n, -q_2-\frac{n+q_1}{2}>0 , 			
			\end{cases}
			\\
			&= \frac{\omega_{n-1}}{2} \left(\frac{|\alpha-1|}{8}\right)^{-\frac{n+q_1}{2}} \mathscr{B}\left(\frac{n+q_1}{2},q_2+1\right),
		\end{aligned}
	\end{equation}
	where we have used
	\begin{equation}\label{beta_func_def}
		\text{ $\int_0^{1} s^{p-1}\left(1-s\right)^{q-1} ds=\mathrm{B}(p, q)$  and $\int_0^{\infty} s^{p-1}\left(1+s\right)^{-(p+q)} ds=\mathrm{B}(p, q)$ for $p>0$ and $q>0$.} 
	\end{equation}
	Here and in what follows, $\omega_{n-1}$ denotes the volume of the standard $(n-1)$-sphere.
	
	When
	\begin{equation*}
		\begin{cases}
			q_1>-n-2, q_2>-1, & \text{if }	\alpha<1, \\
			q_1>-n-2, -q_2-\frac{n+q_1+2}{2}>0,	 & \text{if } \alpha>1,
		\end{cases}
	\end{equation*}
	we have for any symmetric matrix $A_{ij}$:
	\begin{equation}\label{3.1_e11}
		\begin{aligned}
			&	\ \ \ \  \int_{\mathbb{R}^n} |y|^{q_1}\left(1+\frac{(\alpha-1)|y|^2}{8}\right)_+^{q_2}   A_{i j} y^i y^j dy^n \\ 
			&=  \int_0^{\infty} r^{n+1+q_1}\left(1+\frac{(\alpha-1)r^2}{8}\right)_+^{q_2} dr\int_{s^{n-1}(1)} A_{i j}  z^i z^j dz^{n-1} \\
			&\xlongequal{s=\frac{|\alpha-1|r^2}{8}}
			\begin{cases}
				\frac{\omega_{n-1}}{2n} \left(\frac{1-\alpha}{8}\right)^{-\frac{n+q_1+2}{2}}\operatorname{tr}(A)\int_0^{1} s^{\frac{n+q_1+2}{2}-1}\left(1-s\right)_+^{q_2} ds, & \text{if } \alpha<1, q_1>-n-1, q_2>-1,\\
				\frac{\omega_{n-1}}{2n} \left(\frac{\alpha-1}{8}\right)^{-\frac{n+q_1+2}{2}}\operatorname{tr}(A)\int_0^{\infty} s^{\frac{n+q_1+2}{2}-1}\left(1+s\right)^{q_2} ds, & \text{if } \alpha>1, q_1>-n-1, -q_2-\frac{n+q_1+2}{2}>0,  		
			\end{cases}
			\\
			&= 		\frac{\omega_{n-1}}{2n} \left(\frac{|\alpha-1|}{8}\right)^{-\frac{n+q_1+2}{2}} \mathscr{B}(\frac{n+q_1+2}{2},q_2+1) \operatorname{tr}(A),
		\end{aligned}
	\end{equation}
	where we have used \eqref{beta_func_def} and
	$$
	\int_{s^{n-1}(1)}A_{i j}  z^i z^j =\sum_{i=1} A_{i i} \int_{s^{n-1}(1)}  (z^i)^2dz^{n-1}=\frac{\int_{s^{n-1}(1)}  \sum\limits^n_{i=1}(z^i)^2dz^{n-1}}{n}\operatorname{tr}(A)=\frac{\omega_{n-1}}{n} \operatorname{tr}(A).
	$$
	Moreover,
	when
	\begin{equation*}
		\begin{cases}
			q_1>-n-4, q_2>-1, & \text{if }	\alpha<1, \\
			q_1>-n-4, -q_2-\frac{n+q_1+4}{2}>0,	 & \text{if } \alpha>1,
		\end{cases}
	\end{equation*} 
	we have for any four tensor $\lambda_{ijkl}$:	
	\begin{equation}\label{3.1_e21}
		\begin{aligned}
			&	\ \ \ \  \int_{\mathbb{R}^n} |y|^{q_1}\left(1+\frac{(\alpha-1)|y|^2}{8}\right)_+^{q_2}   \lambda_{i j k l} y^i y^j y^k y^l dy^n \\ 
			&=  \int_0^{\infty} r^{n+3+q_1}\left(1+\frac{(\alpha-1)r^2}{8}\right)_+^{q_2} dr\int_{s^{n-1}(1)}  \lambda_{i j k} z^i z^j z^{k } z^ldz^{n-1} \\
			&\xlongequal{s=\frac{|\alpha-1|r^2}{8}}
			\begin{cases}
				\frac{\omega_{n-1}}{2n(n+2)} \left(\frac{1-\alpha}{8}\right)^{-\frac{n+q_1+4}{2}}\operatorname{E}(\lambda)\int_0^{1} s^{\frac{n+q_1+4}{2}-1}\left(1-s\right)_+^{q_2} ds, & \text{if } \alpha<1,   q_1>-(n+4),  q_2>-1,\\
				\frac{\omega_{n-1}}{2n(n+2)} \left(\frac{\alpha-1}{8}\right)^{-\frac{n+q_1+4}{2}}\operatorname{E}(\lambda)\int_0^{\infty} s^{\frac{n+q_1+4}{2}-1}\left(1+s\right)^{q_2} ds, & \text{if } \alpha>1, q_1>-(n+4),  -q_2-\frac{n+q_1+4}{2}>0,
			\end{cases}
			\\
			&= 		\frac{\omega_{n-1}}{2n(n+2)} \left(\frac{|\alpha-1|}{8}\right)^{-\frac{n+q_1+4}{2}}\mathscr{B}(\frac{n+q_1+4}{2},q_2+1)\operatorname{E}(\lambda),
		\end{aligned}
	\end{equation}
	where 
	\begin{equation}\label{def_of_E}
		\operatorname{E}(\lambda):= \sum\limits_{i j=1}^n\left(\lambda_{i i j j}+\lambda_{i j i j}+\lambda_{i j j i}\right),
	\end{equation}
	we have used \eqref{beta_func_def} and
	$$
	\begin{aligned}
		&\ \ \ \int_{S^{n-1}(1)} \lambda_{i j k l}  z^i z^j z^k z^l dz^{n-1}\\
		& =\frac{\omega_{n-1}}{n(n+2)}\left\{3 \sum_{i=1}^n \lambda_{i i i i}+\sum_{i \neq j}\left(\lambda_{i i j j}+\lambda_{i j i j}+\lambda_{i j j i}\right)\right\}\\
		& =\frac{\omega_{n-1}}{n(n+2)} \operatorname{E}(\lambda) ,
	\end{aligned}
	$$
	since 
	$\int_{S^{n-1}(1)} z_i^4 dz^{n-1}=3 \int_{S^{n-1}(1)} z_i^2 z_j^2 dz^{n-1}=\frac{3 \pi^{n / 2}}{(n+2)\Gamma(\frac{n}{2} +1)}$ for $i\ne j$ (c.f. (A.4) and (A.5) in \cite{G} )  and 
	each the integral of which $a_i$ appears for odd times is zero because the integral over one hemisphere cancels the integral over the other.

	Next, we derive the following lemma, which is required for calculating each term in the power series expansions of the $\mathcal{L}^{\pm}_{\alpha}$-functionals and $\mathcal{W}^{\pm}_{\alpha}$-functionals.
	
	\begin{lem}\label{key_lemma}
		Let $(M^n,g)$ be a Riemannian manifold of dimension $n\geq 3$ and $p\in V\subset\mathring{M^n}$, where $V$ is a neighborhood of $p$. 
		Denote the function $\operatorname{H}:\mathbb{R}_+\to \mathbb{R}_+$:
		$$\operatorname{H}(r)=\left(1+ \frac{(\alpha-1)r^2}{8}\right)_+^{\frac{1}{1-\alpha}}.$$  
		Take the function $\xi(x,t)$ on $M$ such that $p \in \operatorname{supp}(\eta) \subset\subset V$, which has the following expansion around $p$:
		\begin{equation}\label{expansion_of_xi}
			\begin{aligned}
				&\xi^2(x,t)=\sum_{k=0}^2 \phi_k(x)t^k + o(t^2),\\
				&\phi_0(x)=1+\textbf{a}_{ij}x^ix^j+e_{ijk}x^ix^jx^k+b_{ijkl}x^ix^jx^kx^l+o(|x|^4),\\
				&\phi_1(x) =\beta_1+q_ix^i+d_{ij}x^ix^j+o(|x|^2),\\
				&\phi_2(x)=\beta_2+o(1),
			\end{aligned}
		\end{equation}
		where $\{x^k\}^n_{k=1}$ is the normal coordinates of $T_pM$.
		
		\noindent(i)
		When  either $0<\alpha<1$, or ($1<\alpha\le \frac{n}{n-2}$ and $\frac{m}{\alpha-1}-\frac{n}{2}-1>0$), we have
		\begin{equation}\label{expansion_D1st}
			t^{-\frac{n}{2}}  \int_V 
			\left( \operatorname{H}(\frac{d(p,x)}{\sqrt{t}})\xi(x,t)\right)^m d\mu=D_0(m)+D_1(m)t+o(t),
		\end{equation}
		Moreover, when either $0<\alpha<1$, or ($1<\alpha\le \frac{n}{n-2}$ and $\frac{m}{\alpha-1}-\frac{n}{2}-2>0$),  we have
		\begin{equation}\label{expansion_D2rd}
			t^{-\frac{n}{2}}  \int_V 
			\left( \operatorname{H}(\frac{d(p,x)}{\sqrt{t}})\xi(x,t)\right)^m d\mu=D_0(m)+D_1(m)t+D_2(m)t^2+o(t^2),
		\end{equation}
		where
		$$
		D_0(m)=\frac{\omega_{n-1}}{2}\left(\frac{|\alpha-1|}{8}\right)^{-\frac{n}{2}}\mathscr{B}(\frac{n}{2},\frac{m}{1-\alpha}+1),
		$$
		$$
		\frac{D_1(m)}{D_0(m)}=\frac{m}{2}\beta_1+\frac{m}{2n}\left(\frac{|\alpha-1|}{8}\right)^{-1}\frac{\mathscr{B}(\frac{n}{2}+1,\frac{m}{1-\alpha}+1)}{\mathscr{B}(\frac{n}{2},\frac{m}{1-\alpha}+1)}\operatorname{tr}(\textbf{a})-\frac{1}{6n}
		\left(\frac{|\alpha-1|}{8}\right)^{-1}\frac{\mathscr{B}(\frac{n}{2}+1,\frac{m}{1-\alpha}+1)}{\mathscr{B}(\frac{n}{2},\frac{m}{1-\alpha}+1)}\operatorname{Sc}(p),
		$$
		$$
		\begin{aligned}
			\frac{D_2(m)}{D_0(m)}&=\frac{1}{n(n+2)}	\left(\frac{|\alpha-1|}{8}\right)^{-2}\frac{\mathscr{B}(\frac{n}{2}+2,\frac{m}{1-\alpha}+1)}{\mathscr{B}(\frac{n}{2},\frac{m}{1-\alpha}+1)}\left[\frac{m}{2}\operatorname{E}(b)+\frac{m(m-2)}{8}\operatorname{E}(\textbf{a} \otimes \textbf{a})+\operatorname{E}(v)-\frac{m}{12}\operatorname{E}(\textbf{a} \otimes \operatorname{Rc})\right]\\
			&+\frac{1}{n}	\left(\frac{|\alpha-1|}{8}\right)^{-1}\frac{\mathscr{B}(\frac{n}{2}+1,\frac{m}{1-\alpha}+1)}{\mathscr{B}(\frac{n}{2},\frac{m}{1-\alpha}+1)}\left[\frac{m}{2}\operatorname{tr}(d)+\frac{m(m-2)}{4}\beta_1 \operatorname{tr}(\textbf{a})-\frac{m}{12}\beta_1 \operatorname{Sc}(p)\right]\\
			&+\left(\frac{m}{2}\beta_2+\frac{m(m-2)}{8}\beta_1^2\right),
		\end{aligned}
		$$
		where $\mathscr{B}(p,q)$ is defined in \eqref{def_mathcsr_B}, $\operatorname{E}$ is defined in \eqref{def_of_E} and  $v$ is the tensor defined as
		\begin{equation}\label{def_v}
			v_{ijkl}=\frac{1}{24}\left(-\frac{3}{5} \nabla_k \nabla_l R_{ij}-\frac{2}{15}\sum\limits_{s, t=1}^n R_{isjt}R_{kslt}+\frac{1}{3} R_{ij} R_{k l}\right)(p). 
		\end{equation}
		
		\noindent(ii)  When  either $0<\alpha<1$, or ($1<\alpha\le \frac{n}{n-2}$ and $\frac{2\alpha}{\alpha-1}-\frac{n}{2}-2>0$), we have
		\begin{equation}\label{expansion_A1st}
			t^{1-\frac{n}{2}}  \int_V \left|\nabla
			\left( \operatorname{H}(\frac{d(p,x)}{\sqrt{t}})\xi(x,t)\right)\right|^2 d\mu  =A_0+A_1t+o(t)	,
		\end{equation}
		Moreover,  when  either $0<\alpha<1$, or ($1<\alpha\le \frac{n}{n-2}$ and $\frac{2\alpha}{\alpha-1}-\frac{n}{2}-3>0$), we have
		\begin{equation}\label{expansion_A2rd}
			t^{1-\frac{n}{2}}  \int_V \left|\nabla
			\left( \operatorname{H}(\frac{d(p,x)}{\sqrt{t}})\xi(x,t)\right)\right|^2 d\mu  =A_0+A_1t+A_2t^2+o(t^2)	,
		\end{equation}
		where
		$$
		A_0=\frac{\omega_{n-1}}{32}\left(\frac{|\alpha-1|}{8}\right)^{-\frac{n+2}{2}}\mathscr{B}(\frac{n}{2}+1,\frac{2\alpha}{1-\alpha}+1),
		$$
		$$
		\frac{A_1}{A_0}=\frac{1}{n}\left(\frac{|\alpha-1|}{8}\right)^{-1}	\frac{\mathscr{B}(\frac{n}{2}+2,\frac{2\alpha}{1-\alpha}+1)}{\mathscr{B}(\frac{n}{2}+1,\frac{2\alpha}{1-\alpha}+1)}\left(\operatorname{tr}(\textbf{a})-\frac{1}{6}\operatorname{Sc}(p)\right)-\frac{8}{n}	\frac{\mathscr{B}(\frac{n}{2}+1,\frac{2\alpha}{1-\alpha}+2)}{\mathscr{B}(\frac{n}{2}+1,\frac{2\alpha}{1-\alpha}+1)}\operatorname{tr}(\textbf{a})+\beta_1, 
		$$
		$$
		\begin{aligned}
			\frac{A_2}{A_0}&=\frac{1}{n(n+2)}	\left(\frac{|\alpha-1|}{8}\right)^{-2}\frac{\mathscr{B}(\frac{n}{2}+3,\frac{2\alpha}{1-\alpha}+1)}{\mathscr{B}(\frac{n}{2}+1,\frac{2\alpha}{1-\alpha}+1)}\left(\operatorname{E}(b)+\operatorname{E}(v)-\frac{1}{6}\operatorname{E}(\textbf{a} \otimes \operatorname{Rc})\right)\\
			&+\frac{1}{n}\left(\frac{|\alpha-1|}{8}\right)^{-1}	\frac{\mathscr{B}(\frac{n}{2}+2,\frac{2\alpha}{1-\alpha}+1)}{\mathscr{B}(\frac{n}{2}+1,\frac{2\alpha}{1-\alpha}+1)}\left(\operatorname{tr}(d)-\frac{1}{6}\beta_1\operatorname{Sc}(p)\right)+\beta_2+\frac{16}{n}	\frac{\mathscr{B}(\frac{n}{2}+1,\frac{2}{1-\alpha}+1)}{\mathscr{B}(\frac{n}{2}+1,\frac{2\alpha}{1-\alpha}+1)}\operatorname{tr}(\textbf{a}^2)\\
			&	\quad -\frac{8}{n}	\frac{\mathscr{B}(\frac{n}{2}+1,\frac{2\alpha}{1-\alpha}+2)}{\mathscr{B}(\frac{n}{2}+1,\frac{2\alpha}{1-\alpha}+1)}\operatorname{tr}(d)+
			\frac{1}{n(n+2)}	\left(\frac{|\alpha-1|}{8}\right)^{-1}\frac{\mathscr{B}(\frac{n}{2}+2,\frac{2\alpha}{1-\alpha}+2)}{\mathscr{B}(\frac{n}{2}+1,\frac{2\alpha}{1-\alpha}+1)}\left(-16\operatorname{E}(b)+\frac{4}{3}\operatorname{E}(\textbf{a} \otimes \operatorname{Rc})\right), 
		\end{aligned}
		$$
		
		\noindent(iii) When either $0<\alpha<1$, or ($1<\alpha\le \frac{n}{n-2}$ and $\frac{2\alpha}{\alpha-1}-\frac{n}{2}-1>0$), we have
		\begin{equation}
			\begin{aligned}
				\begin{aligned}
					&t^{1-\frac{n}{2}}  \int_V \operatorname{Sc}(x)\cdot  \left( \operatorname{H}(\frac{d(p,x)}{\sqrt{t}})\xi(x,t)\right)^2  d\mu  \\ =&\frac{\omega_{n-1}}{2}\left(\frac{|\alpha-1|}{8}\right)^{-\frac{n}{2}}\mathscr{B}(\frac{n}{2},\frac{2}{1-\alpha}+1)\times\\
					&\Big\{\operatorname{Sc}(p)t+\frac{1}{n}\left(\frac{|\alpha-1|}{8}\right)^{-1}	\frac{\mathscr{B}(\frac{n}{2}+1,\frac{2}{1-\alpha}+1)}{\mathscr{B}(\frac{n}{2},\frac{2}{1-\alpha}+1)}\left(\frac{1}{2}\Delta \operatorname{Sc}(p)-\frac{1}{6}\operatorname{Sc}^2(p)+\operatorname{tr}(\textbf{a})\operatorname{Sc}(p)\right)t^2+\beta_1 \operatorname{Sc}(p)t^2+o(t^2)\Big\}	,
				\end{aligned}
			\end{aligned}
		\end{equation}
	\end{lem}
	\begin{proof} \textbf{(i)} 
		Let $\Sigma_p\subset T_pM$ be the segment domain of $p$, and let $\operatorname{det}\left(g_{k \ell}(x)\right) := 0$ and $\xi^2(x,t):= 0$ outside $\Sigma_p$.
		Recall that in the normal coordinates $\{x^k\}^n_{k=1}$ of $T_pM$,  $\operatorname{det}\left(g_{k \ell}(x)\right) $ has the following  power series expansion near $p$ (see Lemma 3.4 in p. 210 of \cite{STbook})
		\begin{equation}\label{expansion_detg}
			\begin{aligned}
				\operatorname{det}\left(g_{k \ell}(x)\right)^\frac{1}{2} &  =1-\frac{1}{6} R_{ij}(p) x^i x^j-\frac{1}{12} \nabla_k R_{ij}(p) x^i x^j x^k +v_{ijkl} x^i x^j x^k x^l  +O\left(|x|^5\right),
			\end{aligned}
		\end{equation}	
		where $v$ is the tensor defined in \eqref{def_v}. 
		By \eqref{expansion_of_xi} and \eqref{expansion_detg}, we have 
		\begin{equation}\label{key_lemma_term1}
			\begin{aligned}
				&\ \ \ \ \xi^m(x,t)\operatorname{det}\left(g_{k \ell}(x)\right)^\frac{1}{2}\\
				=&\left\{1+\frac{m}{2}\left(\textbf{a}_{ij}x^ix^j+\beta_1t\right)+\frac{m}{2}\left(e_{ijk}x^ix^jx^k+q_i x^it\right)+\frac{m}{2}\left(b_{ijkl}x^ix^jx^kx^l+d_{ij} x^ix^jt+\beta_2 t^{2}\right)\right.\\
				&\left.+\frac{m(m-2)}{8}\left(\beta_1^2t^2+2\beta_1 \textbf{a}_{ij}x^ix^jt+\textbf{a}_{ij}\textbf{a}_{kl}x^ix^jx^kx^l\right)+o(|x|^2)t+o(|x|^4)+o(1)t^2\right\}\times\\
				&\left\{1-\frac{1}{6} R_{ij}(p) x^i x^j-\frac{1}{12} \nabla_k R_{ij}(p) x^i x^j x^k +v_{ijkl} x^i x^j x^k x^l +o(|x|^4) \right\}\\
				&=Q_{m,1}+Q_{m,2}+P,
			\end{aligned}
		\end{equation}	
		where
		\begin{align}
			Q_{m,1}:=&1+\left(\frac{m}{2}\textbf{a}_{ij}-\frac{1}{6} R_{ij}(p) \right)x^i x^j+\frac{m}{2}\beta_1t,\label{def_Q_1}\\
			Q_{m,2}:=&  \left\{\frac{m}{2}\left(e_{ijk}x^ix^jx^k+q_i x^it\right)-\frac{1}{12} \nabla_k R_{ij}(p) x^i x^j x^k\right\}\nonumber\\
			& +\left\{\left(\frac{m}{2}b_{ijkl}+\frac{m(m-2)}{8}\textbf{a}_{ij}\textbf{a}_{kl}+v_{ijkl}-\frac{m}{12}\textbf{a}_{ij} R_{kl}(p)\right)x^ix^jx^kx^l\right.\nonumber\\
			&\left.\ \ \ \ +\left(\frac{m}{2}d_{ij}+\frac{m(m-2)}{4}\beta_1 \textbf{a}_{ij}-\frac{m}{12}\beta_1R_{ij}(p)\right) x^ix^jt +\left(\frac{m}{2}\beta_2+\frac{m(m-2)}{8}\beta_1^2\right) t^2 \right\} ,\label{def_Q_2}\\
			P  :=& o(|x|^2)t+o(|x|^4)+o(1)t^2.\nonumber
		\end{align}

		Denote $$W:= \Sigma_p \cap exp^{-1}_p \left(V\cap \operatorname{supp}\{\xi\}\right),$$ we have
		\begin{equation}\label{term1}
			\begin{aligned}
				&\quad   t^{-\frac{n}{2}}  \int_V 
				\left( \operatorname{H}(\frac{d(p,x)}{\sqrt{t}})\xi(x,t)\right)^m d\mu\\
				&= t^{-\frac{n}{2}}  \int_{W} 
				\operatorname{H}^m(\frac{|x|}{\sqrt{t}}) \xi^m(x,t)\operatorname{det}\left(g_{k \ell}(x)\right)^\frac{1}{2}dx^n\\
				&= t^{-\frac{n}{2}}  \int_{W} 
				\operatorname{H}^m(\frac{|x|}{\sqrt{t}}) \left(Q_{m,1}+Q_{m,2}+P\right)dx^n.
			\end{aligned}
		\end{equation}
		
		\textbf{Next, we calculate each term of \eqref{term1} separately.}
		
		When $0 < \alpha < 1$, and noting that $\operatorname{H}\left(\frac{d(p,x)}{\sqrt{t}}\right) = 0$ if $d(p,x)^2 \geq \frac{8}{1-\alpha}t$, we have for sufficiently small $t$ and any smooth function $f$ on $M$:
		\begin{equation}\label{alphal1}
			\int_{W} \operatorname{H}^m\left(\frac{|x|}{\sqrt{t}}\right) f dx^n = \int_{T_pM} \operatorname{H}^m\left(\frac{|x|}{\sqrt{t}}\right) f dx^n.
		\end{equation}

		When $1 < \alpha \leq \frac{n}{n-2}$, taking some ball $B(o,r_0) \subset W$, since
		\begin{equation}\label{estimate_Q}
			\begin{aligned}
				&\quad t^{-\frac{n}{2}}\int_{T_pM \setminus W} \operatorname{H}^m\left(\frac{|x|}{\sqrt{t}}\right) |x|^k \, dx^n \\
				&\le t^{-\frac{n}{2}}\int_{T_pM \setminus B(o,r_0)} \operatorname{H}^m\left(\frac{|x|}{\sqrt{t}}\right) |x|^k \, dx^n \\
				&\leq \left(\frac{\alpha-1}{8}\right)^{\frac{m}{1-\alpha}} t^{\frac{m}{\alpha-1} - \frac{n}{2}} \int_{T_pM \setminus B(o,r_0)} |x|^{\frac{2m}{1-\alpha} + k} \, dx^n \\
				&= \left(\frac{\alpha-1}{8}\right)^{\frac{m}{1-\alpha}} \omega_{n-1} t^{\frac{m}{\alpha-1} - \frac{n}{2}} \int^{\infty}_{r_0} r^{\frac{2m}{1-\alpha} + n + k - 1} \, dr \\
				&= o(t^{\frac{k}{2}}) \quad \text{if } \frac{m}{\alpha-1} - \frac{n}{2} - \frac{k}{2} > 0,
			\end{aligned}
		\end{equation}
		we have
		\begin{equation}\label{estimate_Q_1}
			\begin{aligned}
				&\quad \left| t^{-\frac{n}{2}} \int_{T_pM \setminus W} \operatorname{H}^m\left(\frac{|x|}{\sqrt{t}}\right) Q_{m,1} \, dx^n \right| \\
				&\leq l_1 \left| t^{-\frac{n}{2}} \int_{T_pM \setminus B(o,r_0)} \left(1 + \frac{(\alpha-1)|x|^2}{8t}\right)^{\frac{m}{1-\alpha}} |x|^2 \, dx^n \right| \\
				&= o(t) \quad \text{ if } \frac{m}{\alpha-1} - \frac{n}{2} - 1 > 0,
			\end{aligned}
		\end{equation}
		and		
		\begin{equation}\label{estimate_Q_2}
			\begin{aligned}
				&\quad \left| t^{-\frac{n}{2}}\int_{T_pM\backslash W} \operatorname{H}^m(\frac{|x|}{\sqrt{t}})  \left(Q_{m,1}+Q_{m,2}\right) dx^n\right|\\
				&\le l_2\left| t^{-\frac{n}{2}}\int_{T_pM\backslash B(o,r_0)} \left(1+ \frac{(\alpha-1)|x|^2}{8t}\right)^{\frac{m}{1-\alpha}}|x|^4 dx^n\right|\\
				&= o(t^2) \quad \text{ if }\frac{m}{\alpha-1}-\frac{n}{2}-2>0,
			\end{aligned}
		\end{equation}
		where  we have used
		\begin{equation}\label{estimate_forQ}
			Q_{m,1}(x,t)\leq l_1|x|^2 \text{ and } Q_{m,1}(x,t)+Q_{m,2}(x,t)\le l_2|x|^4 \text{ for } x\in  T_pM\backslash B(o,r_0), 0<t\le 1,
		\end{equation}
		where $l_1$ and $l_2$ is a constant depending on $r_0$, $\alpha$, $n$, $m$,  $|\nabla^k\operatorname{\operatorname{Rm}}|(p)$ for $k=0,1,2$, and the expansion coefficients of $\xi^2$ in \eqref{expansion_of_xi}.
		
		By (\ref{3.1_e0}) and (\ref{3.1_e11}) with $q_1=0$ and $q_2=\frac{m}{1-\alpha}$, \eqref{def_Q_1}, \eqref{alphal1} and \eqref{estimate_Q_1}, we get
		for either $0<\alpha<1$ or ($1<\alpha\le \frac{n}{n-2}$ and $\frac{m}{\alpha-1}-\frac{n}{2}-1>0$):
		\begin{equation}\label{D1st_1}
			\begin{aligned}
				&\quad   t^{-\frac{n}{2}}\int_{ W} \operatorname{H}^m(\frac{|x|}{\sqrt{t}})Q_{m,1} dx^n\\
				&= t^{-\frac{n}{2}}\left(\int_{T_pM}-\int_{T_pM\backslash W}\right)\operatorname{H}^m(\frac{|x|}{\sqrt{t}})Q_{m,1} dx^n\\
				&\xlongequal{y=\frac{x}{\sqrt{t}}}\int_{ T_pM}  \left(1+ \frac{(\alpha-1)|y|^2}{8}\right)_+^{\frac{m}{1-\alpha}} Q_{m,1}(t^{\frac{1}{2}}y,t)dy^n+o(t)\\
				&=D_0(m)\left(1+\frac{D_1(m)}{D_0(m)}t\right)+o(t).
			\end{aligned}
		\end{equation}

		Moreover, by  (\ref{3.1_e0}), \eqref{3.1_e11} and (\ref{3.1_e21}) with $q_1=0$ and $q_2=\frac{m}{1-\alpha}$, \eqref{def_Q_2}, \eqref{alphal1}  and \eqref{estimate_Q_2}, we get
		for  either $0<\alpha<1$ or ($1<\alpha\le \frac{n}{n-2}$ and $\frac{m}{\alpha-1}-\frac{n}{2}-2>0$):
		\begin{equation}\label{D2rd_1}
			\begin{aligned}
				&\quad   t^{-\frac{n}{2}}\int_{ W} \operatorname{H}^m(\frac{|x|}{\sqrt{t}})\left(Q_{m,1}+Q_{m,2}\right) dx^n\\
				&= t^{-\frac{n}{2}}\left(\int_{T_pM}-\int_{T_pM\backslash W}\right)\operatorname{H}^m(\frac{|x|}{\sqrt{t}})\left(Q_{m,1}+Q_{m,2}\right) dx^n\\
				&\xlongequal{y=\frac{x}{\sqrt{t}}}\int_{ T_pM}  \left(1+ \frac{(\alpha-1)|y|^2}{8}\right)_+^{\frac{m}{1-\alpha}} \left(Q_{m,1}(t^{\frac{1}{2}}y,t)+Q_{m,2}(t^{\frac{1}{2}}y,t)\right)dy^n+o(t^2)\\
				&=D_0(m)\left(1+\frac{D_1(m)}{D_0(m)}t+\frac{D_2(m)}{D_0(m)}t^2\right)+o(t^2)
			\end{aligned}
		\end{equation}
		where we have used that $\int_{ T_pM} \left(1+ \frac{(\alpha-1)|y|^2}{8}\right)_+^{\frac{1}{1-\alpha}} \left\{\frac{m}{2}\left(e_{ijk}y^iy^jy^k+q_i y^i\right)-\frac{1}{12} \nabla_k R_{ij}(p) y^i y^j y^k\right\}t^{\frac{3}{2}}dy^n=0$ since
		the integral over one hemisphere cancels the integral over the other. 
		
		For any function $F(x)=o(|x|^k)$ and $F(x)$  bounded on $W$,
		we have
		\begin{equation}\label{reason}
			\begin{aligned}
				&\quad   \left|t^{-\frac{n}{2}}\int_{ W} \operatorname{H}^m(\frac{|x|}{\sqrt{t}})F(x)dx^n \right|\\
				&\xlongequal{y=\frac{x}{\sqrt{t}}} \left|\int_{ t^{-\frac{1}{2}}W}  \left(1+ \frac{(\alpha-1)|y|^2}{8}\right)_+^{\frac{m}{1-\alpha}} o(t^{\frac{k}{2}}|y|^k)dy^n\right|\\
				&\le  t^{\frac{k}{2}}\int_{ T_pM}  \left(1+ \frac{(\alpha-1)|y|^2}{8}\right)_+^{\frac{m}{1-\alpha}}  \frac{o(t^{\frac{k}{2}}|y|^k)}{t^{\frac{k}{2}}|y|^k}|y|^k dy^n\\
				&=o(t^{\frac{k}{2}}) \text{ if either $0<\alpha<1$ or ($1<\alpha\le \frac{n}{n-2}$ and $\frac{m}{\alpha-1}-\frac{n}{2}-\frac{k}{2}>0)$},
			\end{aligned}
		\end{equation}
		where we have used $\int_{ T_pM}  \left(1+ \frac{(\alpha-1)|y|^2}{8}\right)_+^{\frac{m}{1-\alpha}} |y|^k dy^n<\infty$  if either $0<\alpha<1$ or ($1<\alpha\le \frac{n}{n-2}$ and $\frac{m}{\alpha-1}-\frac{n}{2}-\frac{k}{2}>0$), so $\int_{ T_pM}  \left(1+ \frac{(\alpha-1)|y|^2}{8}\right)_+^{\frac{m}{1-\alpha}}  \frac{o(t^{\frac{k}{2}}|y|^k)}{t^{\frac{k}{2}}|y|^k}|y|^k dy^n\to 0$ by Lebesgue dominated theorem. It follows that, by also noting 
		$\left(Q_{m,2}+P\right)=o(|x|^2)+O(|x|)t+o(t)$ and $P = o(|x|^2)t+o(|x|^4)+o(1)t^2$, we have
		\begin{equation}\label{D1st_2}
			\begin{aligned}
				\left|t^{-\frac{n}{2}}\int_{ W} \operatorname{H}^m(\frac{|x|}{\sqrt{t}})\left(Q_{m,2}+P\right)dx^n \right|=o(t)
				\text{ for either $0<\alpha<1$ or ($1<\alpha\le \frac{n}{n-2}$ and $\frac{m}{\alpha-1}-\frac{n}{2}-1>0$)},
			\end{aligned}
		\end{equation}
		and   
		\begin{equation}\label{D2rd_2}
			\begin{aligned}
				\left|t^{-\frac{n}{2}}\int_{ W} \operatorname{H}^m(\frac{|x|}{\sqrt{t}})Pdx^n \right|=o(t^2)  \text{ for either $0<\alpha<1$ or ($1<\alpha\le \frac{n}{n-2}$ and $\frac{m}{\alpha-1}-\frac{n}{2}-2>0$)}.
			\end{aligned}
		\end{equation}

		\textbf{Then \eqref{expansion_D1st} follows from \eqref{D1st_1} and \eqref{D1st_2}, and
			\eqref{expansion_D2rd} follows from \eqref{D2rd_1} and \eqref{D2rd_2}.}

		\textbf{(ii)} We calculate 
		\begin{equation}\label{term2}
			\begin{aligned}
				&\quad  t^{1-\frac{n}{2}}  \int_V \left|\nabla
				\left( \operatorname{H}(\frac{d(p,x)}{\sqrt{t}})\xi(x,t)\right)\right|^2 d\mu  	\\
				&=t^{1-\frac{n}{2}}  \int_W\left| \left\{ - \frac{|x|\nabla |x|}{4t\left(1+ \frac{(\alpha-1)|x|^2}{8t}\right)_+} \xi+\nabla\xi\right\}\left(1+ \frac{(\alpha-1)|x|^2}{8t}\right)_+^{\frac{1}{1-\alpha}} \right|^2\operatorname{det}\left(g_{k \ell}(x)\right)^\frac{1}{2} dx^n\\
				&=  t^{-\frac{n}{2}}  \int_W\frac{|x|^2}{16t}\operatorname{H}^{2\alpha}\left(\frac{|x|}{\sqrt{t}}\right)\xi^2 \operatorname{det}\left(g_{k \ell}(x)\right)^\frac{1}{2} dx^n+t^{1-\frac{n}{2}}  \int_W\operatorname{H}^{2}\left(\frac{|x|}{\sqrt{t}}\right)|\nabla \xi|^2\operatorname{det}\left(g_{k \ell}(x)\right)^\frac{1}{2} dx^n \\
				&\quad -\frac{1}{8}t^{-\frac{n}{2}}  \int_W\operatorname{H}^{1+\alpha}\left(\frac{|x|}{\sqrt{t}}\right)\langle \nabla|x|^2,  \nabla\xi^2 \rangle\operatorname{det}\left(g_{k \ell}(x)\right)^\frac{1}{2} dx^n.	
			\end{aligned}
		\end{equation}
		\textbf{Next, we calculate each term of \eqref{term2} separately.}
		
		\textbf{First, we calculate the first term in \eqref{term2}.}
		Applying \eqref{key_lemma_term1} with $m=2$,(\ref{3.1_e0}) and (\ref{3.1_e11}) with $q_1=2$, $q_2=\frac{2\alpha}{1-\alpha}$,   
		we get for either $0<\alpha<1$ or ($1<\alpha\le \frac{n}{n-2}$ and $\frac{2\alpha}{\alpha-1}-\frac{n}{2}-2>0$):
		\begin{equation}\label{expansion_A1st_i}
			\begin{aligned}
				&t^{-\frac{n}{2}}\int_{ W}\frac{|x|^2}{16t}\operatorname{H}^{2\alpha}\left(\frac{|x|}{\sqrt{t}}\right)\xi^2(x,t)\operatorname{det}\left(g_{k \ell}(x)\right)^\frac{1}{2}dx^n\\
				=&t^{-\frac{n}{2}}\left(\int_{T_pM}-\int_{T_pM\backslash W}\right)\frac{|x|^2}{16t}\operatorname{H}^{2\alpha}\left(\frac{|x|}{\sqrt{t}}\right)Q_{2,1}dx^n+t^{-\frac{n}{2}}\int_{ W}\frac{|x|^2}{16t}\operatorname{H}^{2\alpha}\left(\frac{|x|}{\sqrt{t}}\right)\left(Q_{2,2}+P\right)dx^n\\
				\xlongequal{y=\frac{x}{\sqrt{t}}}&
				\int_{T_pM} \frac{|y|^2}{16}\left(1+ \frac{(\alpha-1)|y|^2}{8}\right)_+^{\frac{2\alpha}{1-\alpha}} \xi^2(t^{\frac{1}{2}}y,t)	Q_{2,1}(t^{\frac{1}{2}}y,t)dy^n \\
				=& A_0\left\{1+\frac{1}{n}\left(\frac{|\alpha-1|}{8}\right)^{-1}	\frac{\mathscr{B}(\frac{n}{2}+2,\frac{2\alpha}{1-\alpha}+1)}{\mathscr{B}(\frac{n}{2}+1,\frac{2\alpha}{1-\alpha}+1)}\left(\operatorname{tr}(\textbf{a})-\frac{1}{6}\operatorname{Sc}(p)\right)t+\beta_1 t\right\}+o(t),
			\end{aligned}
		\end{equation}
		where we have used  
		for either $0<\alpha<1$ or ($1<\alpha\le \frac{n}{n-2}$ and $\frac{2\alpha}{\alpha-1}-\frac{n}{2}-2>0$): 	by applying \eqref{estimate_forQ} with $m=2$, 
		\eqref{estimate_Q}   with $k=4$  and  \eqref{reason}  with $k=4$, 
		\begin{equation}
			\begin{aligned}
				&\left| t^{-\frac{n}{2}}\int_{T_pM\backslash W}\frac{|x|^2}{16t}\operatorname{H}^{2\alpha}\left(\frac{|x|}{\sqrt{t}}\right) Q_{2,1} dx^n\right|= o(t),\\
				&\left|t^{-\frac{n}{2}}\int_{ W}\frac{|x|^2}{16t}\operatorname{H}^{2\alpha}\left(\frac{|x|}{\sqrt{t}}\right)\left(Q_{2,2}+P\right) dx^n\right|=o(t).
			\end{aligned}
		\end{equation}

		Moreover, applying (\ref{3.1_e0}), (\ref{3.1_e11})  and (\ref{3.1_e21}) with $q_1=2$, $q_2=\frac{2\alpha}{1-\alpha}$,  \eqref{key_lemma_term1} with $m=2$, 
		we get for either $0<\alpha<1$ or ($1<\alpha\le \frac{n}{n-2}$ and $\frac{2\alpha}{\alpha-1}-\frac{n}{2}-3>0$):
		\begin{equation}\label{expansion_A2rd_i}
			\begin{aligned}
				&t^{-\frac{n}{2}}\int_{ W}\frac{|x|^2}{16t}\operatorname{H}^{2\alpha}\left(\frac{|x|}{\sqrt{t}}\right)\xi^2(x,t)\operatorname{det}\left(g_{k \ell}(x)\right)^\frac{1}{2}dx^n\\
				=&t^{-\frac{n}{2}}\left(\int_{T_pM}-\int_{T_pM\backslash W}\right)\frac{|x|^2}{16t}\operatorname{H}^{2\alpha}\left(\frac{|x|}{\sqrt{t}}\right)\left(Q_{2,1}+Q_{2,2}\right)dx^n+t^{-\frac{n}{2}}\int_{ W}\frac{|x|^2}{16t}\operatorname{H}^{2\alpha}\left(\frac{|x|}{\sqrt{t}}\right)Pdx^n\\
				\xlongequal{y=\frac{x}{\sqrt{t}}}&
				\int_{T_pM} \frac{|y|^2}{16}\left(1+ \frac{(\alpha-1)|y|^2}{8}\right)_+^{\frac{2\alpha}{1-\alpha}} \xi^2(t^{\frac{1}{2}}y,t)	\left(Q_{2,1}+Q_{2,2}\right)(t^{\frac{1}{2}}y,t)dy^n +o(t^2)\\
				=& A_0\left\{1+\frac{1}{n}\left(\frac{|\alpha-1|}{8}\right)^{-1}	\frac{\mathscr{B}(\frac{n}{2}+2,\frac{2\alpha}{1-\alpha}+1)}{\mathscr{B}(\frac{n}{2}+1,\frac{2\alpha}{1-\alpha}+1)}\left(\operatorname{tr}(\textbf{a})-\frac{1}{6}\operatorname{Sc}(p)\right)t+\beta_1 t\right.\\
				&	+
				\frac{1}{n(n+2)}	\left(\frac{|\alpha-1|}{8}\right)^{-2}\frac{\mathscr{B}(\frac{n}{2}+3,\frac{2\alpha}{1-\alpha}+1)}{\mathscr{B}(\frac{n}{2}+1,\frac{2\alpha}{1-\alpha}+1)}\left(\operatorname{E}(b)+\operatorname{E}(v)-\frac{1}{6}\operatorname{E}(\textbf{a} \otimes \operatorname{Rc})\right)t^2\\
				&\left.+\frac{1}{n}\left(\frac{|\alpha-1|}{8}\right)^{-1}	\frac{\mathscr{B}(\frac{n}{2}+2,\frac{2\alpha}{1-\alpha}+1)}{\mathscr{B}(\frac{n}{2}+1,\frac{2\alpha}{1-\alpha}+1)}\left(\operatorname{tr}(d)-\frac{1}{6}\beta_1\operatorname{Sc}(p)\right)t^2+\beta_2t^2\right\}+o(t^2),
			\end{aligned}
		\end{equation}
		where we have used 
		for either $0<\alpha<1$ or ($1<\alpha\le \frac{n}{n-2}$ and $\frac{2\alpha}{\alpha-1}-\frac{n}{2}-3>0$): by applying \eqref{estimate_forQ} with $m=2$, 
		\eqref{estimate_Q}   with $k=6$  and  \eqref{reason}  with $k=6$,
		\begin{equation}
			\begin{aligned}
				&\left| t^{-\frac{n}{2}}\int_{T_pM\backslash W}\frac{|x|^2}{16t}\operatorname{H}^{2\alpha}\left(\frac{|x|}{\sqrt{t}}\right) \left(Q_{2,1}+Q_{2,2}\right) dx^n\right|= o(t^2),\\
				&\left|t^{-\frac{n}{2}}\int_{ W}\frac{|x|^2}{16t}\operatorname{H}^{2\alpha}\left(\frac{|x|}{\sqrt{t}}\right)P dx^n\right|=o(t^2).
			\end{aligned}
		\end{equation}

		\textbf{Second, we calculate the second term in \eqref{term2}.}
		We calculate that
		\begin{equation}\label{second_term_nabla}
			\begin{aligned}
				&\quad |\nabla \xi|^2\operatorname{det}\left(g_{k \ell}(x)\right)^\frac{1}{2}= \frac{|\nabla \xi^2|^2}{4\xi^2}\operatorname{det}\left(g_{k \ell}(x)\right)^\frac{1}{2}\\
				&=\frac{|\nabla \phi_0|^2+tO(d)+O(t^2)}{4\xi^2}\operatorname{det}\left(g_{k \ell}(x)\right)^\frac{1}{2}\\
				&=\frac{\sum\limits_{k=1}| \frac{\partial}{\partial x^k} (\textbf{a}_{ij}x^ix^j)|^2+o(d^2)+tO(d)+O(t^2)}{4\xi^2}\operatorname{det}\left(g_{k \ell}(x)\right)^\frac{1}{2}\\
				&=\sum_{i=1} \textbf{a}_{ij}\textbf{a}_{ik}x^jx^k+o(d^2)+tO(d)+O(t^2).
			\end{aligned}
		\end{equation}
		Hence, we can write $|\nabla \xi|^2\operatorname{det}\left(g_{k \ell}(x)\right)^\frac{1}{2}=o(|x|^s)+O(t)$ on $W$ for all $0<s\le 1$. By \eqref{reason}, we have
		for either $0<\alpha<1$ or ($1<\alpha\le \frac{n}{n-2}$ and $\frac{2}{\alpha-1}-\frac{n}{2}=\frac{2\alpha}{\alpha-1}-\frac{n}{2}-2>0$): take $s$ small enough such that $\frac{2}{\alpha-1}-\frac{n+s}{2}>0$ if $\alpha>1$ and  $0<s\le 1$ if $0<\alpha<1$ ,
		\begin{equation}\label{expansion_A1st_ii}
			\begin{aligned}
				&t^{1-\frac{n}{2}}\int_V\operatorname{H}^{2}\left(\frac{|x|}{\sqrt{t}}\right)|\nabla \xi|^2d\mu\\
				=&t^{1-\frac{n}{2}}\int_{W}\operatorname{H}^{2}\left(\frac{|x|}{\sqrt{t}}\right)\left(o(|x|^s)+O(t)\right)dx^n\\
				\xlongequal{y=\frac{x}{\sqrt{t}}}& t\int_{W}\left(1+ \frac{(\alpha-1)|y|^2}{8}\right)_+^{\frac{2}{1-\alpha}}\left(\frac{o(t^{\frac{s}{2}}|y|^s)}{t^{\frac{s}{2}}|y|^s}t^{\frac{s}{2}}|y|^s+O(t)\right)dy^n\\
				=& o(t).
			\end{aligned}
		\end{equation}
		Moreover, applying \eqref{second_term_nabla}, (\ref{3.1_e11}) with $q_1=0$ and $q_2=\frac{2}{1-\alpha}$, we get for either $0<\alpha<1$ or ($1<\alpha\le \frac{n}{n-2}$ and $\frac{2}{\alpha-1}-\frac{n}{2}-1=\frac{2\alpha}{\alpha-1}-\frac{n}{2}-3>0$):
		\begin{equation}\label{expansion_A2rd_ii}
			\begin{aligned}
				&t^{1-\frac{n}{2}}\int_V\operatorname{H}^{2}\left(\frac{|x|}{\sqrt{t}}\right)|\nabla \xi|^2d\mu\\
				=&t^{1-\frac{n}{2}}\left(\int_{T_pM}-\int_{T_pM\backslash W}\right)\operatorname{H}^{2}\left(\frac{|x|}{\sqrt{t}}\right)\left(\sum_{i=1} \textbf{a}_{ij}\textbf{a}_{ik}x^jx^k\right)dx^n\\
				&+t^{1-\frac{n}{2}}\int_{T_pM}\operatorname{H}^{2}\left(\frac{|x|}{\sqrt{t}}\right)\left(o(|x|^2)+tO(|x|)+O(t^2)\right)dx^n\\
				\xlongequal{y=\frac{x}{\sqrt{t}}}& t\int \left(1+ \frac{(\alpha-1)|y|^2}{8}\right)_+^{\frac{2}{1-\alpha}}\left(\sum_{i=1} \textbf{a}_{ij}\textbf{a}_{ik}y^jy^kt\right)dy^n+o(t^2) \\
				=& A_0\cdot\frac{16}{n}	\frac{\mathscr{B}(\frac{n}{2}+1,\frac{2}{1-\alpha}+1)}{\mathscr{B}(\frac{n}{2}+1,\frac{2\alpha}{1-\alpha}+1)}\operatorname{tr}(\textbf{a}^2)t^2+o(t^2),
			\end{aligned}
		\end{equation}
		where we have used 
		for either $0<\alpha<1$ or ($1<\alpha\le \frac{n}{n-2}$ and $\frac{2}{\alpha-1}-\frac{n}{2}-1>0$): 		by applying \eqref{estimate_Q} with $k=2$ and $m=2$, and \eqref{reason},
		\begin{equation}
			\begin{aligned}
				&\left| t^{1-\frac{n}{2}}\int_{T_pM\backslash W}\operatorname{H}^{2}\left(\frac{|x|}{\sqrt{t}}\right)\left(\sum_{i=1} \textbf{a}_{ij}\textbf{a}_{ik}x^jx^k\right) dx^n\right|= o(t^2),\\
				&\left|t^{1-\frac{n}{2}}\int_{ W}\operatorname{H}^{2}\left(\frac{|x|}{\sqrt{t}}\right)\left(o(|x|^2)+tO(|x|)+O(t^2)\right) dx^n\right|=o(t^2).
			\end{aligned}
		\end{equation}
		
		\textbf{Third, we calculate the third term in \eqref{term2}.}
		We calculate that 
		\begin{equation*}
			\begin{aligned}
				&\quad\left(\nabla d^2 \cdot \nabla \xi^2\right)\operatorname{det}\left(g_{k \ell}(x)\right)^\frac{1}{2}\\
				&	= g^{rs} \frac{\partial d^2}{\partial x_r} \frac{\partial \xi^2}{\partial x_s} \operatorname{det}\left(g_{k \ell}(x)\right)^\frac{1}{2}\\
				&	= \left(\delta_{rs}+\frac{1}{3}R_{rijs}(p)x^ix^j+O(d^3)\right) \frac{\partial d^2}{\partial x_r} \frac{\partial \xi^2}{\partial x_s} \operatorname{det}\left(g_{k \ell}(x)\right)^\frac{1}{2}\\
				&= 4\textbf{a}_{ij}x^ix^j+T+P,
			\end{aligned}
		\end{equation*}
		where
		$$
		\begin{aligned}
			T := \left(6e_{ijk}x^ix^jx^k+2q_ix^it\right)+\left(4d_{ij}x^ix^j t+8b_{ijkl}x^ix^jx^kx^l-\frac{2}{3}\textbf{a}_{ij}R_{kl}(p)x^ix^jx^kx^l+\frac{4}{3}\sum_r R_{ijkr}(p)\textbf{a}_{rl}x^ix^jx^kx^l\right) ,
		\end{aligned}
		$$
		$$
		\begin{aligned}
			P  := o(|x|^2)t+o(|x|^4)+o(1)t^2.
		\end{aligned}
		$$
		Applying (\ref{3.1_e0})-(\ref{3.1_e21}) with $q_1=0$, $q_2=\frac{1+\alpha}{1-\alpha}=\frac{2\alpha}{1-\alpha}+1$,
		we get for
		either $0<\alpha<1$ or ($1<\alpha\le \frac{n}{n-2}$ and $\frac{2\alpha}{\alpha-1}-\frac{n}{2}-2>0$):
		\begin{equation}\label{expansion_A1st_iii}
			\begin{aligned}
				&\quad -t^{-\frac{n}{2}}\int_W\frac{1}{8}\operatorname{H}^{1+\alpha} \langle \nabla|x|^2,  \nabla\xi^2 \rangle \operatorname{det}\left(g_{k \ell}(x)\right)^\frac{1}{2} dx^n \\
				&= -t^{-\frac{n}{2}}\left(\int_{T_pM}-\int_{T_pM\backslash W}\right)\frac{1}{8}\operatorname{H}^{1+\alpha} 4\textbf{a}_{ij}x^ix^jdx^n-t^{-\frac{n}{2}}\int_{W}\frac{1}{8}\operatorname{H}^{1+\alpha} \left(T+P\right)dx^n \\
				&\xlongequal{y=\frac{x}{\sqrt{t}}} -\int \frac{1}{8}\left(1+ \frac{(\alpha-1)|y|^2}{8}\right)_+^{\frac{1+\alpha}{1-\alpha}}  4\textbf{a}_{ij}y^iy^jtdy^n +o(t)\\
				&= -A_0\cdot\frac{8}{n}	\frac{\mathscr{B}(\frac{n}{2}+1,\frac{2\alpha}{1-\alpha}+2)}{\mathscr{B}(\frac{n}{2}+1,\frac{2\alpha}{1-\alpha}+1)}\operatorname{tr}(\textbf{a})t+o(t),
			\end{aligned}
		\end{equation}
		where we have used
		for either $0<\alpha<1$ or ($1<\alpha\le \frac{n}{n-2}$ and $\frac{2\alpha}{\alpha-1}-\frac{n}{2}-2>0$): by applying \eqref{estimate_Q} with $k=2$ and $m=1+\alpha$, and
		\eqref{reason},
		\begin{equation}
			\begin{aligned}
				&\left| t^{-\frac{n}{2}}\int_{T_pM\backslash W}\frac{1}{8}\operatorname{H}^{1+\alpha} 4\textbf{a}_{ij}x^ix^jdx^n\right|= o(t),\\
				&\left|t^{-\frac{n}{2}}\int_{ W}\frac{1}{8}\operatorname{H}^{1+\alpha} \left(T+P\right) dx^n\right|=o(t).
			\end{aligned}
		\end{equation}

		Moreover, applying (\ref{3.1_e0}), (\ref{3.1_e11}) and (\ref{3.1_e21}) with $q_1=0$, $q_2=\frac{1+\alpha}{1-\alpha}=\frac{2\alpha}{1-\alpha}+1$,
		we get for either $0<\alpha<1$ or ($1<\alpha\le \frac{n}{n-2}$ and $\frac{2\alpha}{\alpha-1}-\frac{n}{2}-3>0$):
		\begin{equation}\label{expansion_A2rd_iii}
			\begin{aligned}
				&\quad -t^{-\frac{n}{2}}\int_W\frac{1}{8}\operatorname{H}^{1+\alpha} \langle \nabla|x|^2,  \nabla\xi^2 \rangle \operatorname{det}\left(g_{k \ell}(x)\right)^\frac{1}{2} dx^n \\
				&= -t^{-\frac{n}{2}}\left(\int_{T_pM}-\int_{T_pM\backslash W}\right)\frac{1}{8}\operatorname{H}^{1+\alpha} \left( 4\textbf{a}_{ij}x^ix^j+T\right) dx^n-t^{-\frac{n}{2}}\int_{W}\frac{1}{8}\operatorname{H}^{1+\alpha} Pdx^n\\
				&\xlongequal{y=\frac{x}{\sqrt{t}}} -\int \frac{1}{8}\left(1+ \frac{(\alpha-1)|y|^2}{8}\right)_+^{\frac{1+\alpha}{1-\alpha}}  \left( 4\textbf{a}_{ij}y^iy^jt+T(t^{\frac{1}{2}}y,t)\right)dy^n +o(t^2)\\
				&= A_0\left\{-\frac{8}{n}	\frac{\mathscr{B}(\frac{n}{2}+1,\frac{2\alpha}{1-\alpha}+2)}{\mathscr{B}(\frac{n}{2}+1,\frac{2\alpha}{1-\alpha}+1)}\operatorname{tr}(\textbf{a})t-\frac{8}{n}	\frac{\mathscr{B}(\frac{n}{2}+1,\frac{2\alpha}{1-\alpha}+2)}{\mathscr{B}(\frac{n}{2}+1,\frac{2\alpha}{1-\alpha}+1)}\operatorname{tr}(d)t^2\right.\\
				&	\left.\quad +
				\frac{1}{n(n+2)}	\left(\frac{|\alpha-1|}{8}\right)^{-1}\frac{\mathscr{B}(\frac{n}{2}+2,\frac{2\alpha}{1-\alpha}+2)}{\mathscr{B}(\frac{n}{2}+1,\frac{2\alpha}{1-\alpha}+1)}\left(-16\operatorname{E}(b)+\frac{4}{3}\operatorname{E}(\textbf{a} \otimes \operatorname{Rc})\right)t^2\right\}+o(t^2),
			\end{aligned}
		\end{equation}
		where we have used 
		$$
		E(\sum_rR_{ijkr}(p)\textbf{a}_{rl})=\sum\limits_{i, j}\sum_r\left(R_{iijr}(p)\textbf{a}_{rj}+R_{ijir}(p)\textbf{a}_{rj}+R_{ijjr}(p)\textbf{a}_{ri}\right)=0.
		$$
		and
		for either $0<\alpha<1$ or ($1<\alpha\le \frac{n}{n-2}$ and  $\frac{2\alpha}{\alpha-1}-\frac{n}{2}-3>0$): 		by applying \eqref{estimate_Q} with $k=4$ and $m=1+\alpha$, and \eqref{reason},
		\begin{equation}
			\begin{aligned}
				&\left| t^{-\frac{n}{2}}\int_{T_pM\backslash W}\frac{1}{8}\operatorname{H}^{1+\alpha} \left(4\textbf{a}_{ij} +T\right)dx^n\right|= o(t^2),\\
				&\left|t^{-\frac{n}{2}}\int_{ W}\frac{1}{8}\operatorname{H}^{1+\alpha} P dx^n\right|=o(t^2).
			\end{aligned}
		\end{equation}

		\textbf{Then \eqref{expansion_A1st} follows from \eqref{expansion_A1st_i}, \eqref{expansion_A1st_ii} and \eqref{expansion_A1st_iii}. Then \eqref{expansion_A2rd} follows from \eqref{expansion_A2rd_i}, \eqref{expansion_A2rd_ii} and \eqref{expansion_A2rd_iii}.}
		
		(iii)
		Also notice that when $0<\alpha<1$ or ($1<\alpha\le \frac{n}{n-2}$ and $\frac{2\alpha}{\alpha-1}-\frac{n}{2}-3>0$), we have
		\begin{equation}\label{term7}
			\begin{aligned}
				&\quad 	 
				t^{1-\frac{n}{2}}  \int_V \operatorname{Sc}(x)\cdot  \left( \operatorname{H}(\frac{d(p,x)}{\sqrt{t}})\xi(x,t)\right)^2  d\mu 
				\\
				&= 	t^{1-\frac{n}{2}} \int_{W} \left(\operatorname{Sc}(p)+\nabla_i\operatorname{Sc}(p)x^i+\frac{1}{2}\nabla_i\nabla_j \operatorname{Sc}(p) x^i x^j+o(|x|^2)\right)\\
				&\times
				\left(1+\textbf{a}_{ij}x^ix^j+o(|x|^2)+\beta_1 t+o(t)\right)\times (1-\frac{1}{6} R_{ij}(p) x^i x^j+o(|x|^2))\operatorname{H}^2(\frac{|x|}{\sqrt{t}})dx^n	\\
				&= 	t^{1-\frac{n}{2}} \left(\int_{T_pM}- \int_{T_pM\backslash W}\right)\left(\operatorname{Sc}(p)+\frac{1}{2}\nabla_i\nabla_j \operatorname{Sc}(p) x^i x^j+\operatorname{Sc}(p)\textbf{a}_{ij} x^i x^j-\frac{1}{6}\operatorname{Sc}(p) R_{ij}(p)x^i x^j+\beta_1 \operatorname{Sc}(p) t\right)\operatorname{H}^2(\frac{|x|}{\sqrt{t}})dx^n\\
				&\quad+t^{1-\frac{n}{2}} \int_{W} \left(o(|x|^2)+o(t)\right)\operatorname{H}^2(\frac{|x|}{\sqrt{t}})dx^n\\
				&\xlongequal{y=\frac{x}{\sqrt{t}}} 	t\int_{T_pM} \left(\operatorname{Sc}(p)+\frac{1}{2}\nabla_i\nabla_j \operatorname{Sc}(p) y^i y^j t+\operatorname{Sc}(p)\textbf{a}_{ij} y^i y^j t-\frac{1}{6}\operatorname{Sc}(p) R_{ij}(p)y^i y^jt+\beta_1 \operatorname{Sc}(p) t\right)\operatorname{H}^2(|y|)dy^n\\
				& \quad+o(t^2)+t \int_{W} \left(o(t|y|^2)+o(t)\right)\operatorname{H}^2(|y|)dy^n\\
				&=\frac{\omega_{n-1}}{2}\left(\frac{|\alpha-1|}{8}\right)^{-\frac{n}{2}}\mathscr{B}(\frac{n}{2},\frac{2}{1-\alpha}+1)\times\\
				&\Big\{\operatorname{Sc}(p)t+\frac{1}{n}\left(\frac{|\alpha-1|}{8}\right)^{-1}	\frac{\mathscr{B}(\frac{n}{2}+1,\frac{2}{1-\alpha}+1)}{\mathscr{B}(\frac{n}{2},\frac{2}{1-\alpha}+1)}\left(\frac{1}{2}\Delta \operatorname{Sc}(p)-\frac{1}{6}\operatorname{Sc}^2(p)+\operatorname{tr}(\textbf{a})\operatorname{Sc}(p)\right)t^2+\beta_1 \operatorname{Sc}(p)t^2\Big\}+o(t^2),
			\end{aligned}
		\end{equation}
		where we have used for $0<\alpha<1$ or ($1<\alpha\le \frac{n}{n-2}$ and $\frac{2\alpha}{\alpha-1}-\frac{n}{2}-3>0$): applying \eqref{estimate_Q} with $k=2$ and $m=2$, and by \eqref{reason},
		\begin{equation}
			\begin{aligned}
				&\left| t^{1-\frac{n}{2}}\int_{T_pM\backslash W}\left(\operatorname{Sc}(p)+\frac{1}{2}\nabla_i\nabla_j \operatorname{Sc}(p) x^i x^j+\operatorname{Sc}(p)\textbf{a}_{ij} x^i x^j-\frac{1}{6}\operatorname{Sc}(p) R_{ij}(p)x^i x^j+\beta_1 \operatorname{Sc}(p) t\right)\operatorname{H}^2(\frac{|x|}{\sqrt{t}})\right|= o(t^2),\\
				&\left|t^{1-\frac{n}{2}} \int_{W} \left(o(|x|^2)+o(t)\right)\operatorname{H}^2(\frac{|x|}{\sqrt{t}})dx^n\right|=o(t^2).
			\end{aligned}
		\end{equation}
	\end{proof}
	
	\section{ Power series expansion formulas of $\mathcal{L}^{\pm}_{\alpha}$-functionals and Proof of Theorem \ref{const_rigidity}}

	In this section, we derive the power series expansion formulas for the $\mathcal{L}^{\pm}_{\alpha}$-functionals and subsequently provide the proof of Theorem~\ref{const_rigidity}.

	\begin{thm}\label{expansion_L_a}
		Let $(M^n,g)$ be an $n$-dimensional Riemannian manifold and $p\in V\subset \mathring{M^n}$, where $V$ is a neighborhood of $p$.

		\noindent (i)	Denote
		\begin{equation}\label{def_u-}
			v_-(x,t):=(4\pi t)^{-\frac{n}{2(\alpha+1)}}\operatorname{H}\left(\frac{|x|}{\sqrt{t}}\right),\quad u_-:=\frac{v_-}{\left(\int_{T_pM}v_-^{\alpha+1}dx^n\right)^{\frac{1}{\alpha+1}}}=\frac{t^{-\frac{n}{2(\alpha+1)}}\operatorname{H}\left(\frac{|x|}{\sqrt{t}}\right)}{D_0(\alpha+1)^{\frac{1}{\alpha+1}}},
		\end{equation}
		and 
		\begin{equation*}
			\tau_-(t)^{1+2\Gamma_{\alpha}}:=\frac{\mathfrak{m}\Gamma_{\alpha}}{1+\Gamma_{\alpha}}\left(\frac{A_0}{D_0(\alpha+1)^{\frac{2}{\alpha+1}}}\right)^{-1}\frac{D_0(2\alpha)}{D_0(\alpha+1)^{\frac{2\alpha}{\alpha+1}}}t^{1+2\Gamma_{\alpha}}.\footnote{See \eqref{tau_connection} for the reason behind our choice of $\tau_{\pm}(t)$.\label{tau_def}}
		\end{equation*}
		Choose the function $\xi_-(x,t)$ to have the expansion \eqref{expansion_of_xi} such that   $p \in \operatorname{supp}(\xi_-) \subset\subset V$ and $$\int_V (u_-\xi_-)^{\alpha+1}d\mu \equiv 1.$$	
		When $0<\alpha<1$, we have 
		\begin{equation}\label{expansion_-}
			\mathcal{L}^-_{\alpha}(V,g, u_-(x,t)\xi_-(x,t),\tau_-(t))=\mathfrak{m}^{1-\frac{\Gamma_{\alpha}}{2\Gamma_{\alpha}+1}}\Sigma^-_{\alpha} \left(\zeta_1 \operatorname{Sc}(p)t-32\zeta_2\Delta \operatorname{Sc}(p)t^2+32\zeta_2\operatorname{\uppercase\expandafter{\romannumeral2}} t^2+\operatorname{\uppercase\expandafter{\romannumeral3}}^-t^2+o(t^2)\right),
		\end{equation}
		where
		\begin{equation}\label{see_why}
			\begin{aligned}
				&\operatorname{\uppercase\expandafter{\romannumeral2}} :=\chi\left|\textbf{a}-\frac{2(\alpha+1) \operatorname{Rc}(p)}{3\chi}\right|^2+\frac{4((n+5) \alpha-n-3)(\alpha-1) }{9\chi}|\operatorname{Rc}|^2(p)-\frac{1}{6}|\operatorname{Rm}|^2(p),
			\end{aligned}		
		\end{equation}
		\begin{equation}
			\zeta_1:=\frac{8}{n\left(n(\alpha-1)-4\right)},		\zeta_2:=\frac{\zeta_1}{8\left(n(\alpha-1)+2\alpha-6\right)},		\chi:=(n+6) \alpha^2-2(n+3)\alpha+ n+4>0.
		\end{equation}
		and
		$\operatorname{\uppercase\expandafter{\romannumeral3}}^-$ is a constant satisfying
		$
		\operatorname{\uppercase\expandafter{\romannumeral3}}^-=0 \text{ if } \operatorname{Sc}(p)=0 \text{ and } \operatorname{tr}(\textbf{a})=0$ \footnote{ For the proof of Theorem \ref{const_rigidity}, the exact values of $\operatorname{\uppercase\expandafter{\romannumeral3}}^{\pm}$ are not required. However, these precise values  become essential for Theorem \ref{Yamabe_rigidity}. We will compute
			the exact value of $\operatorname{\uppercase\expandafter{\romannumeral3}}^{\pm}$ in Lemma \ref{expansion_W_a}.\label{llama3.2}}. Here, $\textbf{a}$ is the tensor which defined in the expansion  \eqref{expansion_of_xi}.

		\noindent (ii)	 Denote 
		\begin{equation}\label{def_u+}
			v_+(x,t):=(4\pi t )^{-\frac{n}{4\alpha}}\operatorname{H}(\frac{|x|}{\sqrt{t}}),\quad u_+:=\frac{v_+}{(\int_{T_pM}v_+^{2\alpha}dx^n)^{\frac{1}{2\alpha}}}=\frac{t ^{-\frac{n}{4\alpha}}\operatorname{H}(\frac{|x|}{\sqrt{t}})}{D_0(2\alpha)^{\frac{1}{2\alpha}}}, 
		\end{equation}
		and 
		$$\tau_+(t)^{1-\Theta_{\alpha}}:=\frac{\mathfrak{m}\Theta_{\alpha}}{1-2\Theta_{\alpha}}\left(\frac{A_0}{D_0(2\alpha)^{\frac{1}{\alpha}}}\right)^{-1}\frac{D_0(\alpha+1)}{ D_0(2\alpha)^{\frac{\alpha+1}{2\alpha}}}t^{1-\Theta_{\alpha}}.^*$$
		Choose the function $\xi_+(x,t)$ to have the expansion \eqref{expansion_of_xi} such that  such that   $p \in \operatorname{supp}(\xi_+) \subset\subset V$ and $$\int_V (u_+\xi_+)^{2\alpha}d\mu \equiv 1.$$	 \\
		When $1<\alpha\le \frac{n}{n-2}$ and $\alpha<\frac{n+4}{n}$,
		we have 
		\begin{equation}\label{expansion_+_1}
			\mathcal{L}^+_{\alpha}(V,g, u_+(x,t)\xi_+(x,t),\tau_+(t))=\mathfrak{m}^{1-\frac{\Theta_{\alpha}}{1-\Theta_{\alpha}}}\Sigma^+_{\alpha}   \left(\zeta_1 \operatorname{Sc}(p)t+o(t)\right)
		\end{equation}
		Moreover, when $1<\alpha\le \frac{n}{n-2}$ and $\alpha<\frac{n+6}{n+2}$,
		we have 
		\begin{equation}\label{expansion_+_2}
			\mathcal{L}^+_{\alpha}(V,g, u_+(x,t)\xi_+(x,t),\tau_+(t))=\mathfrak{m}^{1-\frac{\Theta_{\alpha}}{1-\Theta_{\alpha}}}\Sigma^+_{\alpha}   \left(\zeta_1 \operatorname{Sc}(p)t-32\zeta_2\Delta \operatorname{Sc}(p)t^2+32\zeta_2\operatorname{\uppercase\expandafter{\romannumeral2}} t^2+\operatorname{\uppercase\expandafter{\romannumeral3}}^+t^2+o(t^2)\right),
		\end{equation}
		and
		$\operatorname{\uppercase\expandafter{\romannumeral3}}^+$ is a constant satisfying
		$
		\operatorname{\uppercase\expandafter{\romannumeral3}}^+=0 \text{ if } \operatorname{Sc}(p)=0 \text{ and } \operatorname{tr}(\textbf{a})=0\textsuperscript{\ref{llama3.2}}.
		$

	\end{thm}
	
	\begin{rem}\label{Rk_choose_of_a}
		If we choose 
		\[
		\mathbf{a} = \frac{2(\alpha+1)}{3\chi} \operatorname{Rc}(p),
		\]
		then the coefficient of the $|\operatorname{Rc}|^2(p)$ term is minimized in expansions \eqref{expansion_-} and \eqref{expansion_+_2}.
		Here and below, we denote by $\mathscr{B}_p(V)$ the set of all such functions:
		\begin{equation}\label{B_def}
			\mathscr{B}_p(V) = \biggl\{\eta(x,t) \biggm| 
			\begin{aligned}
				&\eta \text{ have the expansion \eqref{expansion_of_xi} such that  $p \in \operatorname{supp}(\eta) \subset\subset V$ with $\textbf{a}=\frac{2(\alpha+1)}{3\chi} \operatorname{Rc}(p)$,  } \\
				&\text{  $\int_V (u_-\eta)^{\alpha+1}d\mu \equiv 1$ if $0<\alpha<1$ and $\int_V (u_+\eta)^{2\alpha}d\mu \equiv 1$ if  $1<\alpha\le \frac{n}{n-2}$.} 
			\end{aligned}
			\biggr\}
		\end{equation}
	\end{rem}
	\begin{proof}[Proof of Theorem \ref{expansion_L_a}] \textbf{(i)}
		By Lemma \ref{key_lemma},  we get for
		$0<\alpha<1$: 
		\begin{equation}\label{expan_c_i1}
			\begin{aligned}
				&\quad\mathcal{L}^-_{\alpha}(V,g, u_-(x,t)\xi_-(x,t),\tau_-(t))\\
				&=\left(\frac{\mathfrak{m}\Gamma_{\alpha}}{1+\Gamma_{\alpha}}\left(\frac{A_0}{D_0(\alpha+1)^{\frac{2}{\alpha+1}}}\right)^{-1}\frac{D_0(2\alpha)}{D_0(\alpha+1)^{\frac{2\alpha}{\alpha+1}}}\right)^{\frac{1+\Gamma_{\alpha}}{1+2\Gamma_{\alpha}}}	\frac{ A_0}{ D_0(\alpha+1)^{\frac{2}{\alpha+1}}} \left(A_0^{-1} t^{1-\frac{n}{2}}  \int_V \left|\nabla
				\left( \operatorname{H}(\frac{d(p,x)}{\sqrt{t}})\xi_-(x,t)\right)\right|^2 d\mu \right)
				\\
				&\quad +\mathfrak{m}\left(\frac{\mathfrak{m}\Gamma_{\alpha}}{1+\Gamma_{\alpha}}\left(\frac{A_0}{D_0(\alpha+1)^{\frac{2}{\alpha+1}}}\right)^{-1}\frac{D_0(2\alpha)}{D_0(\alpha+1)^{\frac{2\alpha}{\alpha+1}}}\right)^{-\frac{\Gamma_{\alpha}}{1+2\Gamma_{\alpha}}}
				\frac{D_0(2\alpha)}{ D_0(\alpha+1)^{\frac{2\alpha}{\alpha+1}}}
				\left(D_0(2\alpha)^{-1}t^{-\frac{n}{2}} 
				\int_V 
				\left( \operatorname{H}(\frac{d(p,x)}{\sqrt{t}})\xi_-(x,t)\right)^{2\alpha}\right)\\
				&\quad -\frac{2\Gamma_{\alpha}+1}{\Gamma_{\alpha}}\mathfrak{m}^{1-\frac{\Gamma_{\alpha}}{2\Gamma_{\alpha}+1}}\Sigma^-_{\alpha}+o(t^2)\\
				&=\mathfrak{m}^{1-\frac{\Gamma_{\alpha}}{2\Gamma_{\alpha}+1}}\Sigma^-_{\alpha}\left(1+\frac{A_1}{A_0}t+\frac{A_2}{A_0}t^2+o(t^2)\right)
				+\frac{\Gamma_{\alpha}+1}{\Gamma_{\alpha}}m^{1-\frac{\Gamma_{\alpha}}{2\Gamma_{\alpha}+1}}\Sigma^-_{\alpha}\left(1+\frac{D_1(2\alpha)}{D_0(2\alpha)}t+\frac{D_2(2\alpha)}{D_0(2\alpha)}t^2+o(t^2)\right) \\
				&\quad -\frac{2\Gamma_{\alpha}+1}{\Gamma_{\alpha}}\mathfrak{m}^{1-\frac{\Gamma_{\alpha}}{2\Gamma_{\alpha}+1}}\Sigma^-_{\alpha}+o(t^2)\\
				&=\mathfrak{m}^{1-\frac{\Gamma_{\alpha}}{2\Gamma_{\alpha}+1}}\Sigma^-_{\alpha}\left\{\left(\frac{A_1}{A_0}+\frac{\Gamma_{\alpha}+1}{\Gamma_{\alpha}}\frac{D_1(2\alpha)}{D_0(2\alpha)}\right)t+\left(\frac{A_2}{A_0}+\frac{\Gamma_{\alpha}+1}{\Gamma_{\alpha}}\frac{D_2(2\alpha)}{D_0(2\alpha)}\right)t^2+o(t^2)\right\}\\		&=\mathfrak{m}^{1-\frac{\Gamma_{\alpha}}{2\Gamma_{\alpha}+1}}\Sigma^-_{\alpha} \left\{\left(\frac{A_1}{A_0}+\frac{(1-\gamma)}{\gamma\alpha}\frac{D_1(2\alpha)}{D_0(2\alpha)}-\frac{2}{\gamma(\alpha+1)}\frac{D_1(\alpha+1)}{D_0(\alpha+1)}\right)t\right.\\
				&\left.\quad\quad+\left(\frac{A_2}{A_0}+\frac{(1-\gamma)}{\gamma\alpha}\frac{D_2(2\alpha)}{D_0(2\alpha)}-\frac{2}{\gamma(\alpha+1)}\frac{D_2(\alpha+1)}{D_0(\alpha+1)}\right)t^2+o(t^2)\right\},
			\end{aligned}		
		\end{equation}
		where we have used  
		$$
		\mathcal{N}_{\alpha,n} = \frac{A_0^{\frac{\gamma}{2}}D_0(2\alpha)^{\frac{1-\gamma}{2\alpha}}}{D_0(\alpha+1)^{\frac{1}{\alpha+1}}}, \quad 
		\Sigma^-_{\alpha} = \left(\frac{\Gamma_{\alpha}}{\Gamma_{\alpha}+1}\right)^{1-\frac{\Gamma_{\alpha}}{2\Gamma_{\alpha}+1}} \mathcal{N}_{\alpha,n}^{\frac{2}{\gamma}\cdot\frac{\Gamma_{\alpha}}{2\Gamma_{\alpha}+1}}, \quad
		\frac{\Gamma_{\alpha}+1}{\Gamma_{\alpha}} = \frac{1-\gamma}{\gamma \alpha},
		$$
		and by 
		$
		\int_V (u_- \xi_-)^{\alpha+1} d\mu \equiv 1
		$
		and Lemma \ref{key_lemma}, 
		\begin{equation}\label{normal_1}
			\frac{D_1(\alpha+1)}{D_0(\alpha+1)} = 0 \quad \text{and} \quad \frac{D_2(\alpha+1)}{D_0(\alpha+1)} = 0.
		\end{equation}
		
		Observing from Lemma \ref{key_lemma},
		for $0<\alpha<1$, we can write that
		\begin{equation}\label{keylemma:1st}
			\begin{aligned}
				\frac{A_1}{A_0}+\frac{1-\gamma}{\gamma\alpha}\frac{D_1(2\alpha)}{D_0(2\alpha)}-\frac{2}{\gamma(\alpha+1)}\frac{D_1(\alpha+1)}{D_0(\alpha+1)}:=c_1^{-}\operatorname{Sc}(p)+l_1^{-}\beta_1+l_2^{-}\operatorname{tr}(\textbf{a}),
			\end{aligned}		
		\end{equation}
		and
		\begin{equation}\label{ot2_term}
			\begin{aligned}
				&\frac{A_2}{A_0}+\frac{1-\gamma}{\gamma\alpha}\frac{D_2(2\alpha)}{D_0(2\alpha)}-\frac{2}{\gamma(\alpha+1)}\frac{D_2(\alpha+1)}{D_0(\alpha+1)}\\
				:=& c_2^{-}\operatorname{tr}(\textbf{a}^2)+c_3^{-}\operatorname{E}(v)+c_4^{-}\operatorname{E}(\textbf{a} \otimes \textbf{a})+c_5^{-}\operatorname{E}(\textbf{a} \otimes \operatorname{Rc})\\
				&+c_6^{-}\beta_1\operatorname{tr}(\textbf{a})+c_7^{-}\beta_1^2+c_8^{-}\beta_1 \operatorname{Sc}(p)+k_1^{-}\beta_2+k_2^{-}\operatorname{E}(b)+k_3^{-}\operatorname{tr}(d),
			\end{aligned}		
		\end{equation}
		Here, the constants  $c_i^{-}$, $l_i^{-}$ and $k_i^{-}$, depend only on $n$ and $\alpha$.

		\textbf{We next calculate $c_1^{-}$, $c_2^{-}$, $c_3^{-}$, $c_4^{-}$ and $c_5^{-}$ term by term using the Lemma \ref{key_lemma}.}

		Using Lemma \ref{key_lemma} and \eqref{keylemma:1st}, we get
		\begin{equation}\label{c_1_-}
			\begin{aligned}
				c_1^-=& \frac{1}{6n}\left(\frac{|\alpha-1|}{8}\right)^{-1}\left\{-	\frac{\mathrm{~B}(\frac{n}{2}+2,\frac{2\alpha}{1-\alpha}+1)}{\mathrm{~B}(\frac{n}{2}+1,\frac{2\alpha}{1-\alpha}+1)}-\frac{1-\gamma}{\gamma\alpha}\frac{\mathrm{~B}(\frac{n}{2}+1,\frac{2\alpha}{1-\alpha}+1)}{\mathrm{~B}(\frac{n}{2},\frac{2\alpha}{1-\alpha}+1)}+\frac{2}{\gamma(\alpha+1)}\frac{\mathrm{~B}(\frac{n}{2}+1,\frac{1+\alpha}{1-\alpha}+1)}{\mathrm{~B}(\frac{n}{2},\frac{1+\alpha}{1-\alpha}+1)}\right\}\\
				=&\frac{1}{6n}\left(\frac{1-\alpha}{8}\right)^{-1}\left(-\frac{\frac{n}{2}+1}{\frac{2\alpha}{1-\alpha}+\frac{n}{2}+2}-\frac{1-\gamma}{\gamma\alpha}\frac{\frac{n}{2}}{\frac{2\alpha}{1-\alpha}+\frac{n}{2}+1}+\frac{2}{\gamma(\alpha+1)}\frac{\frac{n}{2}}{\frac{1+\alpha}{1-\alpha}+\frac{n}{2}+1}\right)\\
				=&\zeta_1,
			\end{aligned}   
		\end{equation}
		where we have used 
		\begin{equation}\label{thm:beta_ii_1}
			\frac{\mathrm{~B}(p+1,q)}{\mathrm{~B}(p,q)}=\frac{p}{p+q} \text{ when $p>0$ and $q>0$.}
		\end{equation}
		The final simplification step is computed using the mathematical software Maple.
		
		Using Lemma \ref{key_lemma} and \eqref{ot2_term}, we get
		\begin{equation}\label{c_2_-}
			\begin{aligned}
				c_2^-=&\frac{16}{n}	\frac{\mathrm{~B}(\frac{n}{2}+1,\frac{2\alpha}{1-\alpha}+3)}{\mathrm{~B}(\frac{n}{2}+1,\frac{2\alpha}{1-\alpha}+1)}=128(\alpha+1)\zeta_2,\\
				c_3^-=& \frac{1}{n(n+2)}\left(\frac{|\alpha-1|}{8}\right)^{-2}\left\{
				\frac{\mathrm{~B}(\frac{n}{2}+3,\frac{2\alpha}{1-\alpha}+1)}{\mathrm{~B}(\frac{n}{2}+1,\frac{\alpha\gamma}{1-\alpha}+1)}
				+\frac{1-\gamma}{2\alpha}\frac{\mathrm{~B}(\frac{n}{2}+2,\frac{2\alpha}{1-\alpha}+1)}{\mathrm{~B}(\frac{n}{2},\frac{2\alpha}{1-\alpha}+1)}\right.\\
				&\left.-\frac{2}{\gamma(1+\alpha)}\frac{\mathrm{~B}(\frac{n}{2}+2,\frac{1+\alpha}{1-\alpha}
					+1)}{\mathrm{~B}(\frac{n}{2},\frac{1+\alpha}{1-\alpha}
					+1)}  \right\}\\
				=&640\zeta_2,\\
				c_4^-=& \frac{1}{n(n+2)}\left(\frac{|\alpha-1|}{8}\right)^{-2}\left\{\frac{(1-\gamma)(2\alpha-2)}{4\gamma}\frac{\mathrm{~B}(\frac{n}{2}+2,\frac{2\alpha}{1-\alpha}+1)}{\mathrm{~B}(\frac{n}{2},\frac{2\alpha}{1-\alpha}+1)}-\frac{(\alpha-1)}{4\gamma}\frac{\mathrm{~B}(\frac{n}{2}+2,\frac{1+\alpha}{1-\alpha}+1)}{\mathrm{~B}(\frac{n}{2},\frac{1+\alpha}{1-\alpha}+1)}\right\}\\
				=&\left(96\alpha^2-160\alpha+16n(\alpha-1)^2\right)\zeta_2,\\
				c_5^-=& \frac{1}{n(n+2)}\left(\frac{|\alpha-1|}{8}\right)^{-2}\left\{-\frac{1}{6}\frac{\mathrm{~B}(\frac{n}{2}+3,\frac{2\alpha}{1-\alpha}+1)}{\mathrm{~B}(\frac{n}{2}+1,\frac{2\alpha}{1-\alpha}+1)}-\frac{1-\gamma}{\gamma\alpha}\frac{2\alpha}{12} \frac{\mathrm{~B}(\frac{n}{2}+2,\frac{2\alpha}{1-\alpha}+1)}{\mathrm{~B}(\frac{n}{2},\frac{2\alpha}{1-\alpha}+1)}
				\right.\\
				&\left.+	\frac{2}{\gamma(1+\alpha)}\frac{\alpha+1}{12}\frac{\mathrm{~B}(\frac{n}{2}+2,\frac{1+\alpha}{1-\alpha}+1)}{\mathrm{~B}(\frac{n}{2},\frac{1+\alpha}{1-\alpha}+1)} \right\}+\frac{1}{n(n+2)}\left(\frac{|\alpha-1|}{8}\right)^{-1}\frac{4}{3}\frac{\mathrm{~B}(\frac{n}{2}+2,\frac{2\alpha}{1-\alpha}+2)}{\mathrm{~B}(\frac{n}{2}+1,\frac{2\alpha}{1-\alpha}+1)}\\
				=&-\frac{64(\alpha+1)}{3 }\zeta_2.
			\end{aligned}		
		\end{equation}
		where we have used \eqref{thm:beta_eq_1},
		\begin{equation}\label{thm:beta_eq_1}
			\frac{\mathrm{~B}(p,q+2)}{\mathrm{~B}(p,q)}=\frac{q+1}{p+q+1}\cdot\frac{q}{p+q} \text{ when $p>0$ and $q>0$.}
		\end{equation}
		and $\mathrm{~B}(p,q)=\mathrm{~B}(q,p)$. The final simplification steps are computed using the mathematical software Maple.
		
		\textbf{Next, we calculate $l_i^{-}$ and $k_i^{-}$. }
		First, we have
		$$l_1^{-}=k_1^-=1+\frac{1-\gamma}{\gamma\alpha}\times \frac{2\alpha}{2}-\frac{2}{\gamma(\alpha+1)}\times \frac{\alpha+1}{2}=0.$$
		We can use the similar calculations as $c_i^-$, $l_1^{-}$ and $k_1^{-}$ to get $l_2^-=k_2^-=k_3^-=0$. Another way to see why $l_2^-=k_2^-=k_3^-=0$ is the following: By the Gagliardo-Nirenberg inequality of the Euclidean space, we know that
		$\mathcal{L}^-_{\alpha}(V,g, u,t)\ge 0$  for all $u$ in $ W_0^{1,2}(V)$ if $\operatorname{\operatorname{Rm}}\equiv 0$ on $V$. Then we find that when $\operatorname{\operatorname{Rm}}\equiv 0$ on $V$,
		\begin{equation}\label{1st_term_0}
			\quad\mathcal{L}^-_{\alpha}(V,g, u,t)=l_1^{-}\beta_1t+l_2^-\operatorname{tr}(\textbf{a})t+o(t)\ge 0,
		\end{equation}
		for any $\textbf{a}$. Although
		$\operatorname{tr}(\textbf{a})$, $\beta_1$ and $\operatorname{Sc}(p)$ should safisfy the following relation by  \eqref{normal_1}:
		\begin{equation}\label{normal_1_condition}
			\begin{aligned}
				&\frac{D_1(\alpha+1)}{D_0(\alpha+1)}=\frac{\alpha+1}{2n}\left(\frac{|\alpha-1|}{8}\right)^{-1}\frac{\mathrm{B}(\frac{n}{2}+1,\frac{1+\alpha}{1-\alpha}+1)}{\mathrm{B}(\frac{n}{2},\frac{1+\alpha}{1-\alpha}+1)}\operatorname{tr}(\textbf{a})+\frac{\alpha+1}{2}\beta_1\\
				&-\frac{1}{6n}
				\left(\frac{|\alpha-1|}{8}\right)^{-1}\frac{\mathrm{B}(\frac{n}{2}+1,\frac{1+\alpha}{1-\alpha}+1)}{\mathrm{B}(\frac{n}{2},\frac{1+\alpha}{1-\alpha}+1)}\operatorname{Sc}(p)=0,
			\end{aligned}
		\end{equation}
		we conclude from \eqref{1st_term_0} that $l_2^-=0$  since  $l_1^{-}=0$ in \eqref{1st_term_0}. For the same reasons, we can also
		get  $k_2^-=k_3^-=0$.

		Since
		$
		E(\textbf{a} \otimes \operatorname{Rc})=\operatorname{tr}(\textbf{a})\operatorname{Sc}(p)+2\sum\limits_{i j=1}^n \textbf{a}_{ij}R_{ij}(p),
		$
		$
		E(\textbf{a} \otimes \textbf{a})=\left(\operatorname{tr}(\textbf{a})\right)^2+2 \operatorname{tr} (a^2)
		$
		and
		$$
		\begin{gathered}
			E(v) =\frac{1}{24} \sum_{i j=1}^n\left\{-\frac{3}{5} \nabla_{i i} R_{j j}-\frac{6}{5} \nabla_{i j} R_{i j}+\frac{1}{3} R_{i i}R_{j j}\right. \\ \left.\quad+\frac{2}{3} R_{i j}^2-\frac{2}{15} \sum_{s, t=1}^n\left(R_{i s i t} R_{j s j t}+R_{i s j t}^2+R_{i s j t} R_{i t j s}\right)\right\} \\ =\frac{1}{360}\left(5 \operatorname{Sc}^2+8|\operatorname{Rc}|^2-3|\operatorname{Rm}|^2-18 \Delta\operatorname{Sc}\right)(p) 
		\end{gathered}
		$$
		(c.f. P197 in \cite{G}), we have 
		$$
		\begin{aligned}
			&\frac{A_2}{A_0}+\frac{1-\gamma}{\alpha\gamma}\frac{D_2(2\alpha)}{D_0(2\alpha)}-\frac{2}{\gamma(1+\alpha)}\frac{D_2(\alpha+1)}{D_0(\alpha+1)}\\
			=&c_2^-\operatorname{tr}(\textbf{a}^2)+c_3^-\operatorname{E}(v)+c_4^-\operatorname{E}(\textbf{a} \otimes \textbf{a})+c_5^-\operatorname{E}(\textbf{a} \otimes \operatorname{Rc})+c_6^-\beta_1\operatorname{tr}(\textbf{a})+c_7^-\beta_1^2+c_8^-\beta_1 \operatorname{Sc}(p)\\
			=&c_2^-\operatorname{tr}(\textbf{a}^2)+\frac{c_3^-}{360}\left(5 \operatorname{Sc}^2+8|\operatorname{Rc}|^2-3|\operatorname{Rm}|^2-18 \Delta\operatorname{Sc}\right)(p)+c_4^-\left((\operatorname{tr}(\textbf{a}))^2+2\operatorname{tr}(\textbf{a}^2)\right)\\
			&+c_5^-\left(\operatorname{tr}(\textbf{a})\operatorname{Sc}(p)+2 \sum\limits_{i j=1}^n \textbf{a}_{ij}R_{ij}(p)\right)+c_6^-\beta_1\operatorname{tr}(\textbf{a})+c_7^-\beta_1^2+c_8^-\beta_1 \operatorname{Sc}(p)\\
			= &-32\zeta_2\Delta \operatorname{Sc}(p)+32\zeta_2\operatorname{\uppercase\expandafter{\romannumeral2}}+\operatorname{\uppercase\expandafter{\romannumeral3}}^-,\\
		\end{aligned}
		$$	
		where
		$$
		\begin{aligned}
			\operatorname{\uppercase\expandafter{\romannumeral2}}:=&\frac{1}{32\zeta_2}\Big\{\left(c_2^-+2c_4^-\right)\operatorname{tr}(\textbf{a}^2)+\frac{c_3^-}{360}\left(8|\operatorname{Rc}|^2-3|\operatorname{Rm}|^2\right)(p)+2c_5^-\sum\limits_{i j=1}^n \textbf{a}_{ij}R_{ij}(p)\Big\}\\
			=&\left(\chi\operatorname{tr}(\textbf{a}^2)-\frac{4(\alpha+1) }{3}\sum\limits_{i j=1}^n\textbf{a}_{ij}R_{ij}(p)+\frac{4}{9}|\operatorname{Rc}|^2(p)\right)-\frac{1}{6}|\operatorname{Rm}|^2(p)\\
			=& \chi\left|\textbf{a}-\frac{2(\alpha+1) \operatorname{Rc}(p)}{3\chi}\right|^2+\frac{4((n+5) \alpha-n-3)(\alpha-1) }{9\chi}|\operatorname{Rc}|^2(p)-\frac{1}{6}|\operatorname{Rm}|^2(p)\\
		\end{aligned}
		$$	
		and by \eqref{normal_1_condition},
		\begin{equation}\label{III_-}
			\begin{aligned}
				\operatorname{\uppercase\expandafter{\romannumeral3}}^-:=&\frac{c_3^-}{72}\operatorname{Sc}^2(p)+c_4^-(\operatorname{tr}(\textbf{a}))^2 +c_5^-\operatorname{tr}(\textbf{a})\operatorname{Sc}(p)+c_6^-\beta_1\operatorname{tr}(\textbf{a})+c_7^-\beta_1^2+c_8^-\beta_1 \operatorname{Sc}(p)\\
				=& C_1^-\operatorname{Sc}^2(p)+C_2^-\operatorname{tr}(\textbf{a})\operatorname{Sc}(p)+C_3^-\operatorname{tr}(\textbf{a})^2.
			\end{aligned}
		\end{equation}
		Here $C_1^-$, $C_2^-$ and $C_3^-$ are the constants depending on
		$\alpha,n$. Hence, $
		\operatorname{\uppercase\expandafter{\romannumeral3}}^-=0 \text{ if } \operatorname{Sc}(p)=0 \text{ and } \operatorname{tr}(\textbf{a})=0$.

		\textbf{(ii)}
		By \eqref{expansion_D2rd} and \eqref{expansion_A2rd},
		we get that for $1<\alpha\le \frac{n}{n-2}$ and $\alpha<\frac{n+6}{n+2}$:
		\begin{equation}\label{expan_c_i2}
			\begin{aligned}
				&\quad\mathcal{L}^+_{\alpha}(V,g, u_+(x,t)\xi_+(x,t),\tau_+(t))\\
				&=\left(\frac{\mathfrak{m}\Theta_{\alpha}}{1-2\Theta_{\alpha}}\left(\frac{A_0}{D_0(2\alpha)^{\frac{1}{\alpha}}}\right)^{-1}\frac{D_0(\alpha+1)}{ D_0(2\alpha)^{\frac{\alpha+1}{2\alpha}}}\right)^{\frac{1-2\Theta_{\alpha}}{1-\Theta_{\alpha}}}	
				\frac{A_0}{D_0(2\alpha)^{\frac{1}{\alpha}}}
				\left(A_0^{-1}  t^{1-\frac{n}{2}} \int_V \left|\nabla
				\left( \operatorname{H}(\frac{d(p,x)}{\sqrt{t}})\xi_+(x,t)\right)\right|^2 d\mu \right)
				\\
				&\quad +\mathfrak{m}\left(\frac{\mathfrak{m}\Theta_{\alpha}}{1-2\Theta_{\alpha}}\left(\frac{A_0}{D_0(2\alpha)^{\frac{1}{\alpha}}}\right)^{-1}\frac{D_0(\alpha+1)}{ D_0(2\alpha)^{\frac{\alpha+1}{2\alpha}}}\right)^{-\frac{\Theta_{\alpha}}{1-\Theta_{\alpha}}}	
				\frac{D_0(\alpha+1)}{ D_0(2\alpha)^{\frac{\alpha+1}{2\alpha}}}
				\left(D_0(\alpha+1)^{-1} t^{-\frac{n}{2}}  \int_V 
				\left( \operatorname{H}(\frac{d(p,x)}{\sqrt{t}})\xi_+(x,t)\right)^{\alpha+1}\right)\\
				&\quad -\frac{1-\Theta_{\alpha}}{\Theta_{\alpha}}\mathfrak{m}^{1-\frac{\Theta_{\alpha}}{1-\Theta_{\alpha}}}\Sigma^+_{\alpha}+o(t^2)\\
				&=\mathfrak{m}^{1-\frac{\Theta_{\alpha}}{1-\Theta_{\alpha}}}\Sigma^+_{\alpha}\left(1+\frac{A_1}{A_0}t+\frac{A_2}{A_0}t^2+o(t^2)\right)+\mathfrak{m}^{1-\frac{\Theta_{\alpha}}{1-\Theta_{\alpha}}}\frac{1-2\Theta_{\alpha}}{\Theta_{\alpha}}\Sigma^+_{\alpha}\left(1+\frac{D_1(\alpha+1)}{D_0(\alpha+1)}t+\frac{D_2(\alpha+1)}{D_0(\alpha+1)}t^2+o(t^2)\right) \\
				&\quad-\frac{1-\Theta_{\alpha}}{\Theta_{\alpha}}\mathfrak{m}^{1-\frac{\Theta_{\alpha}}{1-\Theta_{\alpha}}}\Sigma^+_{\alpha}+o(t^2)\\
				&=\mathfrak{m}^{1-\frac{\Theta_{\alpha}}{1-\Theta_{\alpha}}}\Sigma^+_{\alpha}\left\{\left(\frac{A_1}{A_0}+\frac{1-2\Theta_{\alpha}}{\Theta_{\alpha}}\frac{D_1(\alpha+1)}{D_0(\alpha+1)}\right)t+\left(\frac{A_2}{A_0}+\frac{1-2\Theta_{\alpha}}{\Theta_{\alpha}}\frac{D_2(\alpha+1)}{D_0(\alpha+1)}\right)t^2+o(t^2)\right\}\\		&=\mathfrak{m}^{1-\frac{\Theta_{\alpha}}{1-\Theta_{\alpha}}}\Sigma^+_{\alpha}  \left\{\left(\frac{A_1}{A_0}+\frac{2(1-\theta)}{\theta(1+\alpha)}\frac{D_1(\alpha+1)}{D_0(\alpha+1)}-\frac{1}{\theta\alpha}\frac{D_1(2\alpha)}{D_0(2\alpha)}\right)t\right.\\
				&\left.\quad\quad+\left(\frac{A_2}{A_0}+\frac{2(1-\theta)}{\theta(1+\alpha)}\frac{D_2(\alpha+1)}{D_0(\alpha+1)}-\frac{1}{\theta\alpha}\frac{D_2(2\alpha)}{D_0(2\alpha)}\right)t^2+o(t^2)\right\}\\		&=\mathfrak{m}^{1-\frac{\Theta_{\alpha}}{1-\Theta_{\alpha}}}\Sigma^+_{\alpha}  \left\{\left(\frac{A_1}{A_0}-\frac{2}{\gamma(1+\alpha)}\frac{D_1(\alpha+1)}{D_0(\alpha+1)}+\frac{1-\gamma}{\gamma\alpha}\frac{D_1(2\alpha)}{D_0(2\alpha)}\right)t\right.\\
				&\left.\quad\quad+\left(\frac{A_2}{A_0}-\frac{2}{\gamma(1+\alpha)}\frac{D_2(\alpha+1)}{D_0(\alpha+1)}+\frac{1-\gamma}{\gamma\alpha}\frac{D_2(2\alpha)}{D_0(2\alpha)}\right)t^2+o(t^2)\right\},
			\end{aligned}		
		\end{equation}
		where we have used that 
		$$ \text{$\mathcal{G}_{\alpha,n}=\frac{A_0^{\frac{\theta}{2}}D_0(\alpha+1)^{\frac{1-\theta}{\alpha+1}}}{D_0(2\alpha)^{\frac{1}{2\alpha}}}$,\  $\Sigma^+_{\alpha}= \left(\frac{\Theta_{\alpha}}{1-2\Theta_{\alpha}}\right)^{1-\frac{\Theta_{\alpha}}{1-\Theta_{\alpha}}} \mathcal{G}_{\alpha,n}^{\frac{2}{\theta}\cdot\frac{\Theta_{\alpha}}{(1-\Theta_{\alpha})}}$,\  $\frac{1-2\Theta_{\alpha}}{\Theta_{\alpha}}=\frac{2(1-\theta)}{\theta(1+\alpha)}$,\  $\frac{1}{\theta}+\frac{1}{\gamma}=1$}, $$ and by
		$\int_V (u_+\xi_+)^{\alpha+1}d\mu \equiv 1 $ and Lemma \ref{key_lemma}, 
		\begin{equation}\label{normal_2}
			\frac{D_1(2\alpha)}{D_0(2\alpha)}=0 \text{\quad and\quad } \frac{D_2(2\alpha)}{D_0(2\alpha)}=0.
		\end{equation} 
		Similar to \eqref{expan_c_i2}, when $1<\alpha\le \frac{n}{n-2}$ and $\alpha<\frac{n+4}{n}$, by \eqref{expansion_D1st} and \eqref{expansion_A1st}, we also have
		\begin{equation}
			\begin{aligned}
				\mathcal{L}^+_{\alpha}(V,g, u_+(x,t)\xi_+(x,t),\tau_+(t))	=\mathfrak{m}^{1-\frac{\Theta_{\alpha}}{1-\Theta_{\alpha}}}\Sigma^+_{\alpha}  \{\left(\frac{A_1}{A_0}-\frac{2}{\gamma(1+\alpha)}\frac{D_1(\alpha+1)}{D_0(\alpha+1)}+\frac{1-\gamma}{\gamma\alpha}\frac{D_1(2\alpha)}{D_0(2\alpha)}\right)t+o(t)\}.
			\end{aligned}		
		\end{equation}
		Observing from Lemma \ref{key_lemma}, for $\alpha>1$,
		we can write that
		\begin{equation}
			\begin{aligned}
				\frac{A_1}{A_0}+\frac{1-\gamma}{\gamma\alpha}\frac{D_1(2\alpha)}{D_0(2\alpha)}-\frac{2}{\gamma(\alpha+1)}\frac{D_1(\alpha+1)}{D_0(\alpha+1)}:=c_1^{+}\operatorname{Sc}(p)+l_1^{+}\beta_1+l_2^{+}\operatorname{tr}(\textbf{a})
			\end{aligned}		
		\end{equation}
		and
		\begin{equation}
			\begin{aligned}
				&\frac{A_2}{A_0}+\frac{1-\gamma}{\gamma\alpha}\frac{D_2(2\alpha)}{D_0(2\alpha)}-\frac{2}{\gamma(\alpha+1)}\frac{D_2(\alpha+1)}{D_0(\alpha+1)}:\\
				:=& c_2^{+}\operatorname{tr}(\textbf{a}^2)+c_3^{-}\operatorname{E}(v)+c_4^{+}\operatorname{E}(\textbf{a} \otimes \textbf{a})+c_5^{+}\operatorname{E}(\textbf{a} \otimes \operatorname{Rc})\\
				&+c_6^{+}\beta_1\operatorname{tr}(\textbf{a})+c_7^{+}\beta_1^2+c_8^{+}\beta_1 \operatorname{Sc}(p)+k_1^+\beta_2+k_2^{+}\operatorname{E}(b)+k_3^{+}\operatorname{tr}(d)
			\end{aligned}		
		\end{equation}
		Here, the constants $c_i^{+}$, $l_i^+$, $k_i^+$ depend only on $n$ and $\alpha$. 
		
		\textbf{Next, we demonstrate that
			\begin{align*}
				c_i^{+}(\alpha,n) &= c_i^{-}(\alpha,n) \quad \text{(ignoring the domain of variable $\alpha$)},
				l_i^+ = 0, 
				k_i^+ = 0.
		\end{align*}}
		
		Indeed, by \eqref{def_mathcsr_B}
		\begin{equation}\label{thm:beta_ii_2}
			\frac{\mathscr{B}(p+1,q)}{\mathscr{B}(p,q)}=\frac{\mathrm{~B}(p+1,-q-p)}{\mathrm{~B}(p,-p-q+1)}=-\frac{p}{p+q} \text{ when $p>0$ and $q<-p$},
		\end{equation}
		and
		\begin{equation}\label{thm:beta_eq_2}
			\frac{\mathscr{B}(p,q+2)}{\mathscr{B}(p,q)}=\frac{1}{\frac{\mathrm{~B}(p,-p-q+1)}{\mathrm{~B}(p,-q-p-1)}}=\frac{1}{\frac{-p-q}{-p-q+p}\cdot \frac{-p-q-1}{-p-q-1+p}}=\frac{q+1}{p+q+1}\cdot\frac{q}{p+q} \text{ when $p>0$ and $q<-p-1$,}  
		\end{equation}
		A comparison of \eqref{thm:beta_ii_1} with \eqref{thm:beta_ii_2} reveals that $\frac{\mathscr{B}(p+1,q)}{\mathscr{B}(p,q)}$ changes sign when comparing the cases $q<-p$ and $q>0$. This shows that
		$$
		\begin{aligned}
			c_1^+
			=&-\frac{1}{6n}\left(\frac{|1-\alpha|}{8}\right)^{-1}\left\{-	\frac{\mathscr{B}(\frac{n}{2}+2,\frac{2\alpha}{1-\alpha}+1)}{\mathscr{B}(\frac{n}{2}+1,\frac{2\alpha}{1-\alpha}+1)}
			+\frac{2}{\gamma(\alpha+1)}\frac{\mathscr{B}(\frac{n}{2}+1,\frac{1+\alpha}{1-\alpha}+1)}{\mathscr{B}(\frac{n}{2},\frac{1+\alpha}{1-\alpha}+1)} 	
			-\frac{1-\gamma}{\gamma\alpha}\frac{\mathscr{B}(\frac{n}{2}+1,\frac{2\alpha}{1-\alpha}+1)}{\mathscr{B}(\frac{n}{2},\frac{2\alpha}{1-\alpha}+1)}\right\}\\
			=&\zeta_1.
		\end{aligned}
		$$
		Here, $c_1^+(\alpha,n)=c_1^-(\alpha,n)$ since
		both $|1-\alpha|$ and $\frac{\mathscr{B}(p+1,q)}{\mathscr{B}(p,q)}$ change sign compared to \eqref{c_1_-}. 
		
		A comparison of \eqref{thm:beta_ii_1} with \eqref{thm:beta_ii_2} reveals that $\frac{\mathscr{B}(p+2,q)}{\mathscr{B}(p,q)}$ remains the same when comparing the cases $q<-p-1$ and $q>0$.
		Compared to \eqref{c_2_-}, we have $c_2^+(\alpha,n)=c_2^-(\alpha,n)$ (ignoring the domain of variable $\alpha$), that is,
		\begin{equation}\label{c_2_+}
			\begin{aligned}
				c_2^+=\frac{16}{n}	\frac{\mathscr{~B}(\frac{n}{2}+1,\frac{2\alpha}{1-\alpha}+3)}{\mathscr{~B}(\frac{n}{2}+1,\frac{2\alpha}{1-\alpha}+1)}=128(\alpha+1)\zeta_2.
			\end{aligned}
		\end{equation}
		
		Similarly, we can show that $c_i^{+}(\alpha,n) = c_i^{-}(\alpha,n)$,  $l_i^+=0$, $k_i^+=0$ based on \eqref{thm:beta_ii_1},  \eqref{thm:beta_ii_2}, \eqref{thm:beta_ii_1} and \eqref{thm:beta_ii_2}. 
		Notice that for the case $\alpha>1$,$\operatorname{tr}(\textbf{a})$, $\beta_1$ and $\operatorname{Sc}(p)$, we should satisfy the following relation by \eqref{normal_2}:
		\begin{equation}\label{normal_2_condition}
			\begin{aligned}
				&\frac{D_1(2\alpha)}{D_0(2\alpha)}=\frac{2\alpha}{2n}\left(\frac{|\alpha-1|}{8}\right)^{-1}\frac{\mathscr{B}(\frac{n}{2}+1,\frac{2\alpha}{1-\alpha}+1)}{\mathscr{B}(\frac{n}{2},\frac{2\alpha}{1-\alpha}+1)}\operatorname{tr}(\textbf{a})+\frac{2\alpha}{2}\beta_1\\
				&-\frac{1}{6n}
				\left(\frac{|\alpha-1|}{8}\right)^{-1}\frac{\mathscr{B}(\frac{n}{2}+1,\frac{2\alpha}{1-\alpha}+1)}{\mathscr{B}(\frac{n}{2},\frac{2\alpha}{1-\alpha}+1)}\operatorname{Sc}(p)=0.
			\end{aligned}
		\end{equation}  
		Hence,  we still have
		\begin{equation}\label{III_+}
			\begin{aligned}
				\operatorname{\uppercase\expandafter{\romannumeral3}}^+:= C_1^+\operatorname{Sc}^2(p)+C_2^+\operatorname{tr}(\textbf{a})\operatorname{Sc}(p)+C_3^+\operatorname{tr}(\textbf{a})^2.
			\end{aligned}
		\end{equation}
		Here $C_1^+$, $C_2^+$ and $C_3^+$ are the constants depending on
		$\alpha,n$. Here, $C_i^+$ may be different from $C_2^-$. However, we still have $
		\operatorname{\uppercase\expandafter{\romannumeral3}}^+=0 \text{ if } \operatorname{Sc}(p)=0 \text{ and } \operatorname{tr}(\textbf{a})=0$.
		
	\end{proof}

	Next, we prove Theorem~\ref{const_rigidity}. In fact, we establish a more general result (Theorem~\ref{mu_rigidity}), from which Theorem~\ref{const_rigidity} follows directly.
	\begin{thm}\label{mu_rigidity}
		Let $(M^n,g)$ be a Riemannian manifold of dimension $n\geq 3$, and let $V$ be an open bounded subset of $M^n$. 
		
		\noindent\textbf{(a) For the case $0<\alpha<1$:}
		\begin{enumerate}
			\item When $0<\alpha<1$,
			if there exists $\tau_0>0$ such that for all $\tau\le\tau_0$
			\begin{equation}\label{thm_comparison_iso_2_66}
				\mathbb{L}^-_{\alpha}(V,g, \tau)\ge -o(\tau),
			\end{equation}
			\textbf{then} the scalar curvature of $V$ satisfies
			$$
			\operatorname{Sc}(x)\le 0 \quad \text{for all } x \in V.
			$$	
			
			\item When $0<\alpha<1$,
			if there exists $\tau_0>0$ such that for all $\tau\le\tau_0$
			\begin{equation}\label{thm_comparison_iso_2_6}
				\mathbb{L}^-_{\alpha}(V,g, \tau)\ge -o(\tau^2),
			\end{equation}
			and if additionally
			\begin{equation}\label{thm_Rc_point_compare_68888}
				\int_{V} \lambda_1(\operatorname{Rc}) \, d\mu \geq 0, 
			\end{equation}
			\textbf{then} $V$ must be flat. Here, $\lambda_1(\operatorname{Rc})$ denotes the first eigenvalue of the Ricci tensor.
			
			\item When $\alpha$ satisfies the more restrictive range \eqref{6_i_3}, 
			if both \eqref{thm_comparison_iso_2_6}
			and
			\begin{equation}\label{thm_R_point_compare_6}
				\int_{V} \operatorname{Sc} \, d\mu \geq 0
			\end{equation}
			hold, \textbf{then} $V$ must be flat.
		\end{enumerate} 
		
		\noindent\textbf{(b) For the case $1<\alpha\le \frac{n}{n-2}$:}
		\begin{enumerate}
			\item When $\alpha$ satisfies the range \eqref{6_i_2},        
			if there exists  $\tau_0>0$ such that for all $\tau\le\tau_0$
			\begin{equation}\label{comparison_iso_2_66}
				\mathbb{L}^+_{\alpha}(V,g, \tau)\ge -o(\tau),
			\end{equation}
			\textbf{then} the scalar curvature satisfies 
			$$\text{$\operatorname{Sc}(x) \leq 0$ for all $x \in V$.}$$
			
			\item When $\alpha$ satisfies the more restrictive range \eqref{6_i_4666},
			if there exists $\tau_0>0$ such that for all $\tau\le\tau_0$
			\begin{equation}\label{comparison_iso_2_6}
				\mathbb{L}^+_{\alpha}(V,g, \tau)\ge -o(\tau^2),
			\end{equation}
			and if additionally
			\begin{equation}\label{thm_Rc_point_compare_6888866}
				\int_{V} \lambda_1(\operatorname{Rc}) \, d\mu \geq 0, 
			\end{equation}
			\textbf{then} $V$ must be flat.

			\item When $\alpha$ satisfies the more restrictive range \eqref{6_i_4},if both \eqref{comparison_iso_2_6} and
			\begin{equation}\label{integral_nonnegative}
				\int_{V} \operatorname{Sc} \, d\mu \geq 0
			\end{equation}
			hold, \textbf{then} $V$ must be flat.
		\end{enumerate}
	\end{thm}
	
	\begin{proof}
		\textbf{(a)}\textbf{ First, we prove \textbf{(a)}(1).}
		For any $p\in V$,  we obtain from \eqref{expansion_-} and \eqref{thm_comparison_iso_2_66} that for $t$ small enough:
		\begin{equation}\label{lemma32:111}
			\mathcal{L}^-_{\alpha}(V,g, u_-(x,t)\xi_-(x,t),\tau_-(t))=\mathfrak{m}^{1-\frac{\Gamma_{\alpha}}{2\Gamma_{\alpha}+1}}\Sigma^-_{\alpha} \left(\zeta_1 \operatorname{Sc}(p)t+o(t)\right)\ge -o(t),
		\end{equation}
		where $u_-(x,t)$ and $\xi_-(x,t)$ are these used in Theorem \ref{expansion_L_a}.
		Since $\zeta_1 < 0$ when $0 < \alpha < 1$, by comparing the $O(t)$ term in \eqref{lemma32:111} we have
		\begin{equation}\label{lemma4.1_i}
			\text{$\operatorname{Sc}(x) \leq 0$ for all $x \in V$,}
		\end{equation}
		since $p$ is arbitrarily chosen. \textbf{This proves \textbf{(a)}(1)}. \textbf{	This proves \textbf{(a)}(2)}.
		
		\textbf{ Second, we prove \textbf{(a)}(2).}   By \eqref{lemma4.1_i}, we know that $
		\text{$\lambda_1(\operatorname{Rc})(x) \leq 0$ for all $x \in V$.}$ This together with \eqref{thm_Rc_point_compare_6888866} imply $\lambda_1(\operatorname{Rc})(x) \equiv 0$ on $V$, so $\operatorname{Rc} \equiv 0$ on $V$ by \eqref{lemma4.1_i}. Now  we take
		$
		\textbf{a}=0 
		$
		in \eqref{expansion_-}. Then we obtain from \eqref{expansion_-} and \eqref{thm_comparison_iso_2_6} that for $t$ small enough:
		\begin{equation}\label{6666666666666666666666}
			\mathcal{L}^-_{\alpha}(V,g, u_-(x,t)\xi_-(x,t),\tau_-(t))=\mathfrak{m}^{1-\frac{\Gamma_{\alpha}}{2\Gamma_{\alpha}+1}}\Sigma^-_{\alpha} \left(-\frac{32\zeta_2}{6}|\operatorname{Rm}|^2(p) t^2+o(t^2)\right)\ge -o(t^2),
		\end{equation}
		where we have used, by \eqref{III_-} and $\textbf{a}=0$, 
		\begin{equation}\label{lemma32:55}
			\mathrm{III}^-=C_1'(\alpha,n)\operatorname{Sc}^2(p)=0.
		\end{equation}
		Since $\zeta_2>0$, we conclude that $\operatorname{Rm}\equiv 0$ on $V$ by \eqref{6666666666666666666666} since $p$ is arbitrarily chosen.
		
		\textbf{ Third, we prove \textbf{(a)}(3).}
		\eqref{thm_R_point_compare_6}  together with \eqref{lemma4.1_i} imply $\operatorname{Sc}\equiv 0$ on $V$. Now we take $\xi_-(x,t)\in \mathscr{B}_p(V)$, that is, we take
		\begin{equation}\label{choosing_a}
			\textbf{a}=\frac{2(\alpha+1)}{3\chi} \operatorname{Rc}(p) 
		\end{equation}
		in \eqref{expansion_-}.
		Then we obtain from \eqref{expansion_-} and \eqref{thm_comparison_iso_2_6} that  for $t$ small enough:
		\begin{equation}\label{lemma32:1}
			\mathcal{L}^-_{\alpha}(V,g, u_-(x,t)\xi_-(x,t),\tau_-(t))=\mathfrak{m}^{1-\frac{\Gamma_{\alpha}}{2\Gamma_{\alpha}+1}}\Sigma^-_{\alpha} \left(\mathrm{II} t^2+o(t^2)\right)\ge -o(t^2),
		\end{equation}
		with 
		\begin{equation}\label{lemma32:5}
			\begin{aligned}
				&\mathrm{II} =32\zeta_2\left\{ \frac{4((n+5) \alpha-n-3)(\alpha-1)}{9\chi}|\operatorname{Rc}|^2(p)-\frac{1}{6}|\operatorname{Rm}|^2(p)\right\},
			\end{aligned}
		\end{equation}
		where we have used, by \eqref{III_-} and \eqref{choosing_a}, 
		\begin{equation}\label{lemma32:55}
			\mathrm{III}^-=C_1''(\alpha,n)\operatorname{Sc}^2(p)=0.
		\end{equation}
		
		From the curvature decomposition, we get
		\begin{equation*}
			\operatorname{Rm}=W+\frac{1}{n-2} \left(\operatorname{Rc}-\frac{\operatorname{Sc}}{n}\right)\odot g+\frac{\operatorname{Sc}}{2n(n-1)} g\odot g,
		\end{equation*}
		and so
		\begin{equation}\label{orthogonal_decomposition}
			\begin{aligned}
				|\operatorname{Rm}|^2 &=\left|\operatorname { W }\right|^2+\left|\frac{1}{n-2} \stackrel{\circ}{\operatorname{Rc} } \odot g\right|^2+\left|\frac{\operatorname{Sc}}{2 n(n-1)} g \odot g\right|^2\\
				&= |W|^2+\frac{4}{n-2}|\operatorname{Rc}|^2-\frac{2}{(n-1)(n-2)} \operatorname{Sc}^2.
			\end{aligned}
		\end{equation}
		Since $\operatorname{Sc}(p)=0$, we obtain
		$$
		|\operatorname{Rm}|^2(p)\ge \frac{4}{n-2}|\operatorname{Rc}|^2(p).
		$$
		Consequently,
		\begin{equation*}
			\begin{aligned}
				&\mathrm{II} \le 32\zeta_2\left\{ \frac{4((n+5) \alpha-n-3)(\alpha-1)}{9\chi}-\frac{4}{6(n-2)}\right\}|\operatorname{Rc}|^2(p).
			\end{aligned}
		\end{equation*}
		Combining this with  \eqref{lemma32:1}, and also note that $\zeta_2>0$ when $0<\alpha<1$ and $\operatorname{Sc}(p)=0$, we obtain
		$$
		\left(\frac{4((n+5) \alpha-n-3)(\alpha-1)}{9\chi}-\frac{4}{6(n-2)}\right)|\operatorname{Rc}|^2(p)\ge 0.
		$$
		This implies $|\operatorname{Rc}|(p) = 0$ if
		\begin{equation}\label{lemma32:2}
			\frac{4((n+5)\alpha - n - 3)(\alpha - 1)}{9\chi} - \frac{4}{6(n - 2)} < 0,
		\end{equation}
		and in this case, we  see from \eqref{lemma32:1} and \eqref{lemma32:5} that $|\operatorname{Rm}|(p) = 0$.
		
		Since $\chi>0$, \eqref{lemma32:2} is equivalent to
		\begin{equation}\label{lemma32:3}
			\left(2 n^2+3 n-38\right) \alpha^2+\left(-4 n^2-2 n+50\right) \alpha+2 n^2-n-24<0.
		\end{equation}
		Inequality \eqref{lemma32:3} always holds for $n=3$. For $n\ge 4$, \eqref{lemma32:3} is equivalent to
		$$
		\alpha\in A:=	\left(\frac{2 n^2+n-25-\sqrt{28 n^2-16 n-287}}{2 n^2+3 n-38},\frac{2 n^2+n-25+\sqrt{28 n^2-16 n-287}}{2 n^2+3 n-38}\right).
		$$
		
		Also note that  $0<\alpha<1$.  We conclude that 
		when $\alpha$ satisfies 
		the range \eqref{6_i_3} (i.e. $\alpha\in (0,1)\cap A $), the conditions \eqref{thm_comparison_iso_2_6} and \eqref{thm_R_point_compare_6} imply that $V$ must be flat.
		\textbf{	This proves \textbf{(a)}(3).}
		
		\textbf{(b)} \textbf{ First, we prove \textbf{(b)}(1).}
		For any $p\in V$, 
		when $1<\alpha\le \frac{n}{n-2}$ and $\alpha<\frac{n+4}{n}$,
		it follows from \eqref{expansion_+_1} and \eqref{comparison_iso_2_66} that
		\begin{equation}\label{lemma32:111222}
			\mathcal{L}^+_{\alpha}(V,g, u_+(x,t)\xi_+(x,t),\tau_+(t))=\mathfrak{m}^{1-\frac{\Theta_{\alpha}}{1-\Theta_{\alpha}}}\Sigma^+_{\alpha}   \left(\zeta_1 \operatorname{Sc}(p)t+o(t)\right)\ge -o(t),
		\end{equation}
		where $u_+(x,t)$ and $\xi_+(x,t)$ are these used in Theorem \ref{expansion_L_a}.  
		Since $\zeta_1 < 0$ when $1<\alpha\le \frac{n}{n-2}$ and $\alpha<\frac{n+4}{n}$, by comparing the $O(t)$ term in \eqref{lemma32:111222}, we get 	\begin{equation}\label{lemma4.1_ii}
			\text{$\operatorname{Sc}(x) \leq 0$ for all $x \in V$,}
		\end{equation}
		since $p$ is arbitrarily chosen.  Notice that the inequalities
		$
		1<\alpha\le \frac{n}{n-2} \quad \text{and} \quad 1<\alpha<\frac{n+4}{n}
		$
		are equivalent to say $\alpha$ belonging to the range range \eqref{6_i_2}.
		\textbf{	This proves \textbf{(b)}(1)}.
		
		\textbf{ Second, the proof of \textbf{(b)}(2) is similar to  \textbf{(a)}(2).}	 Notice that the inequalities
		$
		1<\alpha\le \frac{n}{n-2} \quad \text{and} \quad \alpha<\frac{n+6}{n+2}
		$
		are equivalent to say $\alpha$ belonging to the range range \eqref{6_i_4666}, which guarantees expansion \eqref{expansion_+_2}.
		
		\textbf{ Third, we prove \textbf{(b)}(3).}	
		\eqref{integral_nonnegative}  together with \eqref{lemma4.1_ii} imply $\operatorname{Sc}\equiv 0$ on $V$.
		Now we take $\xi_+(x,t)\in \mathscr{B}_p(V)$, that is, we take
		\begin{equation}\label{choosing_aa}
			\textbf{a}=\frac{2(\alpha+1)}{3\chi} \operatorname{Rc}(p)  
		\end{equation}
		in \eqref{expansion_+_2}.
		Then we obtain from \eqref{expansion_+_2} and \eqref{comparison_iso_2_6} that when $1<\alpha\le \frac{n}{n-2}$ and $\alpha<\frac{n+6}{n+2}$,
		\begin{equation}
			\mathcal{L}^+_{\alpha}(V,g, u_+(x,t)\xi_+(x,t),\tau_+(t))=\mathfrak{m}^{1-\frac{\Theta_{\alpha}}{1-\Theta_{\alpha}}}\Sigma^+_{\alpha}  \left(\mathrm{II} t^2+o(t^2)\right)\ge -o(t^2),
		\end{equation}
		with $\zeta_2>0$ and
		\begin{equation}
			\begin{aligned}
				&\mathrm{II} =32\zeta_2\left\{ \frac{4((n+5) \alpha-n-3)(\alpha-1)}{9\chi}|\operatorname{Rc}|^2(p)-\frac{1}{6}|\operatorname{Rm}|^2(p)\right\},
			\end{aligned}
		\end{equation}
		where we have used, by \eqref{III_+} and \eqref{choosing_aa}, 
		\begin{equation}
			\mathrm{III}^+=C_2'(\alpha,n)\operatorname{Sc}^2(p)=0.
		\end{equation}
		
		When $1<\alpha\le \frac{n}{n-2}$ and $\alpha<\frac{n+6}{n+2}$, following similar arguments as in part (a) (3), 
		we can conclude that if \eqref{lemma32:3} holds, $|\operatorname{Rm}|(p)=0$. Note that
		the range \eqref{6_i_4} is equivalent to $\alpha$ satisfying \eqref{lemma32:3}, $1<\alpha\le \frac{n}{n-2}$ and $\alpha<\frac{n+6}{n+2}$. \textbf{This proves \textbf{(b)}(3).}
		
	\end{proof}
	
	Now we give the proof of Theorem \ref{const_rigidity}.
	\begin{proof}[Proof of Theorem \ref{const_rigidity}:]
		Theorem  \ref{const_rigidity} follows directly from Theorem \ref{mu_rigidity} and Theorem \ref{connection}.
	\end{proof}

	\section{ Power series expansion formulas of $\mathcal{W}^{\pm}_{\alpha}$-functionals and Proof of Theorem \ref{Yamabe_rigidity}}

	In this section, we derive the power series expansion formulas for the $\mathcal{W}^{\pm}_{\alpha}$-functionals and subsequently provide the proof of Theorem~\ref{Yamabe_rigidity}.
	\begin{lem}\label{expansion_W_a}
		We use the same notation as in Theorem \ref{expansion_L_a},
		and choose $\xi_{\pm} \in \mathscr{B}_p(V)$. 
		\\
		\noindent (i)	 
		For $0<\alpha<1$, we have 
		\begin{equation}\label{expansion_W-}
			\begin{aligned}
				&\quad\mathcal{W}^-_{\alpha}(V,g, u_-(x,t)\xi_-(x,t),\tau_-(t))\\
				&=\mathfrak{m}^{1-\frac{\Gamma_{\alpha}}{2\Gamma_{\alpha}+1}}\Sigma^-_{\alpha} \left\{32\zeta_2\left( \frac{4((n+5) \alpha-n-3)(\alpha-1)}{9\chi}|\operatorname{Rc}|^2(p)-\frac{1}{6}|\operatorname{Rm}|^2(p)\right)t^2+j^-(\alpha,n)\operatorname{Sc}^2(p)t^2+o(t^2)\right\},
			\end{aligned}
		\end{equation}
		where $j^-(\alpha,n)$ is a constant depending on $\alpha$ and $n$ that satisfies
		$$j^-(\alpha,n)\to 0 \text{ as } \alpha\to 1.$$

		\noindent (ii)	
		For $1<\alpha\le \frac{n}{n-2}$ and $\alpha<\frac{n+6}{n+2}$,
		we have 
		\begin{equation}\label{expansionW_+_2}
			\begin{aligned}
				&\quad\mathcal{W}^+_{\alpha}(V,g, u_+(x,t)\xi_+(x,t),\tau_+(t))\\
				&=\mathfrak{m}^{1-\frac{\Theta_{\alpha}}{1-\Theta_{\alpha}}}\Sigma^+_{\alpha} \left\{32\zeta_2\left( \frac{4((n+5) \alpha-n-3)(\alpha-1)}{9\chi}|\operatorname{Rc}|^2(p)-\frac{1}{6}|\operatorname{Rm}|^2(p)\right)t^2+j^+(\alpha,n)\operatorname{Sc}^2(p)t^2+o(t^2)\right\},
			\end{aligned}
		\end{equation}
		where $j^+(\alpha,n)$ is a constant depending on $\alpha$ and $n$ that satisfies
		$$j^+(\alpha,n)\to 0 \text{ as } \alpha\to 1.$$
		
	\end{lem}
	
	\begin{proof}
		(i)   By Lemma \ref{key_lemma} (iii), we get for $0<\alpha<1$:
		\begin{equation}\label{sc_expansion_-}
			\begin{aligned}
				&\quad 	 
				\frac{1}{2(1+\alpha)}\tau^{\Gamma_{\alpha}+1} \int_V \operatorname{Sc}(x) \cdot \left( u_-(x,t)\xi_-(x,t)\right)^2 d\mu
				\\
				&=\frac{1}{2(1+\alpha)}\left(\frac{\mathfrak{m}\Gamma_{\alpha}}{1+\Gamma_{\alpha}}\left(\frac{A_0}{D_0(\alpha+1)^{\frac{2}{\alpha+1}}}\right)^{-1}\frac{D_0(2\alpha)}{D_0(\alpha+1)^{\frac{2\alpha}{\alpha+1}}}\right)^{\frac{1+\Gamma_{\alpha}}{1+2\Gamma_{\alpha}}}\frac{ A_0}{D_0(\alpha+1)^{\frac{2}{\alpha+1}}}	\left(A_0^{-1} t^{1-\frac{n}{2}}   \int_V \operatorname{Sc}(x)\cdot  \left( \operatorname{H}(\frac{d(p,x)}{\sqrt{t}})\xi_-(x,t)\right)^2 d\mu \right)\\
				&=\mathfrak{m}^{1-\frac{\Gamma_{\alpha}}{2\Gamma_{\alpha}+1}}\Sigma^-_{\alpha} \frac{\frac{\omega_{n-1}}{2}\left(\frac{|\alpha-1|}{8}\right)^{-\frac{n}{2}}\mathscr{B}(\frac{n}{2},\frac{2}{1-\alpha}+1)}{2(1+\alpha)A_0}\times\\
				&\Big\{\operatorname{Sc}(p)t-\frac{8}{n(\alpha-1)+2\alpha-6}\left(\frac{1}{2}\Delta \operatorname{Sc}(p)-\frac{1}{6}\operatorname{Sc}^2(p)+\operatorname{tr}(\textbf{a})\operatorname{Sc}(p)\right)t^2+\beta_1 \operatorname{Sc}(p)t^2+o(t^2)\Big\}\\
				&=-\mathfrak{m}^{1-\frac{\Gamma_{\alpha}}{2\Gamma_{\alpha}+1}}\Sigma^-_{\alpha} \zeta_1\Big\{\operatorname{Sc}(p)t-\frac{8}{n(\alpha-1)+2\alpha-6}\left(\frac{1}{2}\Delta \operatorname{Sc}(p)-\frac{1}{6}\operatorname{Sc}^2(p)+\operatorname{tr}(\textbf{a})\operatorname{Sc}(p)\right)t^2+\beta_1 \operatorname{Sc}(p)t^2+o(t^2)\Big\}\\
				&=\mathfrak{m}^{1-\frac{\Gamma_{\alpha}}{2\Gamma_{\alpha}+1}}\Sigma^-_{\alpha} \Big\{-\zeta_1\operatorname{Sc}(p)t+32\zeta_2\Delta \operatorname{Sc}(p)t^2+\left(-\frac{32\zeta_2}{3}\operatorname{Sc}^2(p)+64\zeta_2\operatorname{tr}(\textbf{a})\operatorname{Sc}(p)-\zeta_1\beta_1  \operatorname{Sc}(p)\right)t^2+o(t^2)\Big\},
			\end{aligned}
		\end{equation}
		where we have used $\left(\frac{\Gamma_{\alpha}}{1+\Gamma_{\alpha}}\left(\frac{A_0}{D_0(\alpha+1)^{\frac{2}{\alpha+1}}}\right)^{-1}\frac{D_0(2\alpha)}{D_0(\alpha+1)^{\frac{2\alpha}{\alpha+1}}}\right)^{\frac{1+\Gamma_{\alpha}}{1+2\Gamma_{\alpha}}}	\frac{ A_0}{D_0(\alpha+1)^{\frac{2}{\alpha+1}}}=\Sigma^-_{\alpha} $ and, by \eqref{thm:beta_eq_1},
		$$
		\frac{\frac{\omega_{n-1}}{2}\left(\frac{|\alpha-1|}{8}\right)^{-\frac{n}{2}}\mathscr{B}(\frac{n}{2},\frac{2}{1-\alpha}+1)}{2(1+\alpha)A_0}=\frac{\frac{\omega_{n-1}}{2}\left(\frac{|\alpha-1|}{8}\right)^{-\frac{n}{2}}\mathrm{B}(\frac{n}{2},\frac{2\alpha}{1-\alpha}+3)}{2(1+\alpha)\frac{\omega_{n-1}}{32}\left(\frac{|\alpha-1|}{8}\right)^{-\frac{n+2}{2}}\mathrm{B}(\frac{n}{2}+1,\frac{2\alpha}{1-\alpha}+1)}=-\zeta_1.
		$$
		
		Since 
		\begin{equation}\label{choose_a}
			\textbf{a}=\frac{2(\alpha+1)}{3\chi} \operatorname{Rc}(p),
		\end{equation}
		we have that, by \eqref{expansion_-},
		\begin{equation}
			\begin{aligned}
				&\quad\mathcal{W}^-_{\alpha}(V,g, u_-(x,t)\xi_-(x,t),\tau_-(t))\\
				&=\mathcal{L}^-_{\alpha}(V,g, u_-(x,t)\xi_-(x,t),\tau_-(t))+\frac{1}{2(1+\alpha)}\tau^{\Gamma_{\alpha}+1} \int_V \operatorname{Sc}(x) \cdot \left( u_-(x,t)\xi_-(x,t)\right)^2 d\mu\\
				&=\mathfrak{m}^{1-\frac{\Gamma_{\alpha}}{2\Gamma_{\alpha}+1}}\Sigma^-_{\alpha} \left\{32\zeta_2\left( \frac{4((n+5) \alpha-n-3)(\alpha-1)}{9\chi}|\operatorname{Rc}|^2(p)-\frac{1}{6}|\operatorname{Rm}|^2(p)\right)t^2+\operatorname{\uppercase\expandafter{\romannumeral3}}^- t^2\right.\\
				&\left.\quad+\left(-\frac{32\zeta_2}{3}\operatorname{Sc}^2(p)+64\zeta_2\operatorname{tr}(\textbf{a})\operatorname{Sc}(p)-\zeta_1\beta_1  \operatorname{Sc}(p)\right)t^2+o(t^2)\right\}.
			\end{aligned}
		\end{equation}
		Here, by \eqref{III_-}, we have
		\begin{equation}
			\begin{aligned}
				\operatorname{\uppercase\expandafter{\romannumeral3}}^-=&\frac{c_3^-}{72}\operatorname{Sc}^2(p)+c_4^-(\operatorname{tr}(\textbf{a}))^2 +c_5^-\operatorname{tr}(\textbf{a})\operatorname{Sc}(p)+c_6^-\beta_1\operatorname{tr}(\textbf{a})+c_7^-\beta_1^2+c_8^-\beta_1 \operatorname{Sc}(p).
			\end{aligned}
		\end{equation}
		
		\textbf{We need to calculate $c_6^{-}$, $c_7^{-}$, $c_8^{-}$ in \eqref{III_-}.}
		By Lemma \ref{key_lemma}, \eqref{ot2_term}, \eqref{thm:beta_ii_1} and \eqref{thm:beta_eq_1}, we have
		$$
		\begin{aligned}
			c_6^-=& \frac{1}{n}	\left(\frac{|\alpha-1|}{8}\right)^{-1}\left(\frac{1-\gamma}{\gamma\alpha}\frac{2\alpha(\alpha-2)}{4}\frac{\mathrm{B}(\frac{n}{2}+1,\frac{2\alpha}{1-\alpha}+1)}{\mathrm{B}(\frac{n}{2},\frac{2\alpha}{1-\alpha}+1)}-\frac{2}{\gamma(1+\alpha)}\frac{(\alpha+1)(\alpha-1)}{4}\frac{\mathrm{B}(\frac{n}{2}+1,\frac{1+\alpha}{1-\alpha}+1)}{\mathrm{B}(\frac{n}{2},\frac{1+\alpha}{1-\alpha}+1)}\right)\\
			=&\frac{-4(\alpha-1)^2 n-8 \alpha^2+24 \alpha}{n(n(\alpha-1)-4)}=\frac{-(\alpha-1)^2 n-2 \alpha^2+6 \alpha}{2}\zeta_1,\\
			c_7^-=& \left(\frac{1-\gamma}{\gamma\alpha}\frac{2\alpha(2\alpha-2)}{8}-\frac{2}{\gamma(1+\alpha)}\frac{(\alpha+1)(\alpha-1)}{8}\right)\\
			=&\frac{(n-2) \alpha^2-2( n+1) \alpha+n}{4 n},
		\end{aligned}
		$$
		and
		$$
		\begin{aligned}
			c_8^-=& \frac{1}{6n}\left(\frac{|\alpha-1|}{8}\right)^{-1}\left\{-	\frac{\mathrm{~B}(\frac{n}{2}+2,\frac{2\alpha}{1-\alpha}+1)}{\mathrm{~B}(\frac{n}{2}+1,\frac{2\alpha}{1-\alpha}+1)}-\frac{1-\gamma}{\gamma}\frac{\mathrm{~B}(\frac{n}{2}+1,\frac{2\alpha}{1-\alpha}+1)}{\mathrm{~B}(\frac{n}{2},\frac{2\alpha}{1-\alpha}+1)}+\frac{1}{\gamma}   \frac{\mathrm{~B}(\frac{n}{2}+1,\frac{1+\alpha}{1-\alpha}+1)}{\mathrm{~B}(\frac{n}{2},\frac{1+\alpha}{1-\alpha}+1)}\right\}\\
			=&\frac{1}{6n}\left(\frac{|\alpha-1|}{8}\right)^{-1}\left(-\frac{\frac{n}{2}+1}{\frac{2\alpha}{1-\alpha}+\frac{n}{2}+2}-\frac{1-\gamma}{\gamma}\frac{\frac{n}{2}}{\frac{2\alpha}{1-\alpha}+\frac{n}{2}+1}+\frac{2}{\gamma}\frac{\frac{n}{2}}{\frac{1+\alpha}{1-\alpha}+\frac{n}{2}+1}\right)\\
			=&\frac{\alpha+1}{3}\zeta_1.
		\end{aligned}
		$$	
		The final simplification steps are computed using the mathematical software Maple.
		
		Hence,
		\begin{equation}
			\begin{aligned}
				&\operatorname{\uppercase\expandafter{\romannumeral3}}^--\frac{32\zeta_2}{3}\operatorname{Sc}^2(p)+64\zeta_2\operatorname{tr}(\textbf{a})\operatorname{Sc}(p)-\zeta_1\beta_1  \operatorname{Sc}(p)\\
				=&	-\frac{16}{9}\zeta_2\operatorname{Sc}^2(p)+\left(96\alpha^2-160\alpha+16n(\alpha-1)^2\right)\zeta_2(\operatorname{tr}(\textbf{a}))^2 +64\left(1-\frac{(\alpha+1)}{3 }\right)\zeta_2\operatorname{tr}(\textbf{a})\operatorname{Sc}(p)\\
				&+\frac{-(\alpha-1)^2 n-2 \alpha^2+6 \alpha}{2}\zeta_1\beta_1\operatorname{tr}(\textbf{a})+\frac{(n-2) \alpha^2-2( n+1) \alpha+n}{4 n}\beta_1^2+\frac{\alpha-2}{3}\zeta_1\beta_1 \operatorname{Sc}(p).
			\end{aligned}
		\end{equation}
		Since $\int_V (u_-\xi_-)^{\alpha+1}d\mu \equiv 1$, $\operatorname{tr}(\textbf{a})$, $\beta_1$ and $\operatorname{Sc}(p)$ should satisfy \eqref{normal_1_condition}.
		This together with \eqref{choose_a} implies that
		$$
		\begin{aligned}
			\beta_1=-\frac{8 \left((n+4) \alpha^2-2(n+5) \alpha+n+2\right)}{3 (\alpha+1)((n-2) \alpha-n-2)\chi} \operatorname{Sc}(p).
		\end{aligned}
		$$
		When $\alpha\to 1$, we have
		$\zeta_1\to -\frac{2}{n}$, $\zeta_2\to \frac{1}{16n}$, $\beta_1\to -\frac{1}{3}\operatorname{Sc}(p)$,  $\chi\to 4$ and $\textbf{a} \to \frac{1}{3}\operatorname{Rc}(p)$. It follows that  
		$$
		\operatorname{\uppercase\expandafter{\romannumeral3}}^--\frac{32\zeta_2}{3}\operatorname{Sc}^2(p)+64\zeta_2\operatorname{tr}(\textbf{a})\operatorname{Sc}(p)-\zeta_1\beta_1  \operatorname{Sc}(p)=j^-(\alpha,n)\operatorname{Sc}^2(p).
		$$
		with
		$$j^-(\alpha,n)\to 0.$$
		
		(ii) Similar to \eqref{sc_expansion_-}, by Lemma \ref{key_lemma} (iii), we have for $1<\alpha\le \frac{n}{n-2}$ and $\alpha<\frac{n+6}{n+2}$:
		\begin{equation}\label{sc_expansion_+}
			\begin{aligned}
				&\quad 	 
				\frac{1}{2(1+\alpha)}\tau^{1-2\Theta_{\alpha}} \int_V \operatorname{Sc}(x) \cdot \left( u_+(x,t)\xi_+(x,t)\right)^2 d\mu
				\\
				&=\frac{1}{2(1+\alpha)}\left(\frac{\mathfrak{m}\Theta_{\alpha}}{1-2\Theta_{\alpha}}\left(\frac{A_0}{D_0(2\alpha)^{\frac{1}{\alpha}}}\right)^{-1}\frac{D_0(\alpha+1)}{ D_0(2\alpha)^{\frac{\alpha+1}{2\alpha}}}\right)^{\frac{1-2\Theta_{\alpha}}{1-\Theta_{\alpha}}}	
				\frac{A_0}{ D_0(2\alpha)^{\frac{1}{\alpha}}}
				\left(A_0^{-1} t^{1-\frac{n}{2}}   \int_V \operatorname{Sc}(x)\cdot  \left( \operatorname{H}(\frac{d(p,x)}{\sqrt{t}})\xi_+(x,t)\right)^2 d\mu \right)\\
				&=\mathfrak{m}^{1-\frac{\Theta_{\alpha}}{1-\Theta_{\alpha}}}\Sigma^+_{\alpha} \frac{\frac{\omega_{n-1}}{2}\left(\frac{|\alpha-1|}{8}\right)^{-\frac{n}{2}}\mathscr{B}(\frac{n}{2},\frac{2}{1-\alpha}+1)}{2(1+\alpha)A_0}\times\\
				&\Big\{\operatorname{Sc}(p)t-\frac{8}{n(\alpha-1)+2\alpha-6}\left(\frac{1}{2}\Delta \operatorname{Sc}(p)-\frac{1}{6}\operatorname{Sc}^2(p)+\operatorname{tr}(\textbf{a})\operatorname{Sc}(p)\right)t^2+\beta_1 \operatorname{Sc}(p)t^2+o(t^2)\Big\}\\
				&=-\mathfrak{m}^{1-\frac{\Theta_{\alpha}}{1-\Theta_{\alpha}}}\Sigma^+_{\alpha}\zeta_1\Big\{\operatorname{Sc}(p)t-\frac{8}{n(\alpha-1)+2\alpha-6}\left(\frac{1}{2}\Delta \operatorname{Sc}(p)-\frac{1}{6}\operatorname{Sc}^2(p)+\operatorname{tr}(\textbf{a})\operatorname{Sc}(p)\right)t^2+\beta_1 \operatorname{Sc}(p)t^2+o(t^2)\Big\}\\
				&=\mathfrak{m}^{1-\frac{\Theta_{\alpha}}{1-\Theta_{\alpha}}}\Sigma^+_{\alpha} \Big\{-\zeta_1\operatorname{Sc}(p)t+32\zeta_2\Delta \operatorname{Sc}(p)t^2+\left(-\frac{32\zeta_2}{3}\operatorname{Sc}^2(p)+64\zeta_2\operatorname{tr}(\textbf{a})\operatorname{Sc}(p)-\zeta_1\beta_1  \operatorname{Sc}(p)\right)t^2+o(t^2)\Big\},
			\end{aligned}
		\end{equation}
		where we have used $\left(\frac{\Theta_{\alpha}}{1-2\Theta_{\alpha}}\left(\frac{A_0}{D_0(2\alpha)^{\frac{1}{\alpha}}}\right)^{-1}\frac{D_0(\alpha+1)}{ D_0(2\alpha)^{\frac{\alpha+1}{2\alpha}}}\right)^{\frac{1-2\Theta_{\alpha}}{1-\Theta_{\alpha}}}	 \frac{A_0}{ D_0(2\alpha)^{\frac{1}{\alpha}}}= \Sigma^+_{\alpha}$.
		Similar to \eqref{c_2_+}, based on \eqref{thm:beta_ii_1} and \eqref{thm:beta_ii_2}, we still have for $1<\alpha\le \frac{n}{n-2}$:
		$$
		\frac{\frac{\omega_{n-1}}{2}\left(\frac{|\alpha-1|}{8}\right)^{-\frac{n}{2}}\mathscr{B}(\frac{n}{2},\frac{2}{1-\alpha}+1)}{2(1+\alpha)A_0}=\frac{\frac{\omega_{n-1}}{2}\left(\frac{|\alpha-1|}{8}\right)^{-\frac{n}{2}}\mathscr{B}(\frac{n}{2},\frac{2\alpha}{1-\alpha}+3)}{2(1+\alpha)\frac{\omega_{n-1}}{32}\left(\frac{|\alpha-1|}{8}\right)^{-\frac{n+2}{2}}\mathscr{B}(\frac{n}{2}+1,\frac{2\alpha}{1-\alpha}+1)}=-\zeta_1.
		$$
		Similarly, we can show that $c_i^{+}(\alpha,n) = c_i^{-}(\alpha,n)$ for $i=6,7,8$ (ignoring the domain of variable $\alpha$) based on \eqref{thm:beta_ii_1},  \eqref{thm:beta_ii_2}, \eqref{thm:beta_ii_1} and \eqref{thm:beta_ii_2}. 
		
		Since $\int_V (u_+\xi_+)^{\alpha+1}d\mu \equiv 1$, $\operatorname{tr}(\textbf{a})$, $\beta_1$ and $\operatorname{Sc}(p)$ should satisfy \eqref{normal_2_condition}.
		This together with \eqref{choose_a} implies that
		$$
		\begin{aligned}
			\beta_1=-\frac{4 \left((n+4) \alpha^2-2(n+5) \alpha+n+2\right)}{3 \alpha((n-2) \alpha-n-2)\chi} \operatorname{Sc}(p).
		\end{aligned}
		$$
		We still have $\beta_1\to -\frac{1}{3}\operatorname{Sc}(p)$ as $\alpha\to 1$. Hence, we can conclude that Lemma \ref{expansion_W_a} (ii) holds as Lemma \ref{expansion_W_a} (i).
	\end{proof}	
	
	Next, we prove Theorem~\ref{Yamabe_rigidity}. In fact, we establish a more general result (Theorem~\ref{mu_constant_rigidity}), from which Theorem~\ref{Yamabe_rigidity} follows directly.
	
	\begin{thm}\label{mu_constant_rigidity}
		Let $(M^n,g)$ be a Riemannian manifold of dimension $n\geq 3$, and let $V$ be an open bounded subset of $M^n$.    
		Suppose that there exist constants $\kappa^{\pm}$, depending only on $n$, such that for $0<|\alpha-1|<\kappa^{\pm}$, the following holds: 
		there exists $\tau_0$ such that $\tau\leq \tau_0$ satisfies
		\begin{equation}\label{mu_constant_comparison_iso_2_6}
			\mu^{\pm}_{\alpha}(V,g, \tau) \geq -o(\tau^2).
		\end{equation}
		Then $V$ must be flat.
	\end{thm}
	\begin{proof}
		For any $p \in V$, we choose $\xi_{\pm} \in \mathscr{B}_p(V)$,i.e. we take
		$
		\textbf{a} = \frac{2(\alpha+1)}{3\chi} \operatorname{Rc}(p)
		$
		such that $\int_V (u_-\xi_-)^{\alpha+1} d\mu \equiv 1$ if $0 < \alpha < 1$,
		and $\int_V (u_+\xi_+)^{2\alpha} d\mu \equiv 1$ if $1 < \alpha \leq \frac{n}{n-2}$.
		
		By Lemma \ref{expansion_W-}, we have for $0<\alpha<1$ or ($1<\alpha\le \frac{n}{n-2}$ and $\alpha<\frac{n+6}{n+2}$):
		\begin{equation*}
			\begin{aligned}
				&\quad\mathcal{W}^{\pm}_{\alpha}(V,g, u_{\pm}(x,t)\xi_{\pm}(x,t),\tau_{\pm}(t))\\
				&=\operatorname{M}^{\pm}_{\alpha} \left\{32\zeta_2\left( \frac{4((n+5) \alpha-n-3)(\alpha-1)}{9\chi}|\operatorname{Rc}|^2(p)-\frac{1}{6}|\operatorname{Rm}|^2(p)\right)t^2+j^{\pm}(\alpha,n)\operatorname{Sc}^2(p)t^2+o(t^2)\right\},
			\end{aligned}
		\end{equation*}
		with
		$$j^{\pm}(\alpha,n)\to 0 \text{ as } \alpha\to 1.$$
		Here, $\operatorname{M}^{-}_{\alpha}=\mathfrak{m}^{1-\frac{\Gamma_{\alpha}}{2\Gamma_{\alpha}+1}}\Sigma^-_{\alpha}$ and  $\operatorname{M}^{+}_{\alpha}=\mathfrak{m}^{1-\frac{\Theta_{\alpha}}{1-\Theta_{\alpha}}}\Sigma^+_{\alpha}$.
		
		By the curvature decomposition \eqref{orthogonal_decomposition}, we have
		\begin{equation}\label{thm5.2_i}
			|\operatorname{Rm}|^2(p) \ge \frac{4}{n-2}|\operatorname{Rc}|^2(p) - \frac{2}{(n-1)(n-2)} \operatorname{Sc}^2(p).
		\end{equation}
		
		When $0<\alpha<1$ or ($1<\alpha\le \frac{n}{n-2}$  and $\alpha<\frac{n+6}{n+2}$), and if $\left( \frac{4((n+5) \alpha-n-3)(\alpha-1)}{9\chi} - \frac{2}{3(n-2)}\right) < 0$,
		by Lemma \ref{expansion_W_a} and \eqref{mu_constant_comparison_iso_2_6},
		we get for $t$ small enough:
		\begin{align}
			& -o(t^2)\le \mathcal{W}^{\pm}_{\alpha}(V,g, u_{\pm}(x,t)\xi_{\pm}(x,t),\tau_{\pm}(t)) \nonumber \\
			=& \operatorname{M}^{\pm}_{\alpha} \left\{32\zeta_2\left( \frac{4((n+5) \alpha-n-3)(\alpha-1)}{9\chi}|\operatorname{Rc}|^2(p)-\frac{1}{6}|\operatorname{Rm}|^2(p)\right)t^2+j^{\pm}(\alpha,n)\operatorname{Sc}^2(p)t^2+o(t^2)\right\} \label{goback2} \\
			\le & \operatorname{M}^{\pm}_{\alpha} \left\{ 32\zeta_2\left( \frac{4((n+5) \alpha-n-3)(\alpha-1)}{9\chi} - \frac{2}{3(n-2)}\right) |\operatorname{Rc}|^2(p) \right.\nonumber \\
			& \quad \left.+ \left(\frac{32\zeta_2}{3(n-1)(n-2)} + j^{\pm}(\alpha,n)\right) \operatorname{Sc}^2(p)\right\}t^2 +o(t^2)\label{goback} \\
			\le & 32\zeta_2 \operatorname{M}^{\pm}_{\alpha} \left\{\frac{1}{n}\left( \frac{4((n+5) \alpha-n-3)(\alpha-1)}{9\chi} - \frac{2}{3(n-2)}\right) \right. \nonumber \\
			& \quad \left. + \left(\frac{1}{3(n-1)(n-2)} + \frac{j^{\pm}(\alpha,n)}{32\zeta_2}\right)\right\} \operatorname{Sc}^2(p)t^2+o(t^2), \label{goback3}
		\end{align}
		with $\zeta_2>0$, where we have used \eqref{thm5.2_i} and $|\operatorname{Rc}|^2(p)\ge \frac{1}{n}\operatorname{Sc}^2(p)$.
		Hence, when 
		\begin{equation}\label{co_1}
			\text{ $0<\alpha<1$ or ($1<\alpha\le \frac{n}{n-2}$  and $\alpha<\frac{n+6}{n+2}$)},
		\end{equation}
		and
		\begin{equation}\label{co_2}
			\left( \frac{4((n+5) \alpha-n-3)(\alpha-1)}{9\chi} - \frac{2}{3(n-2)}\right) < 0,
		\end{equation}
		and
		\begin{equation}\label{co_3}
			\operatorname{F}(\alpha,n):=\left( \frac{4((n+5) \alpha-n-3)(\alpha-1)}{9\chi} - \frac{2}{3(n-2)}\right) + n\left(\frac{1}{3(n-1)(n-2)} + \frac{j^{\pm}(\alpha,n)}{32\zeta_2}\right) < 0,
		\end{equation}
		we conclude from \eqref{goback3} that $\operatorname{Sc}(p) = 0$. 
		Consequently, $\operatorname{Rc}(p) = 0$ by \eqref{goback}, and $\operatorname{Rm}(p) = 0$ by \eqref{goback2}.  Since $j^{\pm}(\alpha,n)\to 0$, $\zeta_2\to \frac{1}{16n}$  and $\chi\to 4$ as $\alpha \to 1$,  there exist $\kappa^{\pm}$ such that when $0<|\alpha-1|<\kappa^{\pm}$, there hold \eqref{co_1}, \eqref{co_2} and \eqref{co_3}.
	\end{proof}
	
	\begin{rem}\label{exact_number_equation}
		From the proof of Theorem~\ref{mu_constant_rigidity}, we observe that $\kappa^{\pm}$ depends on the intersections of the ranges given in \eqref{co_1}, \eqref{co_2}, and \eqref{co_3}. 
		Through direct calculation using Maple, we see that $\operatorname{F}(\alpha,n)=0$ is a complicated seventh-degree polynomial equation.
	\end{rem}

	Now we give the proof of Theorem \ref{Yamabe_rigidity}.
	\begin{proof}[Proof of Theorem \ref{Yamabe_rigidity}:]
		Theorem  \ref{Yamabe_rigidity} follows directly from Theorem \ref{mu_constant_rigidity} and Theorem \ref{connection}.
	\end{proof}

	\section{the proof of Theorem \ref{rigidity_iso_profile}}
	
	Before presenting the proof of Theorem \ref{mu_rigidity}, we need the following lemma.
	\begin{lem}\label{section}
		Let $(M^n,g)$ be an $n$-dimensional manifold and $p\in \mathring{M^n}$. 	Let $u_{\pm}$ be the functions defined in \eqref{def_u-} and \eqref{def_u+}.
		
		\noindent\textbf{(a) For the case $0<\alpha<1$:}
		\begin{enumerate}
			\item 
			When 
			\begin{equation}\label{6_i_1}
				0<\alpha<1,
			\end{equation}
			if there exist a neighborhood $V_p$ of $p$,  $\xi_-(x,t)\in \mathcal{B}_p(V_p)$ and $\bar{\xi}_-(x,t)\in \mathcal{B}_{p_K}(M^n_K)$ for some point $p_K\in M^n_K$ satisfying
			\begin{equation}\label{comparison_l1}
				\mathcal{L}^-_{\alpha}(V_p,g,u_-\xi_-,t )\ge  \mathcal{L}^-_{\alpha}(M^n_K,g_K, \bar{u}_-\bar{\xi}_-,t)-o(t),
			\end{equation}
			for all $0<t<\tau_0$ and some $\tau_0>0$,
			then the scalar curvature at $p$ satisfies
			$$
			\operatorname{Sc}(p)\le n(n-1)K.
			$$	
			Here, $ M^n_K$ 
			is the space form with constant sectional curvature $K$.
			\item   When $\alpha$ satisfies the more restrictive range \eqref{6_i_3},
			if we assume  that 
			\begin{equation}\label{comparison_ll}
				\mathcal{L}^-_{\alpha}(V_p,g,u_-\xi_-,t )\ge  \mathcal{L}^-_{\alpha}(M^n_K,g_K, \bar{u}_-\bar{\xi}_-,t)-o(t^2),
			\end{equation}
			for all $0<t<\tau_0$ and some $\tau_0>0$,
			and if additionally
			\begin{equation}\label{R_point_compare_6}
				\operatorname{Sc}(p)    \ge n(n-1)K, \quad	\Delta \operatorname{Sc}(p) \ge 0, 
			\end{equation}
			then the sectional curvature at $p$ satisfies
			$$\operatorname{Sec}(p)=K.$$
		\end{enumerate}

		\noindent\textbf{(b) For the case $1<\alpha\le \frac{n}{n-2}$:}
		\begin{enumerate}
			\item 	
			When $\alpha$ satisfies the range \eqref{6_i_2},  
			if there exist a neighborhood $V_p$ of $p$,  $\xi_+(x,t)\in \mathcal{B}_p^(V_p)$ and $\bar{\xi}_+(x,t)\in \mathcal{B}_{p_K}(M^n_K)$ for some point $p_K\in M^n_K$ satisfying
			\begin{equation}\label{comparison_21}
				\mathcal{L}^+_{\alpha}(V_p,g,u_+\xi_+,t )\ge  \mathcal{L}^+_{\alpha}(M^n_K,g_K, \bar{u}_+\bar{\xi}_+,t)-o(t),
			\end{equation}
			for all $0<t<\tau_0$ and some $\tau_0>0$,
			then the scalar curvature at $p$ satisfies
			$$
			\operatorname{Sc}(p)\le n(n-1)K.
			$$	
			Here, $ M^n_K$ 
			is the space form with constant sectional curvature $K$.
			\item  When $\alpha$ satisfies the more restrictive range \eqref{6_i_4},
			if we assume  that \begin{equation}\label{comparison_2}
				\mathcal{L}^+_{\alpha}(V_p,g,u_+\xi_+,t )\ge  \mathcal{L}^+_{\alpha}(M^n_K,g_K, \bar{u}_+\bar{\xi}_+,t)-o(t^2),
			\end{equation}
			for all $0<t<\tau_0$ and some $\tau_0>0$, and if additionally
			\begin{equation}\label{R_point_compare_66}
				\operatorname{Sc}(p)    \ge  n(n-1)K, \quad	\Delta \operatorname{Sc}(p) \ge 0, 
			\end{equation}
			then the sectional curvature at $p$ satisfies
			$$\operatorname{Sec}(p)= n(n-1)K.$$
		\end{enumerate}
	\end{lem}

	\begin{proof} 
		\textbf{  First, we prove (a)(1) and (b)(1).}
		By (\ref{expansion_-}) \eqref{expansion_+_1} and  (\ref{comparison_l1}) \eqref{comparison_21}, when $\alpha$ satisfies \eqref{6_i_1} or \eqref{6_i_2}, we conclude that 
		\begin{equation}\label{q_1} 
			\begin{aligned}
				\zeta_1 \operatorname{Sc}(p)t 
				\ge \zeta_1 \operatorname{Sc}_K(p)t-o(t)
			\end{aligned}
		\end{equation}
		for all $t\le \tau_0$, 
		where $\operatorname{Sc}_K$ denotes the scalar curvature of the $n$-dimensional space form of constant sectional curvature $K$. Since $\zeta_1<0$ when $\alpha$ satisfies \eqref{6_i_1} or \eqref{6_i_2}, we have $\operatorname{Sc}(p)\le \operatorname{Sc}_K(p_K)=n(n-1)K$.
		
		\textbf{ Second, we prove (a)(2) and (b)(2).}
		By (\ref{expansion_-}) \eqref{expansion_+_2} and  (\ref{comparison_ll}) \eqref{comparison_2}, when $\alpha$ satisfies \eqref{6_i_3} or \eqref{6_i_4}, we conclude that 
		\begin{equation}\label{q_2}
			\begin{aligned}
				&\zeta_1 \operatorname{Sc}(p)t- 32\zeta_2\Delta \operatorname{Sc}(p)t^2+32\zeta_2\left\{ \frac{4((n+5) \alpha-n-3)(\alpha-1)}{9\chi}|\operatorname{Rc}|^2(p)-\frac{1}{6}|\operatorname{Rm}|^2(p)\right\} t^2+C_1'\operatorname{Sc}^2(p) \\
				\ge& \zeta_1 \operatorname{Sc}_K(p)t+32\zeta_2\left\{ \frac{4((n+5) \alpha-n-3)(\alpha-1)}{9\chi}|\operatorname{Rc_K}|^2(p)-\frac{1}{6}|\operatorname{Rm}_K|^2(p)\right\} t^2+C_1'\operatorname{Sc}_K^2(p)-o(t^2)
			\end{aligned}
		\end{equation}
		where  $\operatorname{Rc_K}$ and $\operatorname{Rm}_K$ denote the Ricci curvature and curvature tensor of the $n$-dimensional space form of constant sectional curvature $K$. 
		
		If we have $\operatorname{Sc}(p)\ge n(n-1)K$, then $\operatorname{Sc}(p)=n(n-1)K$. Since $\Delta\operatorname{Sc}(p)\ge 0$ and $\zeta_2>0$, 
		we conclude from (\ref{q_2}) that 
		\begin{equation}\label{6_ii_1}
			\begin{aligned}
				&\frac{4((n+5) \alpha-n-3)(\alpha-1)}{9\chi}|\operatorname{Rc}|^2(p)-\frac{1}{6}|\operatorname{Rm}|^2(p)\\
				\ge& \frac{4((n+5) \alpha-n-3)(\alpha-1)}{9\chi}|\operatorname{Rc_K}|^2(p)-\frac{1}{6}|\operatorname{Rm}_K|^2(p). 
			\end{aligned}
		\end{equation}
		
		From the curvature orthogonal decomposition \eqref{orthogonal_decomposition}, 
		we have
		\begin{equation*}
			\begin{aligned}
				&\frac{4((n+5) \alpha-n-3)(\alpha-1)}{9\chi}|\operatorname{Rc}|^2(p)-\frac{1}{6}\left(\frac{4}{n-2}|\operatorname{Rc}|^2(p)-\frac{2}{(n-1)(n-2)} \operatorname{Sc}^2(p)\right)\\
				\ge& \frac{4((n+5) \alpha-n-3)(\alpha-1)}{9\chi}|\operatorname{Rc_K}|^2(p)-\frac{1}{6}\left(\frac{4}{n-2}|\operatorname{Rc_K}|^2(p)-\frac{2}{(n-1)(n-2)} \operatorname{Sc_K}^2(p)\right)
			\end{aligned}
		\end{equation*}
		If\begin{equation}
			\left( \frac{4((n+5) \alpha-n-3)(\alpha-1)}{9\chi} - \frac{2}{3(n-2)}\right) <0 ,
		\end{equation}
		we conclude that
		$$
		|\operatorname{Rc}|^2(p)\le |\operatorname{Rc}_K|^2(p)=\frac{1}{n}\operatorname{Sc_K}^2(p)=\frac{1}{n}\operatorname{Sc}^2(p),
		$$
		and hence
		$$
		|\operatorname{Rc}|^2(p)=|\operatorname{Rc}_K|^2(p)=\frac{1}{n}\operatorname{Sc}^2(p).
		$$
		and by \eqref{6_ii_1}
		\begin{equation}\label{rm_iii}		|\operatorname{Rm}|^2(p)\le|\operatorname{Rm}_K|^2(p)=2n(n-1)K^2.
		\end{equation}
		Again curvature orthogonal decomposition \eqref{orthogonal_decomposition}, 
		we have
		\begin{equation}\label{q_3}
			|\mathrm{\operatorname{Rm}}|^2(p)\ge \left|\frac{\operatorname{Sc}}{2 n(n-1)} g \odot g\right|^2(p)=2n(n-1)K^2
		\end{equation}
		with equality if and only if $g$ has constant sectional curvature. Then we get $\operatorname{Sec}(p)=K$ by \eqref{rm_iii} and \eqref{q_3}.

	\end{proof}

	Now we give the proof of Theorem \ref{rigidity_iso_profile}.
	
	\begin{proof}[Proof of Theorem \ref{rigidity_iso_profile}]
		For any $p\in V$,
		we take 
		$\eta\in \mathscr{B}_p(V)$, $u=u_-\eta$ when $0<\alpha<1$ or $u=u_+\eta$ when $1<\alpha\le \frac{n}{n-2}$, where $u_{\pm}$ are the functions defined in \eqref{def_u-} and \eqref{def_u+}.
		Next, we apply the spherical symmetrization (Schwarz symmetrization) method.
		
		We can choose $r_0$ to be small enough so that $\operatorname{supp}\{\eta\}\subset B(p,r_0)\subset\subset V$ and therefore there exists $B^K(p_K,r_t)\subset M^n_K$ such that
		$
		\operatorname{Vol}_{g}(\{x \in M^n\mid u(x,t) >0\})=\operatorname{Vol}\left(B^K(p_K,r_t)\right)\le \beta_0.
		$
		Let
		$
		\bar{u}(\cdot,t)
		$
		be a non-negative rotational symmetric function such that
		\begin{equation}\label{vol_equa}
			\operatorname{Vol}\left(\left\{y \in  M^n_K\mid \bar{u}(y,t) \geq s\right\}\right)= \operatorname{Vol}\left(\{x \in V\mid u(x,t) \geq s\}\right)
		\end{equation}
		for all $s>0$ and $\bar{u}(y,t)=0$ when $\bar{d}(p_K,y)\geq r_t$. It is clear that  $\bar{u}(r,t) :=\bar{u}(y,t)$ is non-increasing in $r=\bar{d}(p_K,y)$ for any $t>0$. We define
		$
		\mathcal{M}_s :=\{x \in V\mid u(x,t) \geq s\}, \mathcal{M}_s^{\prime} :=\left\{y \in  M^n_K\mid \bar{u}(y,t) \geq s\right\}
		$
		and $\Gamma_s := \partial \mathcal{M}_s$, $\Gamma_s^{\prime} :=\partial \mathcal{M}_s^{\prime}$. By the co-area formula and (\ref{vol_equa}), we have
		\begin{equation}\label{level_set_equa}
			\int_{\Gamma_s} \frac{1}{|\nabla u(\cdot,t)|} d \sigma=\int_{\Gamma_s^{\prime}} \frac{1}{|\bar{\nabla} \bar{u}(\cdot,t)|} d \sigma_K,
		\end{equation}
		\begin{equation}\label{1.1_1}
			\int_{V} u(\cdot,t)^q d\mu=	\int_{M^n_K} \bar{u}(\cdot,t)^q d\mu_K,
		\end{equation}
		for any $q>0$.    
		
		Since $\mathcal{M}_s^{\prime}$ is a round ball in space form and by (\ref{comparison_iso}), we have
		\begin{equation}\label{key}
			\begin{aligned}
				\operatorname{Area}\left(\Gamma_s^{\prime}\right)  =\operatorname{I}(M^n_K,\operatorname{Vol}\left(\mathcal{M}_s^{\prime}\right)) \le \operatorname{I}(V,\operatorname{Vol}\left(\mathcal{M}_s\right)) \leq \operatorname{Area}\left(\Gamma_s\right).
			\end{aligned}
		\end{equation}
		and hence
		\begin{equation*}
			\begin{aligned}
				& \int_{\Gamma_s^{\prime}}|\bar{\nabla} \bar{u}(\cdot,t)| d \sigma_K \cdot \int_{\Gamma_s^{\prime}} \frac{1}{|\bar{\nabla} \bar{u}(\cdot,t)|} d \sigma_K \\
				=& \left(\text { Area }\left(\Gamma_s^{\prime}\right)\right)^2\leq\left(\text { Area }_{\tilde{g}}\left(\Gamma_s\right)\right)^2 \\
				\leq&  \int_{\Gamma_s}|\nabla u(\cdot,t)| d \sigma\cdot \int_{\Gamma_s} \frac{1}{|\nabla u(\cdot,t)|} d \sigma,
			\end{aligned}
		\end{equation*}
		where we used the H\"{o}lder inequality to obtain the last inequality. By this and (\ref{level_set_equa}), we have
		\begin{equation}\label{1.1_3}
			\int_{M^n_K}|\bar{\nabla} \bar{u}(\cdot,t)|^2  d \mu_K
			\leq \int_{V}|\nabla u(\cdot,t)|^2 d \mu.
		\end{equation}
		It follows that (\ref{1.1_1}) and (\ref{1.1_3}),
		we have
		\begin{equation}\label{1.1_8}
			\mathcal{L}^{\pm}_{\alpha}(V,g, u,t )\ge  \mathcal{L}^{\pm}_{\alpha}(M^n_K,g_K, \bar{u},t).
		\end{equation}
		
		Now we let $\alpha$  satisfy \eqref{6_i_3} or \eqref{6_i_4}.
		For the case $K= 0$, we have $\mathcal{L}^{\pm}_{\alpha}(V,g, u,t )\ge 0$ by Theorem \ref{connection} and the Gagliardo-Nirenberg
		inequality in Euclidean space. In this case, Theorem \ref{rigidity_iso_profile} follows from Theorem  \ref{mu_rigidity}. 
		
		Next, we consider the case $K\ne 0$.
		Taking $s=\bar{u}(r,t)$ in (\ref{level_set_equa}),  $\bar{u}$ is the solution to
		\begin{equation}\label{key_7}
			\int_{\Gamma_r} \frac{1}{|\nabla u(\cdot,t)|} d \sigma= \frac{\operatorname{Area_K}(\partial B^K(p_K,r))}{|\frac{d}{dr} \bar{u}(r,t)|},
		\end{equation}
		with  $\Gamma_r=\{x\in M \mid u(x,t)=\bar{u}(r,t) \}$. 
		Now we rescale the metrics as $\tilde{g}=t^{-1}g$ and  $\tilde{g}_K=t^{-1}g_K$. 
		Then (\ref{key_7})  becomes
		\begin{equation}\label{key_6}
			\int_{\tilde{\Gamma_r}} \frac{1}{|\tilde{\nabla}\tilde{u}(x,t)|} d \sigma_{\tilde{g}}= \frac{\operatorname{Area_{tK}}(\partial B^{tK}(p_K,r))}{|\frac{d}{dr} \tilde{u}_K(r,t)|},
		\end{equation}
		where $\tilde{u}(x,t)=t^{\frac{n}{2(\alpha+1)}}D_0(\alpha+1)^{-\frac{1}{\alpha+1}}u(\sqrt{t}x,t)$ and $\tilde{u}_K(r,t)=t^{\frac{n}{4\alpha}}D_0(2\alpha)^{-\frac{1}{2\alpha}}\bar{u}_K(\sqrt{t}r,t)$ if $0<\alpha<1$,  $\tilde{u}(x,t)=t^{\frac{n}{4\alpha}}u(\sqrt{t}x,t)$ and $\tilde{u}_K(r,t)=t^{\frac{n}{4\alpha}}\bar{u}_K(\sqrt{t}r,t)$ if $\alpha>1$, $\tilde{\Gamma_r}=\{x\in M \mid \tilde{u}(x,t)=\tilde{u}_K(r,t)\}$. Moreover, we can write $\tilde{u}(x,t)= \operatorname{H}(x)\tilde{\eta}(x,t)$ and $\tilde{\eta}^2$ can be written as $\tilde{\eta}^2=1+\frac{2(\alpha+1)}{3\chi} \operatorname{Rc}(p)_{ij}t\tilde{y^i}\tilde{y^j}
		+e_{ijk}t^{\frac{3}{2}}\tilde{y^i}\tilde{y^j}\tilde{y^k}+b_{ijkl}t^2\tilde{y^i}\tilde{y^j}\tilde{y^k}\tilde{y^l}+o(t^2d_{\tilde{g}}^4)+\beta_1 t+q_it^{\frac{3}{2}}\tilde{y^i}+d_{ij}t^2\tilde{y^i}\tilde{y^j}+o(td_{\tilde{g}}^2)t+\beta_2 t^2+o(t^2)
		$, here  $\{\tilde{x}^k\}^n_{k=1}$ be the normal geodesic coordinates centered at $p$ on $T_pM$ with respect to metric $\tilde{g}$.  Taking $t\to 0$ in (\ref{key_6}), we can get
		$\tilde{u}(r,0)=\operatorname{H}(r)$.
		It is straightforward from (\ref{key_6}) and the differentiability of $\tilde{\Gamma_r}$ and $\tilde{\eta}^2$ that $ \bar{u}(r,t)=u_{\pm}(r)\bar{\eta}(r,t)$ with rotational symmetric function $\bar{\eta}(x,t)$ can be written as $\bar{\eta}(x,t)^2=\sum\limits_{k=0}^2\bar{\phi}_k(x)t^k+o(t^2)$ around $(p_K,0)$ with $\bar{\phi}_2$ being continuous at $p_K$, both 4th derivatives of $\bar{\phi}_0$, and 2nd derivatives of $\bar{\phi}_1$ exist at $p_K$. 
		
		Since \text{  $\int_V (u_-\eta)^{\alpha+1}d\mu \equiv 1$ if $0<\alpha<1$ and $\int_V (u_+\eta)^{2\alpha}d\mu \equiv 1$ if  $1<\alpha\le \frac{n}{n-2}$}, 	
		by letting  $t\to 0$ in (\ref{1.1_1}), we get $\bar{\eta}^2(p_K,0)=1$. Also note that $\bar{u}(r,t) $ is non-increasing in $r$ for any $t>0$. Then $\bar{u}(y,t)$ achieves its maximum at $p_K$ for any $t$ and therefore $\bar{\nabla}\bar{u}(p_K,t)=0$.
		By (\ref{expansion_-}), (\ref{expansion_+_2}) and comparing the $O(t)$ terms of (\ref{1.1_8}), we get $\operatorname{Sc}(x)\le n(n-1)K$ for all $x\in V$ since $p$ is arbitrarily chosen. We get $\operatorname{Sc}(x)\equiv n(n-1)K$ on $V$ by \eqref{scalar_curvature_lowerbound}. Therefore,   $\Delta\operatorname{Sc}\equiv 0$
		on $V$.
		By the $O(t)$ terms of (\ref{1.1_1}) and (\ref{expansion_D1st}), we have
		$\frac{D_1(q)}{D_0(q)}=\frac{\bar{D}_1(q)}{\bar{D}_0(q)}$ for any $q$. Hence, we can get $\frac{\partial}{\partial t}\bar{\eta}^2(p_K,0)=\frac{\partial}{\partial t}\eta^2(p,0)=\beta_1$ and $\operatorname{tr}(\bar{\nabla}\bar{\nabla} \bar{\eta}^2)(p_K,0)=\operatorname{tr}(\nabla\nabla \eta^2)(p,0)=\frac{4(\alpha+1)}{3\chi}\operatorname{Sc}(p)=\frac{4(\alpha+1)}{3\chi}n(n-1)K$.  Hence, $\bar{\nabla}\bar{\nabla} \bar{\eta}^2(p_K,0)=\frac{4(\alpha+1)}{3\chi}(n-1)K\delta_{ij}$ since $\bar{\eta}$ is rotational symmetric. Then we get $\bar{u}\in \mathscr{B}_p(M^n_K)$.
		So, Theorem \ref{rigidity_iso_profile}  follows from Lemma \ref{section}.
		
	\end{proof}

\end{document}